\documentclass[a4paper]{amsart}

\usepackage{amsmath,amssymb,amsthm,stmaryrd,mathrsfs}

\usepackage{color,graphicx}
\usepackage[dvipsnames]{xcolor}
\usepackage{accents}
\usepackage{enumerate}
\usepackage{varwidth}
\usepackage{enumitem}

\usepackage{fancyhdr, array,calc,graphicx,url,tabularx}
\usepackage{geometry}
\usepackage{latexsym}
\usepackage{amssymb}
\usepackage{amsmath}
\usepackage{enumerate}
\usepackage{verbatim}
\usepackage{tikz}
\usepackage[colorlinks=true,linkcolor=blue]{hyperref}

\usepackage{changepage}

\usepackage[utf8]{inputenc}
\usepackage{tikz}
\tikzset{
     block/.style={rectangle, draw, fill=red!40, text width=6em,
                   text centered, rounded corners, minimum height=3em},
     arrow/.style={-{Stealth[]}}
     }

\def\undertilde#1{\math{\sf{Ord}}{\vtop{\ialign{##\crcr
$\hfil\displaystyle{#1}\hfil$\crcr\noalign{\kern1.5pt\nointerlineskip}
$\hfil\tilde{}\hfil$\crcr\noalign{\kern1.5pt}}}}}

\makeindex

{
   \end{minipage}
   \vspace*{\stretch{3}}
   \clearpage
}

\def\undertilde#1{{\baselineskip=0pt\vtop
  {\hbox{$#1$}\hbox{$\scriptscriptstyle\sim$}}}{}}

\newcommand{\ADR}{\mathsf{AD}_{\mathbb{R}}}

\newcommand{\treg}{\mathsf{\Theta_{reg}}}

\newcommand{\bls}{\vspace{\baselineskip}}
\newcommand{\tc}{\mathrm{tc}}

\newcommand{\less}{\mathord{<}}

\renewcommand{\gg}{\gamma}

\newcommand{\bR}{{\mathbb{R}}}

\newcommand{\rest}{\restriction}

\newcommand{\Add}{\mathrm{Add}}

\newcommand{\rmc}{\mathrm{c}}

\newcommand{\hp}{{\mathfrak{p}}}
\newcommand{\hq}{{\mathfrak{q}}}
\newcommand{\hr}{{\mathfrak{r}}}
\newcommand{\hs}{{\mathfrak{s}}}

\newcommand{\hw}{{\mathfrak{w}}}

\newcommand{\hh}{{\mathfrak{h}}}
\newcommand{\hm}{{\mathfrak{m}}}
\newcommand{\hk}{{\mathfrak{k}}}
\newcommand{\hx}{{\mathfrak{x}}}
\newcommand{\hy}{{\mathfrak{y}}}

\newcommand{\card}[1]{{\vert #1 \vert} }

\renewcommand{\models}{\vDash}
\newcommand{\powerset}{\mathscr{P}}
\newcommand{\dom}{{\rm dom}}
\newcommand{\rge}{{\rm rge}}
\newcommand{\im}{{\rm im}}
\newcommand{\cp}{{\rm crit }}

\newcommand{\cf}{{\rm cf}}
\newcommand{\lh}{{\rm lh}}

\newtheorem{theorem}{Theorem}[section]
\newtheorem{proposition}[theorem]{Proposition}
\newtheorem{definition}[theorem]{Definition}

\newtheorem{lemma}[theorem]{Lemma}
\newtheorem{corollary}[theorem]{Corollary}
\newtheorem{claim}[theorem]{Claim}

\newtheorem{conjecture}[theorem]{Conjecture}
\newtheorem{question}[theorem]{Question}

\newtheorem{sublemma}[theorem]{Sublemma}

\newtheorem{remark}[theorem]{Remark}
\newtheorem{notation}[theorem]{Notation}
\newtheorem{terminology}[theorem]{Terminology}
\numberwithin{figure}{section}

\newcommand{\rcon}[1]{Conjecture~\ref{#1}}
\newcommand{\rcl}[1]{Claim~\ref{#1}}

\newcommand{\rprop}[1]{Proposition~\ref{#1}}
\newcommand{\rthm}[1]{Theorem~\ref{#1}}
\newcommand{\rlem}[1]{Lemma~\ref{#1}}

\newcommand{\rcor}[1]{Corollary~\ref{#1}}
\newcommand{\rdef}[1]{Definition~\ref{#1}}

\newcommand{\rsec}[1]{Section~\ref{#1}}

\newcommand{\rsubsec}[1]{Section~\ref{#1}}
\newcommand{\rrem}[1]{Remark~\ref{#1}}
\newcommand{\rnot}[1]{Notation~\ref{#1}}
\newcommand{\rter}[1]{Terminology~\ref{#1}}

\newcommand{\ZFC}{\mathsf{ZFC}}

\newcommand{\AD}{\mathsf{AD}}

\newcommand{\bbQ}{\mathbb{Q}}

\def\inseg{\trianglelefteq}

\def\k{\kappa}
\def\a{\alpha}
\def\b{\beta}
\def\d{\delta}

\def\l{\lambda}

\def\cop{{\sf{cop}}}

\def\P{{\mathcal{P} }}
\def\W{{\mathcal{W} }}
\def\Q{{\mathcal{ Q}}}
\def\mH{{\mathcal{ H}}}
\def\K{{\mathcal{ K}}}

\def\L{{\rm{L}}}

\def\R{{\mathcal R}}

\def\H{{\rm{HOD}}}
\newcommand{\OD}{\mathrm{OD}}

\def\M{{\mathcal{M}}}
\def\N{{\mathcal{N}}}

\def\T {{\mathcal{T}}}
\def\U{{\mathcal{U}}}
\def\S{{\mathcal{S}}}
\def\V{{\mathcal{V}}}
\def\X{{\mathcal{X}}}
\def\Y{{\mathcal{Y}}}
\def\Z{{\mathcal{Z}}}

\def\card#1{\left|#1\right|}

\def\cof{\mathop{\rm cof}\nolimits}
\def\iff{\mathrel{\leftrightarrow}}

\def\and{\mathrel{\kern1pt\&\kern1pt}}

\def\inseg{\triangleleft}
\def\insegeq{\trianglelefteq}
\def\fl#1{\lfloor#1\rfloor}

\def\<#1>{\langle\,#1\,\rangle}

 \input xy
 \xyoption{all}
\newcommand{\breals}{\omega^{\omega}}

\newcommand{\DC}{\mathsf{DC}}

\newcommand{\cP}{\mathscr{P}}
\newcommand{\bbP}{\mathbb{P}}

\newcommand{\MM}{\mathsf{MM}}

\newcommand{\PFA}{\mathsf{PFA}}
\newcommand{\pmax}{\mathbb{P}_{\mathrm{max}}}

\newcommand{\ZF}{\mathsf{ZF}}

 \newcommand{\gen}{{\rm gen}}

\newcommand{\ext}{{\rm ext}}

 \def\c{{\mathrm{N}}}

\begin{document}

\title[Nairian Models]{Nairian Models}

\author{Douglas Blue \and Paul B. Larson \and Grigor Sargsyan}

\thanks{The second author was supported in part by NSF research grant DMS-1764320.}

\thanks{The third author's work is funded by the National Science Center, Poland under the Maestro Call, registration number UMO-2023/50/A/ST1/00258.}

\begin{abstract}
We introduce a hierarchy of models of the Axiom of Determinacy called \emph{Nairian models}.
Forcing over the simplest Nairian model, we obtain a model of $\ZFC+{\sf{MM^{++}}}(c)+\neg\square_{\omega_3}+\neg\square(\omega_3)$.
Then, fixing $n\in [3, \omega)$, we design a Nairian model and force over it to produce a model of $\ZFC+{\sf{MM^{++}}}(c)+\forall i\in [2, n]\, \neg\square(\omega_i)$.
We also build a Nairian model that satisfies ${\sf{ZF}}+``\omega_1$ is a supercompact cardinal."

We obtain as corollaries of these constructions
\begin{enumerate}
    \item the consistent failure of  the Iterability Conjecture for the Mitchell-Schindler $\sf{K}^{\rmc}$ construction,

    \item the consistent failure of the Iterability Conjecture for the $\sf{K}^{\rmc}$ construction using $2^{2^{\dots 2^{\omega}}}$-complete (for any finite stack of exponents) background extenders, answering a strong version of a question asked by Steel, and

    \item a negative answer to Trang's question whether ${\sf{ZF}}+``\omega_1$ is a supercompact cardinal" is equiconsistent with ${\sf{ZFC}}+``$there is a proper class of Woodin cardinals that are limits of Woodin cardinals."
\end{enumerate}
These corollaries identify obstructions to extending the methods of (descriptive) inner model theory past a Woodin cardinal which is a limit of Woodin cardinals.
\end{abstract}

\subjclass[2020]{03E55, 03E57, 03E60, 03E45}

\maketitle

\tableofcontents
\setcounter{tocdepth}{1}

\section{Introduction}
In this paper we introduce \emph{Nairian models}, canonical Chang-type models which naturally arise in the context of the Axiom of Determinacy (AD). 

\begin{definition}
Let $\mH$ be a definable class of pairs of ordinals and $\lambda$ be an ordinal. 
The \emph{Chang model} at $\lambda$ relative to $\mH$ is $\c_{\lambda} = L(\mH\rest \lambda, \cup_{\b<\lambda}\b^{\omega})$, where $\mH\rest\lambda = \bigcup_{\gg<\l}\mH\cap (\gamma\times\gamma)$.
\end{definition}

Nairian models are Chang models where the predicate $\mH$ is taken to be the HOD of a model of AD. While this is the most natural method of defining Nairian Models, they also naturally arise inside HODs of models of AD. Both sources of Nairan Models give rise to interesting applications discussed in \textsection\ref{subsec:nm}.


Nairian models promise to unify, in a sense, parts of set theory like determinacy and forcing axioms.
Woodin showed that $\MM^{++}(\mathfrak{c})$, the forcing axiom \emph{Martin's Maximum}${}^{++}$ restricted to partial orders of size continuum, can be forced over sufficiently strong models of determinacy.
In forcing extensions of Nairian models, we obtain $\MM^{++}(\mathfrak{c})$ together with failures of Jensen's square principles which are not consequences of $\MM^{++}(\mathfrak{c}^+)$.
Viewing failures of square principles as fragments of forcing axioms, these results make substantial progress on the problem of forcing $\MM^{++}$ over a model of determinacy.

Nairian models and their forcing extensions have implications for inner model theoretic methodology.
The \emph{inner model problem} for a given large cardinal axiom is to construct a canonical model satisfying that large cardinal axiom.
In light of G\"odel's Incompleteness Theorem, solving its inner problem is arguably the strongest evidence possible for the consistency of a large cardinal axiom.
The inner model problem for measurable cardinals was solved in the 1960s, the inner model problem for Woodin cardinals was solved in the late 1980s by Donald Martin and John Steel \cite{MaSt94}, and the inner model problem for a supercompact cardinal was identified in the late 1960s and early 1970s and remains open.
The \emph{Inner Model Program} is to solve the inner model problem for all axioms of the large cardinal hierarchy.
It constitutes the central goal of inner model theory.
The second goal of inner model theory is to apply it, that is, to measure the extent to which mathematical statements are committed to the higher infinite.
These measurements, and developing the tools to carry them out, are the subject of \emph{core model theory}.

The high-water mark of the Inner Model Program is the early 2000s solution \cite{Ne02wlw} to the inner model problem for a Woodin cardinal which is a limit of Woodin cardinals.
This is a mild large cardinal axiom, far weaker than the large cardinals figuring into much contemporary research in infinitary combinatorics.
In particular, it is much weaker than a \emph{subcompact} cardinal, the inner model problem for which has been open for decades and is one of the subjects of this paper.
The construction of \cite{Ne02wlw} is a \emph{pure extender model}, and the key advance of \cite{Ne02wlw} is the proof that the model is \emph{iterable}.
(Conditioned on iterability hypotheses, there are pure extender mice reaching subcompact cardinals.)

Core model theory does not yet reach these heights.
The methods of \cite{Ne02wlw} work assuming there is a Woodin limit of Woodin cardinals to begin with, whereas core model theory requires methods for building inner models of large cardinals under assumptions which need not assert that large cardinals exist.
For example, \cite{Ne02wlw} does not apply to show that Martin's Maximum---widely believed to be equiconsistent with the existence of a supercompact cardinal, a large cardinal axiom much stronger than the existence of a subcompact cardinal---implies there is an inner model with a Woodin limit of Woodin cardinals.
The methods of core model theory are discussed in the next section.

Another methodology for defining canonical models is to analyze the HOD of a model of $\sf{AD}$.
We recall the definition of $\mathsf{AD}$.
Let $X$ be a set, and let $A\subseteq X^{\omega}$.
Then $G_A$ is the game in which Players I and II take turns to build a sequence $x= \langle x_i: i<\omega \rangle$ of elements of $X$.
Player I wins the run of $G_A$ iff $x\in A$.
$G_A$ is \emph{determined} if either of the players has a winning strategy.
$\AD_X$ is the assertion that for every $A\subseteq X^\omega$, $G_A$ is determined, and $\AD$ is $\AD_{\omega}$.

Deep work of John Steel and W.~Hugh Woodin showed that HOD of the inner model $L(\mathbb{R})$ is a canonical model, a \emph{hod mouse}, assuming $L(\bR)\models\sf{AD}$ \cite{StW16}.
The HOD analysis, which generalizes this work to larger models of $\sf{AD}$, falls far short of a Woodin limit of Woodins at least in part because no nontrivial theory extending the Axiom of Determinacy is known to have sufficient consistency strength. 
A candidate strong determinacy theory has been ``$\sf{ZF + AD} + \omega_1$ is supercompact,'' for Woodin has shown that it is consistent assuming arbitrarily large Woodin limits of Woodin cardinals \cite{woodin2022determinacy}. 
Nam Trang has asked  whether ``$\sf{ZF + AD} + \omega_1$ is supercompact'' is equiconsistent with a Woodin limit of Woodin cardinals \cite{Tr} and, with Daisuke Ikegami, conjectured that Woodin's result is optimal \cite[Conjecture 1]{ikegami2018supercompactness}.
The papers \cite{Tr14a,Tr15,Tr,Tr15MALQ,TrWi21} contribute to resolving Trang's question.

Reaching a new high-water mark via any of the above means has eluded inner model theory since \cite{Ne02wlw}.
This paper gives a partial explanation as to why:
The consistency strengths of useful combinatorics---like consecutive failures of threadability in the $\aleph_n$'s or the supercompactness of $\omega_1$---have been overestimated, and some of the known methods for constructing iterable models consistently fail.

\subsection{Core model theory and the iterability problem for $\sf{K^c}$}
Core model theory originates in work of Dodd and Jensen.
They defined the \emph{core model} $K$, the maximal inner model of the universe of sets satisfying certain properties (like generic invariance and weak covering) assuming an anti-large cardinal hypothesis.

It is a theorem of Woodin that the core model $K$ cannot exist if there is a Woodin cardinal, since in the presence of a Woodin cardinal, no generically absolute model can compute successors of singular cardinals correctly.
Mitchell and Steel addressed this problem by introducing $\sf{K^c}$ constructions.
A $\sf{K^c}$ construction is an inductive process which, if it converges, produces a ``certified'' core model, $\sf{K^c}$, of which the core model $K$ is a Skolem hull.
There are a variety of such constructions in the literature, but they all conform to the following general pattern.
The model $\sf{K^c}$ is constructed from a coherent sequence of extenders\footnote{An extender is a suitably coherent system of ultrafilters. See \cite[Definition 2.4]{OIMT}.}  meeting certain conditions.
The $\sf{K^c}$ construction is a sequence $\langle \M_\xi, \N_\xi, E_\xi: \xi\in {\sf{Ord}} \rangle$ with certain fine structural properties.\footnote{See \cite[Definition 2.1]{JSSS} and \cite[Definition 6.3]{OIMT} for exact statements.}
Each successor model $\M_{\xi+1}$ is a type of Skolem hull of $\N_\xi$, but it may not be well-defined, as it depends on certain \textit{fine structural} conditions which $\N_\xi$ may or may not satisfy.
Provided that $\M_\xi$ is defined, $\N_\xi$ is obtained by either closing under constructibility\footnote{Equivalently, under the rudimentary functions; see \cite{SchZem}.} or by adding an extender $E_\xi$ to $\M_\xi$ as a predicate.\footnote{$E_\xi$ is an $\M_\xi$-extender. It measures subsets of its critical point that are in $\M_\xi$.}
In the latter case, $\N_\xi$ is defined to be the pair $(\M_\xi, E_\xi)$.

To show that a $\sf{K^c}$-construction converges, then, it must be shown that $\M_{\xi}$ is defined at each successor step.
It is one of the deepest theorems of inner model theory that if all countable elementary submodels of $\N_\xi$ are \textit{iterable},\footnote{Iterability essentially means that all the ways of doing the iterated ultrapower construction produce well-founded models \cite[Section 3.1]{OIMT}.} then $\N_\xi$ satisfies all the fine structural properties needed to define $\M_\xi$.
The goal in defining a $\sf{K^c}$-construction is to specify the class of extenders $E_\xi$ in such a way that the countable elementary submodels of $\N_\xi$ are iterable.
The extenders so chosen are said to be \textit{certified} (hence the ``$\sf{c}$'' in ``$\sf{K^c}$'').

The certification condition is the defining property of a $\sf{K^c}$ construction, and there are different conditions in the literature.
For example, one can allow countably complete extenders, or total extenders.
Mitchell and Schindler introduced a clever certification method in which extenders are \textit{certified by a collapse} \cite{MitSch}.
Let ${\sf{K^c_{MiSch}}}$ denote the Mitchell-Schindler $\sf{K^c}$-construction. 
${\sf{K^c_{MiSch}}}$ has had two important applications: it was used in \cite{JSSS} to obtain a proper class model with arbitrarily large Woodin and strong cardinals from the \emph{Proper Forcing Axiom}, and it was used in \cite{KWM} to develop the core model theory of \cite{CMIP} without assuming the existence of a measurable cardinal. 

Fixing a certification method ${\sf{Cert}}$, the \emph{Iterability Conjecture for} $\sf{K^c_{Cert}}$ asserts that for each $\xi$, the countable elementary submodels of $\N_\xi$ are iterable.
\cite[Conjecture 6.5]{OIMT} explicitly states such a conjecture for Steel's certification method in \cite{CMIP}.
The results of \cite{JSSS} require the Iterability Conjecture for ${\sf{K^c_{MiSch}}}$.
The best known positive result on the Iterability Conjecture for $\sf{K^c}$ is proven in \cite{ANS01}: The Iterability Conjecture for ${\sf{K^c_{MiSch}}}$ is true if there is no inner model with a cardinal that is a limit of both Woodin cardinals and strong cardinals.\footnote{The results of \cite{ANS01} are a little stronger than stated here.}

The methodology introduced in \cite{JSSS} dispenses with the core model $K$ and establishes covering properties for its parent model $\sf{K^c}$.
According to this methodology, if the existence of $\sf{K^c}$ is established in $\ZFC$, then it can be used to derive large cardinal strength from various combinatorial statements.
It could be used to show, for instance, that $\MM$ implies the existence of an inner model with a \emph{superstrong} cardinal.

Theorem \ref{JSSSthrm} is a preliminary step toward showing that $\MM$ implies there is a subcompact cardinal in $\sf{K^c_{MiSch}}$.

\begin{theorem}[Jensen-Schimmerling-Schindler-Steel \cite{JSSS}]\label{JSSSthrm}
Assume $\aleph_2^\omega=\aleph_2$, and suppose that the principles $\square(\omega_3)$ and $\square_{\omega_3}$ both fail.
Let $g\subseteq \mathrm{Col}(\omega_3, \omega_3)$ be a $V$-generic filter.
If $V[g]\models ``{\sf{K^c_{MiSch}}}$ converges," then $({\sf{K^c_{MiSch}}})^{V[g]}\models$ ``there is a subcompact cardinal."
\end{theorem}

The hypotheses of Theorem \ref{JSSSthrm} hold in a forcing extension of a Nairian model.
Then---depending on whether the large cardinals needed to define the Nairian model are at least as strong as a subcompact---either the inner model problem for subcompact cardinals is solved, or one of the key methods developed to solve that same inner model problem is refuted.
Our Nairian model exists below a Woodin limit of Woodin cardinals.
As we have remarked, this large cardinal axiom is substantially weaker than a subcompact cardinal, so $({\sf{K^c_{MiSch}}})^{V[g]}\not\models$ ``there is a subcompact cardinal," and hence $V[g]\not\models ``{\sf{K^c_{MiSch}}}$ converges."

\begin{theorem}
	It is not provable in $\sf{ZFC}$ that the ${\sf{K^c_{MiSch}}}$ construction converges.
\end{theorem}

The analogue of Theorem \ref{JSSSthrm} holds for the construction $\sf{K^c}_{2^{\omega}\text{-}closed}$ in which $2^{\omega}$-closed extenders are allowed on the extender sequence.
The hypotheses of that analogue again hold in a forcing extension of a larger model of $\sf{AD}$ we define, and so we refute the conjecture of Steel \cite{irvine2023problemlist} that $\sf{K^c}_{2^{\omega}\text{-}closed}$ provably converges.

\begin{theorem}
	It is not provable in $\sf{ZFC}$ that the $\sf{K^c}_{2^{\omega}\text{-}closed}$ construction converges.
\end{theorem}

The same holds for requiring $2^{2^{\dots 2^{\omega}}}$-closure for any finite stack of exponents.

The $\sf{K}^c$ methodology has been stuck between domestic mice and mice with a Woodin limit of Woodin cardinals.
Domestic implies that it is a combinatorial feature of iteration trees on domestic mice that on domestic levels of $K^c$, the existence of realizable branches suffices to obtain uniqueness of branches, and hence iterability, regardless of backgrounding condition.
Thus it has been reasonable to believe that the definitive branch existence theorems have been obtained and what is missing is a proof of branch uniqueness.
The theorems of this paper do not refute the existence of realizable branches, but they show that consistently there will not be unique (modulo any reasonable criterion) realizable branches.
In such situations, the $\sf{K}^c$ construction reaches a level that is not iterable.
The paper suggests that generalizing $\sf{K}^c$ alone is insufficient for iterability; a general combinatorial result does not underlie iterability.

The third author recently proved that the hypothesis of \rthm{main theorem} is weaker than a Woodin cardinal that is a limit of Woodin cardinals (see \cite{LDC}).
Hence the hypothesis of \rthm{JSSSthrm} is weaker than a Woodin cardinal that is a limit of Woodin cardinals, and hence

\begin{theorem}\label{Kms nonconvergence}
    It is consistent relative to a Woodin limit of Woodin cardinals that the ${\sf{K^c_{MiSch}}}$ construction fails to converge.
\end{theorem}

We conjecture that a pure extender mouse with proper classes of Woodin cardinals, strong cardinals, and strong cardinals reflecting strong cardinals is a consistency strength upper bound on the existence of the Nairian models in this paper. 
From the determinacy perspective, it suffices that the Largest Suslin Axiom holds and, letting $\theta_{\kappa}$ be the largest Suslin cardinal, in $\H|\theta_{\kappa+1}$ there are three strong cardinals which are limits of Woodin cardinals.

\subsection{Forcing over models of determinacy}
Strong mathematical theories are traditionally shown to be relatively consistent by forcing over models of $\ZFC$ satisfying some large cardinal axiom.
For example, Shelah showed that Martin's Maximum is consistent relative to the existence of a supercompact cardinal by collapsing a supercompact to be $\omega_2$ via revised countable support iteration of semiproper forcings.
Forcing instead over models of $\ZF+\AD$ has been used to show the consistency of $\mathsf{ZFC}$ together with certain combinatorial structures on small uncountable cardinals, and to do so with more optimal hypotheses than proofs using $\ZFC$ models.

Let $\Theta$ be the least ordinal $\gg$ such that the reals cannot be surjected onto $\gg$.

\begin{definition}\label{thetareg}\normalfont  $\treg$ is the theory $\ZF+\ADR+``\Theta$ is a regular cardinal."\footnote{Solovay showed \cite{So78}  that the theory $\treg+V=L(\powerset(\bR))$ implies the Axiom of Dependent Choice ($\DC$). The axiom $V=L(X)$ states that $V$ is the minimal model of $\ZF$ containing all ordinals and the set $X$. It is shown in \cite{HMMSC} that the theory $\treg$ is consistent relative to a Woodin cardinal that is a limit of Woodin cardinals.}
\end{definition}

By forcing over a model of $\treg$, Steel and Van Wesep obtained a model of $\ZFC$ in which the non-stationary ideal on $\omega_1$ is saturated \cite{SteelWesep}.
This was the first indication of deep connections between models of $\AD$ and the kind of models of $\ZFC$ that are widely studied by set theorists.
Building on the Steel-Van Wesep work, Woodin \cite{Wo10} developed a general forcing machinery for constructing $\ZFC$ models from models of determinacy.
This $\pmax$ method is so powerful that it seems as though there is nothing about small infinite cardinals that can be forced from large cardinals that cannot be forced over a model of $\AD$.
Our results are the strongest confirmation of this intuition to date.
They are motivated by the following theorem of Woodin \cite[Chapter 9.7]{Wo10}, where $\Add(\kappa, 1)$ is the partial order to add a subset of $\kappa$ by pieces of cardinality less than $\kappa$.

 \begin{theorem}[Woodin]\label{woodin mmc}
 Assume $\treg+V=L(\powerset(\bR))$.
 Then $\pmax*\Add(\omega_3, 1)$ forces $\sf{MM^{++}(\mathfrak{c})}$. 
 \end{theorem}
 
At the time, the best upper bound for $\sf{MM^{++}(\mathfrak{c})}$ was a cardinal $\kappa$ that is $\kappa^+$-supercompact.

Woodin also showed that, over models of $\treg+V=L(\powerset(\bR))$, the poset ${\mathrm{Col}}(\omega_1, \bR)*\Add(\omega_2, 1)$\footnote{This poset first enumerates the reals in order type $\omega_1$ and then adds a Cohen subset to $\omega_2$, forcing $\ZFC$ over any model of $\treg+V=L(\powerset(\bR))$.} forces ${\sf{CH}}+``$there is an $\omega_1$-dense ideal on $\omega_1$" (see \cite{sargsyan2021ideals}). 
Woodin's work thus obtains significant combinatorial structures on $\omega_1$ and $\omega_2$ in forcing extensions of models of determinacy.
It immediately raises the question \emph{What combinatorial structures can exist on the larger $\omega_n$'s of a $\ZFC$ forcing extension of a model of $\treg$?} 
We pursue here the more particular question \emph{Which consequences of $\sf{MM}$ can be forced over models of determinacy?}

These questions can be made exact in various ways, but the line pursued in this paper was initiated by \cite{CLSSSZ}, the first paper to force important consequences of $\sf{MM}$ for $\omega_{3}$ over models of determinacy.
The foundational question that \cite{CLSSSZ} contributes to answering is

 \begin{question}\label{ques mm pmax}
 Is it possible to force $\ZFC$ and $\sf{MM}$, or significant consequences of $\sf{MM}$ for cardinals greater than $\omega_{2}$, over models of determinacy via a poset of the form $\mathbb{P}=\pmax*\mathbb{Q}$, where $\mathbb{P}$ is countably complete and homogeneous?
 \end{question}

Asper\'o and Schindler \cite{AsperoSchindler} have shown that ${\sf{MM^{++}}}$ implies Woodin's $\mathbb{P}_{\mathrm{max}}$ Axiom ${\sf{(*)}}$.
But it is open whether ${\sf{MM^{++}}}$ is consistent with Woodin's Axiom ${\sf{(*)^+}}$ (see \cite{Wo10}). 
A positive, complete answer to Question \ref{ques mm pmax} should solve this problem (at least if one requires additionally that the $\Theta$ of $V$ becomes $\omega_3$ in the generic extension).
A route towards this, which the present paper can be construed as a step in, is the topic of \cite[Chapter 3]{blue2023models}.

\cite{CLSSSZ} forces $\sf{MM^{++}(c)}+\neg\square_{\omega_2}$ over a model of $\treg$.
Square principles (see below) are incompactness principles informing much work in infinitary combinatorics and inner model theory.
Square fails at subcompact cardinals, and stronger large cardinals imply their eventual total failure.
Todor\v{c}evi\'c showed that in fact a weak form of $\MM$ implies that $\square_\k$ fails for every uncountable $\k$ \cite{To84}.
However, $\neg\square_{\omega_2}$ is a consequence of $\sf{MM^{++}(\mathfrak{c}^+)}$, not of $\sf{MM^{++}(\mathfrak{c})}$.
So while the results of \cite{CLSSSZ} do go beyond \cite{Wo10}, they do not go much further.
\cite{CLSSSZ} essentially pushes the basic $\pmax$ idea as far as possible below a major obstacle which we now describe.

 \subsection{Higher models of determinacy}
 
Suppose our goal is to obtain a model of $\sf{MM^{++}}(\mathfrak{c})+\neg\square_{\omega_3}+\neg\square(\omega_3)$ by forcing over a model of $\treg$.
The strategy of \cite{CLSSSZ} is to start with some model $M$ of determinacy, force over it with $\pmax$ to get $\sf{MM^{++}}(\mathfrak{c})$, and then force with $\Add(\omega_3, 1)$ to make the Axiom of Choice hold. 
To go beyond \cite{CLSSSZ}, it is very likely that we must continue by forcing with $\Add(\omega_4, 1)$.
Thus, the poset for forcing $\sf{MM^{++}}(\mathfrak{c})+\neg\square_{\omega_3}+\neg\square(\omega_3)$ must be $\pmax*\Add(\omega_3, 1)*\Add(\omega_4, 1)$, or some variant of it.\footnote{Currently $\pmax$ is the only poset known to force $\sf{MM^{++}}(\mathfrak{c})$ over a model of determinacy.}
 
 The obstacle is that if we were to succeed---if $M^{\pmax*\Add(\omega_3, 1)*\Add(\omega_4, 1)}$ is a model of the theory $\sf{MM^{++}}(\mathfrak{c})+\neg\square_{\omega_3}+\neg\square(\omega_3)$---then we must have that $\Theta^M=\omega_3$ and $(\Theta^+)^M=\omega_4$. 
 This then suggests that $M\models \neg\square_{\Theta}$.
 At the time of writing \cite{CLSSSZ}, no model of determinacy with this property was known to exist.
 If $M\models \neg\square_{\Theta}$, then it must have non-trivial, non-compact structure above $\Theta$, and such structure is not postulated to exist by determinacy axioms, which are statements about sets of reals.
 The theory that axiomatizes $M$ must be a theory not only about sets of reals, then, but also sets of sets of reals. 
 The theory cannot be similar to the theories of the currently used canonical inner models, as such theories usually imply that square like principles hold \cite{SquareCore}.
 For example, if $M$ satisfies that $V=L(\powerset(\bR))$ then $M\models \square_\Theta$, and if $M$ is some kind of mouse or a hybrid mouse over $\powerset(\bR)$, then $M\models \square_\Theta$. 

 The first important point then is that to go beyond \cite{CLSSSZ}, we must first construct models of determinacy that have significant, non-compact combinatorial structure beyond the portion coded by their sets of reals.
In this paper, we will force over such models, Chang-type models first constructed in \cite{CCM}.
In the second part of this paper we will analyze one such model and show that it satisfies the hypothesis of \rthm{mainthrm} (see \rthm{main theorem}).

The second important point was mentioned after Theorem \ref{JSSSthrm}:
Since the theory $\sf{MM^{++}}(\mathfrak{c})+\neg\square_{\omega_3}+\neg\square(\omega_3)$ implies the hypothesis of \rthm{JSSSthrm}, if we were to succeed in obtaining a model of $\sf{MM^{++}}(\mathfrak{c})+\neg\square_{\omega_3}+\neg\square(\omega_3)$ by forcing over a model of determinacy, then we would have to either solve the inner model problem for a subcompact cardinal or refute one of the key methods developed to solve that inner model problem.
Thus the strength of the underlying determinacy model is a key issue.

 \subsection{Universally Baire sets and Chang models over derived models.}\label{uBsubsec}

A set of reals is \emph{universally Baire} if all of its continuous preimages in compact Hausdorff spaces have the property of Baire.
Equivalently, 

\begin{definition}[Feng-Magidor-Woodin, \cite{FMW92}]\label{def:uB}\normalfont\hspace{.1in}
\begin{enumerate}
    \item  Let $(S,T)$ be trees on $\omega \times \kappa$, for some ordinal $\kappa$, and let $Z$ be any set.
    The pair $(S,T)$ is $Z$-\emph{absolutely complementing} if $p[S] = \bR \setminus p[T]$ in every ${\sf{Coll}}(\omega,Z)$-generic extension of $V$.
    \item  A set of reals $A$ is \emph{universally Baire (uB)} if for every set $Z$, there are $Z$-absolutely complementing trees $(S,T)$ with $p[S] = A$.
\end{enumerate}
\end{definition}

\begin{notation}[Universally Baire sets]\label{ub sets} \normalfont
	$\Gamma^\infty$ denotes the set of universally Baire sets.
    If $G$ is generic, let $\Gamma^\infty_G=(\Gamma^\infty)^{V[G]}$. 
\end{notation}

Woodin's \emph{Derived Model Theorem} is the primary method for deriving determinacy models from large cardinals.
Theorem \ref{dmt} is one version of it.
We write ${\sf{Coll}}(\omega, {<}\l)$ for the Levy collapse of $\l$ to $\omega$.
Given $g\subseteq {\sf{Coll}}(\omega, {<}\l)$ and $\a<\l$, let $g_\a=g\cap {\sf{Coll}}(\omega, {<}\a)$.
If $g$ is any generic, then $\bR_g=\bR^{V[g]}$. 

Suppose that $\l$ is an inaccessible cardinal and $g\subseteq {\sf{Coll}}(\omega, {<}\l)$ is generic over $V$.
Working in $V[g]$, let $\Gamma_g$ be the set of all $A\subseteq \bR_g$ such that for some $\a<\l$, there is a pair $(T, S)\in V[g_\a]$ such that $V[g_\a]\models ``(T, S)$ is ${<}\l$-absolutely complementing"\footnote{$\beta$-absolutely complementing for every $\beta<\l$.} and $A=(p[T])^{V[g]}$. 

\begin{theorem}[Woodin, \cite{St07DMT}]\label{dmt}
Suppose $\lambda$ is a regular cardinal that is a limit of Woodin cardinals.
Let $g\subseteq {\sf{Coll}}(\omega, {<}\l)$ be generic.
Then $L(\Gamma_g, \bR_g)\models {\sf{AD^+}}$.
\end{theorem}

The Chang Model is the model smallest inner model of $\ZF$ containing every countable sequence of ordinals.
Variations of this model can be defined be e.g.~adding predicates or restricting to the set of countable sequences from a specific ordinal.
The Nairian models we use in this paper are constructed over derived models computed inside \textit{self-iterable} universes, universes that ``see'' their own iteration strategies.\footnote{See the introduction of \cite{SteelCom}.}
In inner model theory, self-iterable universes are known as \emph{hod mice}.
While self-iterability is a crucial property that is used heavily in the second part of this paper, it is not needed to define the Chang models that we will use.
What is needed is the main theorem of \cite{CCM}, summarized as follows (see also \rsec{sec: chang model}).

\begin{theorem}[Sargsyan, \cite{CCM}]\label{sum: chang models}
Assume $V$ is a hod mouse and $\lambda$ is an inaccessible limit of Woodin cardinals. Then there exist a transitive $M\subseteq H_{\l^+}^V$ and an elementary embedding $j: H_{\l^+}^V\rightarrow M$ such that whenever $g\subseteq {\sf{Coll}}(\omega, <\l)$ is $V$-generic, $L(M^\omega, \Gamma_g, \bR_g)\models \sf{AD^+}$.
\end{theorem}

Assuming the existence of a $\k<\l$ which is a ${<}\l$-strong cardinal, models of the form $L(M^\omega, \Gamma_g, \bR_g)$ are genuinely new determinacy models with noncompact structure beyond their $\Theta$.
This is the intuitive content of \rthm{main theorem}, which identifies a particular hod mouse whose $L(M^\omega, \Gamma_g, \bR_g)$ satisfies the properties listed in \rdef{bowtie}.\footnote{The actual constructions differ from the expository constructions we've employed in this introduction.}

In the determinacy model constructed in \rthm{main theorem}, there exists, for each set $X$, a normal, fine measure on $\powerset_{\omega_1}(X)$.
Thus we answer Trang's question negatively.

\begin{theorem}
The theory $\sf{ZF} + \sf{AD}_{\bR} +$ ``$\omega_1$ is supercompact'' is strictly weaker than a Woodin limit of Woodin cardinals in consistency strength.
\end{theorem}

\subsection{Future research}

In the Nairian model of \rthm{main theorem}, every successor cardinal above $\Theta$ is regular.
This arrangement contrasts sharply with the cardinal structure below $\Theta$ where, for instance, $\omega_1$ and $\omega_2$ are regular but $\omega_n$ has cofinality $\omega_2$ when $n\geq 3$ .
Together with \rthm{mainthrm}, the cardinal structure above $\Theta$ suggests the possibility of forcing the theory ``For every regular cardinal $\k\geq \omega_2$, $\square(\kappa)$ fails," or even full $\sf{MM^{++}}$, over a determinacy model.
The class of determinacy models over which this may be possible is the class of Chang-type models that have been identified in \cite{CCM} and produced via \rthm{sum: chang models}.
Developing the first order theory of the models of ${\sf{AD^+}}$ that have the form $L(M^\omega, \Gamma_g, \bR_g)$ is one of the most important open problems of the area.

\begin{conjecture}\label{cardinals are reg} 
Suppose that $V$ is a hod mouse and $\lambda$ is a Woodin limit of Woodin cardinals.
Let $g\subseteq {\sf{Coll}}(\omega, <\l)$ be $V$-generic, and let $(j, M)$ be as in \rthm{sum: chang models}.
Let $\nu=j(\l)$ and $N=j(V_\l)$.
Then 
\begin{enumerate}
\item $L_\nu(N^\omega, \Gamma_g, \bR_g)\models \ZF$, and 
\item $L_\nu(N^\omega, \Gamma_g, \bR_g)\models``$for every cardinal $\eta\geq\Theta$,  $\eta^+$ is a regular cardinal."
\end{enumerate}
\end{conjecture}

	

There is a pure ${\sf{AD^+}}$ version of \rcon{cardinals are reg}. 
In the context of $\sf{AD^+}$, a cardinal $\k$ is $\mathrm{OD}$-\emph{inaccessible} if for every $\gg<\k$ there is no ordinal definable surjection $f:\powerset(\gamma)\rightarrow \k$.
The \emph{Solovay sequence} is the sequence $(\theta_\a:\a\leq \Omega)$ that enumerates in increasing order the set of $\mathrm{OD}$-inaccessible cardinals and their limits below $\Theta$.
We write $M|\xi$ for $V_\xi^M$.

\begin{conjecture}\label{ad version}
Assume  $\sf{AD^+}$.
 Let $(\theta_\a: \a\leq \Omega)$ be the Solovay Sequence.
 Suppose $\a+1\leq \Omega$ is such that $\H\models ``\theta_{\a+1}$ is a limit of Woodin cardinals".
 Let $\d=\theta_{\a+1}$, $M=V_{\d}^{\H}$, and $N=L_{\d}(\bigcup_{\xi<\d} (M|\xi)^\omega)$.
 Then
\begin{enumerate}
\item $N\models \ZF$, 
\item $\Theta^N=\theta_\a$,
\item $N\models``$for every cardinal $\eta\geq\Theta$,  $\eta^+$ is a regular cardinal," and 
\item for every $\eta<\d$, $\cf((\eta^+)^N)\geq \theta_\a$.
\end{enumerate}
\end{conjecture}

The model $N$ was first studied by Sargsyan and Steel, although it first appeared in \cite{Wo10}.
\rthm{main theorem} verifies a restricted version of \rcon{ad version}.
Regarding the full conjecture, recently Woodin has shown conclusion (1), (2), and (4) in the context of the short extender HOD analysis.
In the same context, Steel showed that the sets of reals in $N$ are those of Wadge rank less than $\theta_{\alpha}$ \cite[Theorem 11.5.7]{SteelCom}.

The exact relationship between \rcon{cardinals are reg} and \rcon{ad version} is unclear; the $N$ of \rcon{ad version} may in fact be the $N$ of some self-iterable universe.
\rcon{ad version} and answering Question \ref{ques mm pmax} with Nairian models are the subjects of \cite{blue202komitas}.

\subsection{The Nairian cardinals}\label{subsec:nm}
Let $\mH$ represent the class of pairs of ordinals $(\alpha, \beta)$ such that the $\alpha$th element of the standard definability order of $\H$ is an element of the $\beta$th.
(This is merely a technical convenience that allows us to give a concise statement of results from the second part of this paper; the only property of $\mH$ we use in the first part of the paper is that it is a definable class of pairs of ordinals.)

Given an ordinal $\gamma$, we write $\mH\rest \gamma$ for $\mH \cap (\gamma \times \gamma)$
and $\c_\gamma^{-}$ for the structure $L_\gamma(\mH, \cup_{\b<\gg}\b^{\omega})$, which is constructed relative to the predicate $\mH$, adding (for each ordinal $\alpha < \gamma$) all $\omega$-sequences from $\alpha$ at stage $\alpha + 1$.
Note that $\gamma$ is the ordinal height of this structure.

Let $\c_\gamma$ denote $L(\mH\rest \gamma, \cup_{\b<\gg}\b^{\omega})$ and $\c^+_\gamma$ denote $\H_{\cup_{\b<\gg}\b^{\omega}}$.

\begin{definition}\label{def:nairian model in ad}\normalfont  Assume ${\sf{AD^+}}$ and suppose $\gamma\leq \Theta$ is an ordinal. Then the \emph{Nairian model} at $\gamma$ is $\c^-_{\gamma}$. 
\end{definition}


\begin{definition}\label{sec:ncardinals} Assume ${\sf{AD^+}}$ and suppose $\gamma\leq \Theta$ is an ordinal. Then 
\begin{enumerate}
    \item $\gg$ is a \emph{Nairian cardinal} if $N^-_\gg\models \ZF$ and $\gg$ is ${\mathrm{OD}}$-inaccessible,\footnote{That is, $\gg$ is a member of the Solovay sequence.} and

    \item $\gg$ is an \emph{almost Nairian cardinal} if $N^-_\gg\models \ZF$ and if $\gg$ is not an ${\mathrm{OD}}$-inaccessible cardinal, then for any set $X\in N^-_\gg$, $\powerset(X)\cap N_\gg=\powerset(X)\cap N^-_\gg$.
\end{enumerate}
 
\end{definition}

Assuming the short extender HOD analysis and $\mathrm{AD}^+$, Woodin showed that a $\mathrm{OD}$-inaccessible cardinal which is a limit of Woodin cardinals in HOD is a Nairian cardinal.
The third author has shown that the hypothesis of this theorem is equiconsistent with a Woodin cardinal which is a limit of Woodin cardinals. However, \rcor{zf in c} implies that almost Nairian cardinals are strictly weaker than a Woodin cardinal that is a limit of Woodin cardinals. This is a key fact that leads to our calculations of upper bounds of failures of squares and failure of $K^c$ constructions to converge. Because of this, we feel it is important to study almost Nairian cardinals as well as the Nairian cardinals. 

Nairian models also arise naturally in hod mice. Suppose $\P\models \ZFC$ is a hod premouse and $\k$ is an inaccessible cardinal in $\P$. Let $\M_\infty$ be the direct limit of all iterates $\Q$ of $\P|\k$ that are in $\P$ and such that $\T_{\P|\k, \Q}$, the $\P$-to-$\Q$ iteration tree, is of length $<\k$. Let $g\subseteq Coll(\omega, <\k)$ be $\P$-generic. Then $N$ is the Nairian model of $\P$ at $\k$ if, letting $\l={\sf{Ord}}\cap \M_\infty$, $N=L_\l(\M_\infty^\omega)$.  Nairian models of hod mice also have interesting applications. For example, \cite{gappo2023changmodelsderivedmodels} uses them to reprove a theorem of Woodin that assuming proper class of Woodin limit of Woodin cardinals, the Chang Model is a model of AD. 

Recently, Lukas Koschat, Sandra M\"uller, and the third author showed that if $\kappa$ is a supercompact cardinal and there are class many supercompact cardinals, then the derived model at $\k$ has cofinally many LSA pointclasses. The following conjecture is then a natural next step, though it may be much harder to prove.

\begin{conjecture} Assume $\kappa$ is a supercompact cardinal and there are class many supercompact cardinals. Let $M$ be the derived model at $\k$. Then $M\models ``\Theta$ is a limit of Nairian cardinals." 
\end{conjecture}

\subsection*{Acknowledgements}
The authors thank Benjamin Siskind, John Steel, and W.~Hugh Woodin for helpful conversations about the content of this paper.

\section{Preliminaries}
\subsection{Square principles}
We recall the square principles introduced by Ronald Jensen \cite{Je72}.

\begin{definition}[$\square_{\kappa}$]\label{def of square}\normalfont
Let $\kappa$ be an uncountable cardinal. 
Then there is a sequence $\langle C_{\alpha} : \alpha < \kappa^{+} \rangle$ such that for each $\alpha < \kappa^{+}$,
\begin{itemize}
\item each $C_{\alpha}$ is a closed cofinal subset of $\alpha$,
\item for each limit point $\beta$ of $C_{\alpha}$, $C_{\beta} = C_{\alpha} \cap \beta$, and
\item the ordertype of each $C_\alpha$ is at most $\kappa$.
\end{itemize}
\end{definition}

\begin{definition}[$\square(\gamma, \delta)$]\label{square round}\normalfont  
Let $\gamma$ be an ordinal and $\delta$ a cardinal.
Then there is a sequence $\langle \mathcal{C}_{\alpha} \mid \alpha < \gamma \rangle$ such that
\begin{itemize}
\item
For each $\alpha < \gamma$,
\begin{itemize}
\item
$0 < |\mathcal{C}_{\alpha}| \leq \delta$,
\item
each element of $\mathcal{C}_{\alpha}$ is club in $\alpha$, and
\item
for each member $C$ of $\mathcal{C}_{\alpha}$, and each limit point $\beta$ of $C$, \[C \cap
\beta \in\mathcal{C}_{\beta}.\]
\end{itemize}
\item
There is no thread through the sequence, that is, there is no club $E \subseteq
\gamma$ such that $E \cap \alpha \in \mathcal{C}_{\alpha}$ for every limit point $\alpha$ of $E$.
\end{itemize}
\end{definition}

A $\square(\gamma,\delta)$-\emph{sequence} is a sequence $\langle C_{\alpha} : \alpha < \gamma \rangle$ as in Definition \ref{square round}.
A \emph{potential} $\square(\gamma,\delta)$-sequence is a sequence $\langle C_{\alpha} : \alpha < \gamma \rangle$ satisfying all but the last condition in Definition \ref{square round}.
An elementary argument shows that if the cofinality of $\gamma$ is greater than $|\delta|^{+}$, then each potential $\square(\gamma, \delta)$-sequence has at most $|\delta|$ many threads.
We note that if $\delta < \eta$, then $\square(\kappa, \delta)$ implies $\square(\kappa, \eta)$.
Finally, the square principles are related.
$\square(\gamma,1)$ is the nonthreadability principle $\square(\gamma)$, and
for any cardinal $\kappa$, $\square_{\kappa}$ implies $\square(\kappa^{+})$, as the ordertype condition in the definition of $\square_{\kappa}$ ensures the nonthreadability of the sequence.

Todor\v{c}evi\'c showed \cite{To84} that if $\gamma$ has cofinality at least $\omega_{2}$, then the restriction of the Proper Forcing Axiom ($\PFA$) to partial orders of cardinality $\gamma^{\omega}$ implies $\square(\gamma, \omega_{1})$ fails.
For $\gamma < \omega_3$, this fragment of $\PFA$ follows from $\MM(\mathfrak{c})$, since $\MM(\mathfrak{c})$ implies that $\mathfrak{c} = \aleph_{2}$ \cite{FMS88}.
Our only use of $\MM^{++}(\mathfrak{c})$ will be to apply Todor\v{c}evi\'c's \rthm{todthrm}.

\begin{theorem}[Todor\v{c}evi\'c]\label{todthrm}
	If $\MM(\mathfrak{c})$ holds, and $\gamma \in [\omega_{2}, \omega_{3})$ has cofinality $\omega_{2}$, then $\square(\gamma, \omega_{1})$ fails. 
\end{theorem}

Theorem \ref{weakmain} is one formulation of the main theorem of the first part of this paper (see Theorem \ref{mainthrm}).

\begin{theorem}\label{weakmain}
The consistency of $\ZFC$ plus the existence of a Woodin limit of Woodin cardinals implies, for each $n \in \omega$, the consistency of $\ZFC$ + $\aleph_2^\omega=\aleph_2$ + $\forall m \in (n\setminus 2) (\neg\square(\omega_{m}, \omega))$.
\end{theorem}

It should be possible, given Todor\v{c}evi\'c's theorem, to replace each instance of $\neg\square(\omega_{n}, \omega)$ with $\neg\square(\omega_{n}, \omega_1)$.

\subsection{The Nairian Model at $\lambda$ and $\Join^\xi_{\lambda}$}\label{cmjoinssec}
If $A\subseteq\breals$, then $w(A)$ denotes the Wadge rank of $A$.
For any ordinal $\alpha$, $\Delta_{\alpha}$ denotes the set of sets of reals of Wadge rank less than $\alpha$.

We will work with models of $\AD^+$\footnote{An ostensible extension of the Axiom of Determinacy due to Woodin, see \cite{La17}.} in which some ordinal $\lambda$ satisfies the following principle for all $n \in \omega$.

\begin{definition}[$\Join^{n}_\lambda$]\label{bowtie} \normalfont
Let $\kappa=\Theta^{\c_\lambda}$.
Then $\Join^{n}_\lambda$ is the statement that
\begin{enumerate}
\item the order type of the set $C=\{\tau\in [\k, \l]: \c_\l\models``\tau$ is a cardinal$"\}$ is at least $n+1$,
\item $\l<\Theta$ and $\l\in C$,
\item $\kappa$ is a regular member of the Solovay sequence,
\item letting $\langle \k_i: i\leq n \rangle$ be the first $n+1$ members of $C$ enumerated in increasing order,\footnote{Thus, $\k_0=\k$.} for every $i\leq n$, $\c^+_{\lambda}\models ``\k_i$ is a regular cardinal",
\item $\c_\lambda^{-} \cap \cP(\bR)
=\c_{\lambda}^+\cap \cP(\bR)=\Delta_\kappa$,
\item\label{bowtiesix} for all $i\leq n$, $\cP(\k_i^{\omega}) \cap \c_{\k_{i+1}}^- = \cP(\k_i^{\omega}) \cap \c^{+}_{\lambda}$, and
\item for all $i+1\leq n$, $\cf(\k_{i+1}) \geq \kappa$. 
\end{enumerate}
\end{definition}

In light of Proposition \ref{Thetarealsrem}, $\Join^{m}_{\lambda}$ implies $\Join^{n}_{\kappa_{n}}$ for all $n \leq m$ in $\omega$. 
Since $\AD^{+}$ implies that successor members of the Solovay sequence below $\Theta$ have countable cofinality \cite[Theorem 13.12]{La17}, $\AD^{+}$ + $\Join_{n}^\xi$ implies that $\kappa = \Theta^{\c_\lambda}$ is a limit member of the Solovay sequence.
Woodin has shown that $\AD^{+}$ implies that $\AD^{+}$ holds in every inner model of $\ZF$ containing $\bR$ and that $\ADR$ holds if and only if the Solovay sequence has limit length (see \cite{La17}).
Thus $\Join^{0}_{\lambda}$ implies $L(\Delta_{\kappa}) \models \ADR$.

Let $\ddagger$ stand for the theory $\sf{ZF}$ + $V=L(\cP(\bR))$ + $\sf{AD}_{\mathbb{R}}$ + ``$\Theta$ is regular", which, as we have just seen, is satisfied by $L(\Delta_{\kappa})$. 
Results of Solovay \cite{So78} show that $\ddagger$ implies $\DC$ (the statement that every tree of height $\omega$ without terminal nodes has a cofinal branch) and that the sharp of each set of reals exists.
By results of Becker and Woodin, $\ADR + \DC$ implies that all subsets of $\breals$ are Suslin, and thus that $\AD^{+}$ holds (see \cite{La17}).

Proposition \ref{Thetarealsrem} shows that the third item in the definition of $\Join^{n}_{\lambda}$ could equivalently be replaced with the statement that $\c^{+}_{\lambda} \cap \cP(\bR) = \Delta_{\kappa}$. 

\begin{proposition}[$\ADR$ + ``$\Theta$ is regular"]\label{Thetarealsrem}
$\cP(\bR) \subseteq \c^{-}_{\Theta}$.
\end{proposition}

\begin{proof}
Since $\AD$ holds, every ultrafilter on the set of finite subsets of any element of $\Theta$ is ordinal definable from an element of $\Theta$.
There is then an $\OD$-wellordering $\langle \nu_{\alpha} : \alpha < \Theta \rangle$ of these ultrafilters.
Let $T$ be set of tuples $(\alpha, \beta, \gamma, \delta)$ such that $\nu_\beta$ projects to $\nu_\alpha$, and the corresponding factor map sends $\gamma$ above $\delta$.

Fix $A \subseteq \breals$. Then (since $\ADR$ holds and $\Theta$ is regular) $A$ is homogeneously Suslin,
as witnessed by some set of measures $\langle \mu_{s} : s \in \omega^{\less\omega}\rangle$. Let $f \colon \omega^{\less\omega} \to \Theta$ be such that each $\mu_{s}$ is $\nu_{f(s)}$. Then $A$ is the set of $x \in \breals$ for which there is no $g \colon \omega \to \Theta$ such that, for all $n \in \omega$, \[(f(x \rest n), f(x \rest (n+1)), g(n), g(n+1)) \in T.\] 
\end{proof}

In the first part of the paper, we will work with models of $\exists \lambda \Join^{n}_{\lambda}$ for positive integers $n$.
Such models are given by the following theorem from the second part of this paper (see Theorem \ref{main theorem}).

\begin{theorem}\label{inputthrm}
  Suppose there exists a Woodin cardinal which is a limit of Woodin cardinals.
  Then in a forcing extension, there is an inner model satisfying $\ddagger$  + $\exists \lambda \forall n < \omega \Join^{n}_{\lambda}$.
\end{theorem}

\subsection{Variants of $\DC$}

The principle of Dependent Choice ($\DC$) can be varied by restricting the nodes of the tree to some set, or by considering trees of uncountable height.

Given a set $X$ and a cardinal $\gamma$, let $\DC_{\gamma}\rest X$ denote the restriction of $\DC_{\gamma}$ to binary relations on $X$. (We also call this $\DC_{\gamma}$ \emph{for relations on} $X$.)
It is easy to see that $X$ is wellordered if and only if $\DC_{\gamma} \rest X$ hold for all ordinals $\gamma$. 

Given a binary relation $R$ on a set $X$ and an ordinal $\delta$, say that $f\colon\delta\rightarrow X$ is an $R$-\emph{chain} (of length $\delta$) if $f(\alpha)R f(\beta)$ holds for all ordinals $\alpha < \beta$ below $\delta$.
Given an ordinal $\eta$, say that $R$ is 
$\eta$-\emph{full} if for all $\alpha < \beta < \delta$, every $R$-chain of length $\alpha$ has an extension to an $R$-chain of length $\beta$. 
Then $\sf{DC}_{\gamma}$ holds for an ordinal $\gamma$ if for every ordinal $\eta\leq \gamma$ and $\eta$-full every binary relation $R$, there is an $R$-chain $f \colon \eta \rightarrow X$.
$\DC$ is just $\DC_{\omega}$.

\subsection{$\pmax$}\label{pmaxssec}
We will need the following facts about Woodin's $\pmax$ from \cite{Wo10}:
\begin{enumerate}
\item $\pmax$ conditions are elements of $H(\aleph_{1})$, and the corresponding order is definable in $H(\aleph_{1})$.

\item $\pmax$ is $\sigma$-closed.

\item\label{cofpres}  Forcing with $\pmax$ over a model of $\AD^{+} + \DC$ preserves the property of having cofinality at least $\omega_{2}$ (this follows from a combination of Theorems 3.45 and 9.32 of \cite{Wo10}, as outlined in Section \ref{threadsec}).

\item If $M\models\ZF + \AD^{+}$ and $G\subseteq \pmax^{M}$ is an $M$-generic filter, then the following hold in $M[G]$:
\begin{enumerate}
\item $2^{\aleph_{0}} = \aleph_{2}$; 
\item $\Theta^{M} = \omega_{3}$; 
\item\label{forsee} $\cP(\omega_{1}) \subseteq L_{\omega_{2}}(\bR)[G]$.
\end{enumerate}

\item Forcing with $\pmax$ over a model of $\ddagger$ 
produces a model of $\ZF$ + $\DC_{\aleph_{2}}$ +  $\MM^{++}(\mathfrak{c})$.
\end{enumerate}

Forcing with $\pmax$ over a model of $\ADR$ cannot wellorder $\cP(\bR)$, since if $\tau$ were a $\pmax$-name for such a wellorder then every element of $\cP(\bR)$ would be ordinal definable from $\tau$ and a real number, which would induce a failure of Uniformization.\footnote{As in \cite[Theorem 6.28]{La17}, which is due to Kechris and Solovay.
(Uniformization is the statement that every subset of $\breals \times \breals$ contains a function with the same domain; this follows easily from the existence of a wellordering of $\breals$, and also from $\ADR$.)}
However,  $\DC_{\aleph_{2}} + 2^{\aleph_{0}} = \aleph_{2}$ implies that $\cP(\bR)$ may be wellordered  by forcing with $\Add(\omega_{3}, 1)$.\footnote{For any ordinal $\gamma$, $\Add(\gamma, 1)$ is the partial order adding a generic subset of $\gamma$ by initial segments.}
Since (by $\DC_{\aleph_{2}}$) $\Add(\omega_{3}, 1)$ does not add subsets of $\omega_{2}$, $\MM^{++}(\mathfrak{c})$ is preserved by forcing with $\Add(\omega_{3}, 1)$ in this context.
This gives Theorem \ref{wsthrm}. 

\begin{theorem}[Woodin {\cite[Theorem 9.39]{Wo10}}]\label{wsthrm} 
Forcing with \[\pmax * \Add(\omega_{3}, 1)\]
over a model of $\ddagger$ 
produces a model of $\ZFC$ + $\MM^{++}(\mathfrak{c})$.
\end{theorem}

It follows from Todor\v{c}evi\'c's Theorem \ref{todthrm} that $\square(\omega_{2}, \omega_{1})$ fails in such a $\pmax$ extension.

The results discussed in \textsection\ref{cmjoinssec} imply that assuming $\ddagger$ + $\Join^{n}_{\lambda}$ (for any $n \in \omega$), we have:
\begin{itemize}
\item $\ADR$ +  $V = L(\cP(\bR))$ + ``$\Theta$ is regular,"
\item the sharp of each subset of $\bR$ exists, and
\item $\c^{+}_{\lambda}$ $\models$ $\ADR$ + ``$\Theta$ is regular."
\end{itemize}
In the first part of this paper we will be forcing over a model of the form $\c^{+}_{\lambda}$ satisfying $\Join^{n}_{\lambda}$, for some ordinal $\lambda$ and an arbitrary positive integer $n$. In this case, $\c^{+}_{\lambda}$ is not a model of ``$V =L(\cP(\bR))$", since, being closed under ordinal definability, it contains the sharp of its version of $\cP(\bR)$ (i.e., $\Delta_{\kappa}$).
So we cannot just cite Theorem \ref{wsthrm} for our main result.
In order to force Choice over a model of the form $\c^{+}_{\lambda}$, it suffices to wellorder $\lambda^{\omega}$ (see the beginning of Part I).
To do this, we will force (over a $\pmax$ extension of $\c^{+}_{\lambda}$, writing $\kappa_{i}$ for $(\kappa^{+i})^{\c^{+}_{\lambda}}$) with the iteration 
\[\Add(\kappa_{0}, 1) * \cdots * \Add(\kappa_{n},1),\]
which we will write as $\oast_{i \leq n}\Add(\kappa_{i},1)$ or $(\oast_{i \leq n}\Add(\kappa^{+i}, 1))^{\c^{+}_{\lambda}[G]}$, where $G \subseteq \pmax$ is $V$-generic, noting that the $\kappa_{i}$'s are preserved as cardinals in $\c^{+}_{\lambda}[G]$ (as in item $(\ref{job4})$ from the beginning of Section \ref{threadsec}).

Theorem \ref{mainthrm}, the main theorem of Part 1, builds on \cite{Wo10} and \cite{CLSSSZ}, especially an argument from the proof of \cite[Theorem 7.3]{CLSSSZ}.

\begin{theorem}\label{mainthrm} 
Suppose that $V\models \ddagger$, $n$ is a positive integer and $\lambda$ is an ordinal for which $\Join^{n}_{\lambda}$ holds. Let $\kappa = \Theta^{\c_{\lambda}}$ and
let $(G, H)$ be a $V$-generic filter for the forcing iteration
\[(\pmax* \oast_{i \leq n}\Add(\k^{+i}, 1))^{\c^{+}_{\l}}.\]
Then
\[\c^+_{\l}[G, H]\models \ZFC + \MM^{++}(\mathfrak{c}) + \forall i \leq (n + 1)\, \neg\square(\omega_{2+i}, \omega).\]
\end{theorem}

\begin{remark}\normalfont
By item (\ref{bowtiesix}) from the definition of $\Join^{n}_{\lambda}$ (see also Remark \ref{obtwo}), the partial orders \[(\pmax* \oast_{i \leq n}\Add(\k^{+i}, 1))^{\c_{\l}}\] and \[(\pmax* \oast_{i \leq n}\Add(\k^{+i}, 1))^{\c^{+}_{\l}}\] are the same, from which it follows that the theorem implies the corresponding version with $\c_{\lambda}$ in place of $\c^{+}_{\lambda}$.
\end{remark}

\bls

\part{Forcing over Nairian models}

Fix $\kappa$, $\lambda$, $G$, $H$ and $K$ as in the statement of Theorem \ref{mainthrm}, and write $H$ as $(H_{0},\ldots,H_{n})$. We let $H \upharpoonright i$ denote $\emptyset$ when $i=0$ and $(H_{0},\ldots,H_{i-1})$ otherwise. We let $\kappa_{i}$ denote $(\kappa^{+i})^{\c^{+}_{\lambda}}$ for each $i \leq n$.
In particular, $\kappa_0 = \kappa$ and $\kappa_{n} = \lambda$. 

In $\c^{+}_{\lambda}$, each set is a surjective image of $\lambda^{\omega} \times \gamma$, for some ordinal $\gamma$. 
Hence $\ZFC$ holds in any forcing extension of $\c^{+}_{\lambda}$ in which $\lambda^{\omega} \cap \c^{+}_{\lambda}$ is wellordered.
Since $\lambda$ has uncountable cofinality in $\c^{+}_{\lambda}$, the forcing 
$\Add(\lambda, 1)^{\c^{+}_{\lambda}[G,H\rest n]}$ (i.e., $\Add(\kappa_{n}, 1)$ in $\c^{+}_{\lambda}[G,H\rest n]$) adds such a wellordering.


\section{Threading coherent sequences}\label{threadsec}

The material in this section is mostly adapted from \cite{CLSSSZ} and reduces the proof of Theorem \ref{mainthrm} to showing the following (see Theorem \ref{theorem:vanilla} and the remarks afterwards):
\begin{enumerate}
   \item\label{job1} $\oast_{i\leq n}\Add(\kappa^{+i}, 1)$ is $\less\omega_{2}$-directed closed in $\c^{+}_{\lambda}[G]$;
  \item\label{job2} $V[G] \models ((\oast_{i\leq n}\Add(\kappa^{+i}, 1))^{\c^{+}_{\lambda}[G]})^{\omega_{1}} \subseteq \c^{+}_{\lambda}[G]$;
  \item\label{job3} $\forall i \leq n$ $\c^{+}_{\lambda}[G, H \upharpoonright i] \models \DC_{\aleph_{2 + i}}$;
  \item\label{job4} $\forall i \leq n$, $(\kappa^{+i})^{\c^{+}_{\lambda}} = (\kappa^{+i})^{\c^{+}_{\lambda}[G,H]}$
\end{enumerate}
Item (\ref{job1}) above follows from the fact that $\c^{+}_{\lambda}[G] \models \DC_{\aleph_{2}}$,\footnote{The case $i=0$ follows from item (\ref{job3}).} which is shown in Lemma \ref{cpgdc2lem}.
Item (\ref{job2}) is Lemma \ref{kaplamcllem}.
Item (\ref{job3}) is Lemma \ref{dcomega_3lem}.
Item (\ref{job4}) follows from item (\ref{job3}), which implies that, for each $i \leq n$, the forcing $\Add(\kappa_{i}, 1)$ over $\c^{+}_{\lambda}[G,H \restriction i]$ doesn't add bounded subsets of $\kappa_{i}$, and Lemma \ref{strongregghlem}, which shows that each $\Add(\kappa_{i}, 1)$ preserves the regularity of $\kappa_{j}$ for all $j \in (i,n]$.

To apply Todor\v{c}evi\'c's Theorem \ref{todthrm} to show that $\square(\omega_{3},\omega),\ldots,\square(\omega_{3+n}, \omega)$ all fail in $\c^{+}_{\lambda}[G,H]$, we need to show that $\kappa_{1},\ldots,\kappa_{n}$ from Theorem \ref{mainthrm} have cofinality $\omega_{2}$ in
$V[G]$.
(Recall that they are less than $\Theta^{V}$, which is $\omega_{3}^{V[G]}$.)
This is a consequence of covering results by Woodin, which we briefly review.
We refer the reader to \cite[Definition 3.30]{Wo10} for the definition of {\em $A$-iterability}.
Given $X\prec H(\omega_2)$, $M_X$ denotes its transitive
collapse.

\begin{theorem}[Woodin {\cite[Theorem 3.45]{Wo10}}] \label{thm:woodin} \normalfont 
Suppose that $M$ is a proper class inner model containing $\bR$ and satisfying $\AD+\DC$.
Suppose that for any $A\in{\mathcal P}(\bR)\cap M$, the set
$$ \{X\prec H(\omega_2)\mid \mbox{\rm$X$ is countable, and $M_X$ is $A$-iterable}\} $$
is stationary. Let $X$ in $V$ be a bounded subset of $\Theta^M$ of size $\omega_1$. Then there is a set
$Y\in M$, of size $\aleph_1$ in $M$, such that $X\subseteq Y$.
\end{theorem}

We apply Theorem \ref{thm:woodin} in the proof of Lemma \ref{lemma:coflemma} with $M$ as a model of the form $L(A,\bR)$ for some $A \subseteq \breals$, and the $V$ of Theorem \ref{thm:woodin} as a $\pmax$ extension of $M$.

\begin{lemma} \label{lemma:coflemma}\normalfont 
Suppose that $M$ is a model of $\ZF + \AD^{+}$ and $\gamma$ is an ordinal of cofinality at least
$\omega_{2}$ in $M$.
Let $G_0\subset \pmax$ be an $M$-generic filter. Then $\gamma$ has cofinality at least $\omega_{2}$ in
$M[G_0]$.
\end{lemma}

\begin{proof}
Suppose first that $\gamma < \Theta^{M}$. Let $X$ be a subset of $\gamma$ of cardinality $\aleph_{1}$ in $M[G_0]$,
and let $A \in \cP(\breals) \cap M$ have Wadge rank at least $\gamma$. Since $|\gamma| \leq 2^{\aleph_{0}}$ in $M[G_{0}]$ and $\cP(\omega_{1})^{M[G_0]}$ is contained in $L(A, \mathbb{R})[G]$,
$X$ is in $L(A, \mathbb{R})[G]$. By \cite[Theorem 9.32]{Wo10}, the hypotheses of Theorem \ref{thm:woodin} are satisfied with $L(A, \bR)$ as $M$ and $L(A, \bR)[G]$ as $V$.
Applying Theorem \ref{thm:woodin}, we have that $X$ is a subset of an element
of $L(A, \mathbb{R})$ of cardinality $\aleph_{1}$ in $L(A, \mathbb{R})$.

The lemma follows immediately from the previous paragraph for $\gamma$ of cofinality less than $\Theta^{M}$ in $M$.
If $\gamma \geq \Theta$ is regular in $M$, there is no cofinal function from $\breals$ to $\gamma$ in $M$, so there is no such
function in $M[G_0]$, either.
The theorem then follows for arbitrary $\gamma$.
\end{proof}

A simpler proof of a weaker version of Lemma \ref{lemma:coflemma} assuming $\ADR$ plus the regularity of $\Theta$ (which suffices for the results of this paper) is given in \cite{Larson:two}.

In conjunction with the facts mentioned at the beginning of this section, Theorem \ref{theorem:vanilla}, with $M_{1}$ as $V$, $M_{0}$ as $\c^{+}_{\lambda}$, $\gamma$ as any member of $\{\kappa_{0},\ldots,\kappa_{n}\}$ and $\bbQ$ as $(\oast_{i\leq n}\Add(\kappa^{+i},1))^{\c^{+}_{\lambda}[G]}$, completes the proof of Theorem \ref{mainthrm}.
The theorem and its proof are from \cite{CLSSSZ}, except that the specific partial order used in \cite{CLSSSZ} has been replaced with a more general class of partial orders.

\begin{theorem} \label{theorem:vanilla} \normalfont 
Suppose that $M_{1}$ is a model of $\ddagger$, and that for some set $X \in M_{1}$ containing
$\breals \cap M_{1}$, $M_{0} = \H^{M_{1}}_{X}$. Suppose also that  $\Theta^{M_{0}} < \Theta^{M_{1}}$ and that $\gamma \in [\Theta^{M_{0}}, \Theta^{M_{1}})$ has cofinality at least $\omega_{2}$ in $M_{1}$.
Let $G_0 \subset \pmax$ be $M_{1}$-generic, and let $I \subset \bbQ$ be
$M_{1}[G_0]$-generic, for some partial order $\bbQ \in M_{0}[G_0]$ which, in $M_{1}[G_0]$, is
$<\omega_{2}$-directed closed and of cardinality at most $\mathfrak{c}$.
Then $\square(\gamma, \omega)$ fails in $M_{0}[G_0][I]$.
\end{theorem}

\begin{proof}
Suppose that $\tau$ is a $\pmax * \dot{\bbQ}$-name in $M_{0}$ for a
$\square(\gamma,\omega)$-sequence. We may assume that the realization of $\tau$ comes with
an indexing of each member of the sequence in order type at most $\omega$. In $M_{1}$, $\tau$ is
ordinal definable from some $S \in X$.

By \cite[Theorems 9.35 and 9.39]{Wo10}, $\DC_{\aleph_{2}}$ and $\MM^{++}(\mathfrak{c})$ hold in $M_{1}[G_0]$.
By Lemma \ref{lemma:coflemma},
$\gamma$ has cofinality $\omega_{2}$ in $M_{1}[G_0]$.
Forcing with ${<}\omega_{2}$-directed closed partial orders
of size at most $\mathfrak{c}$ preserves both $\MM^{++}(\mathfrak{c})$ (see \cite{Larson:separating}) and the cofinality of $\gamma$.
Thus $\DC_{\aleph_{1}}$ and $\MM^{++} (\mathfrak{c})$ hold in the $\dot{\bbQ}_{G_0}$-extension of $M_{1}[G_0]$,
and in this extension, by Todor\v{c}evi\'c's Theorem \ref{todthrm}, every potential $\square(\gamma, \omega)$-sequence is
threaded.

Let $\mathcal{C} = \langle \mathcal{C}_{\alpha} : \alpha < \gamma \rangle$ be the realization of
$\tau$ in the $\dot{\bbQ}_{G}$-extension of $M_{1}[G_0]$.
Since $\gamma$ has cofinality $\omega_{2}$ in this
extension (which satisfies $\DC_{\aleph_{1}}$), $\mathcal{C}$ has at most $\omega$ many
threads, since otherwise one could find a $\mathcal{C}_{\alpha}$ in the sequence with uncountably
many members.
Therefore, some member of some $\mathcal{C}_{\alpha}$ in the realization of $\tau$
will be extended by a unique thread through the sequence, and since the realization of $\tau$ indexes
each $\mathcal{C}_{\alpha}$ in order type at most $\omega$, there is in $M_{1}$ a name, ordinal
definable from $S$, for a thread through the realization of $\tau$.
This name is then a member of
$M_{0} = \H^{M_{1}}_{X}$.
\end{proof}

\subsection{Proof of Theorem \ref{mainthrm}}\label{proofmainthrm}
We now complete the proof of Theorem \ref{mainthrm}, assuming the statements (\ref{job1})-(\ref{job4}) from the beginning of this section. We apply Theorem \ref{theorem:vanilla} with $V$ as $M_{1}$ and $X$ as $\lambda^{\omega}$. Then the model $M_{0}$ from Theorem \ref{theorem:vanilla} is $\c^{+}_{\lambda}$, and $\kappa = \Theta^{M_{0}} < \Theta^{M_{1}}$ by the hypotheses of Theorem \ref{mainthrm}. Let $G = G_{0}$ be $\pmax$-generic over $V$, and let $\bbQ$ be the partial order $\oast_{i\leq n}\Add(\kappa^{+i},1)$ as defined in $\c^{+}_{\lambda}[G]$. 
By item (1) from the beginning of this section, $\bbQ$ is $\less\omega_{2}$-directed closed in $\c^{+}_{\lambda}$. By item (2), every $\omega_{1}$-sequence from $\bbQ$ in $V[G]$ is in $\c^{+}_{\lambda}[G]$. It follows then that $\bbQ$ is $\less\omega_{2}$-directed closed in $V[G]$ as well. 
Letting $H$ be $V[G]$-generic for $\bbQ$, it follows from Theorem \ref{theorem:vanilla} that, for each $\gamma \in [\kappa, \Theta^{V})$ having cofinality at least $\omega_{2}$ in $V$, $\square(\gamma, \omega)$ fails in $\c^{+}_{\lambda}[G, H]$.
By Lemma \ref{lemma:coflemma}, and the assumption of $\Join^{n}_{\lambda}$, each of the ordinals $(\kappa^{+i})^{\c^{+}_{\lambda}}$ for $i \leq n$ satisfies these conditions on $\gamma$. By item (\ref{job4}), the cardinals $\kappa_{0},\ldots,\kappa_{n}$ are the cardinals $\aleph_{3},\ldots,\aleph_{3+n}$ in $\c^{+}_{\lambda}[G,H]$. 
The remarks at the beginning of Part I also show that $\c^{+}_{\lambda}[G, H] \models \ZFC$.


Since $\pmax \subseteq H(\aleph_{1})$, and $\kappa$ is both regular and equal to $\Theta^{\c^{+}_{\lambda}}$, \[(\pmax * \Add(\kappa, 1))^{L(\Delta_{\k})} = (\pmax * \Add(\kappa, 1))^{\c^{+}_{\lambda}}.\] 
Since $\L(\Delta_\kappa) \models \ddagger$, Theorem \ref{wsthrm} implies that $L(\Delta_{\kappa})[G,H_{0}]$ satisfies $\MM^{++}(\mathfrak{c})$ and hence $\neg \square(\omega_{2}, \omega)$.
Then $\c^{+}_{\lambda}[G, H_{0}]$ does as well, since these two models have the same $\cP(\omega_{2})$.
Item (\ref{job3}) implies that $\DC_{\aleph_{3+i}}$ holds in $\c^{+}_{\lambda}[G, H_{0},\ldots,H_{i}]$, for each $i < n$, so $\c^{+}_{\lambda}[G, H_{0}]$ and $\c^{+}_{\lambda}[G, H]$ have the same subsets of $\omega_{2}$.
Thus $\c^{+}_{\lambda}[G, H]\models\MM^{++}(\mathfrak{c})$ as well. \hfill $\qed$

\section{Proving $\DC_{\aleph_{m}}$}\label{dcsec}

One of our four remaining tasks is showing that
\[\c^{+}_{\lambda}[G, H \upharpoonright i] \models \DC_{\aleph_{2 + i}}\] for all $i \leq n$.
\textsection\ref{dcreducesec} reduces each of these to the case of relations on $\lambda^{\omega}$. 
In \textsection\ref{dcproveoutlinessec} we 
prove a lemma which, together with the results of \textsection\ref{strongregsec}, reduces this further to the case of relations on the sets $\kappa_{i}^{\omega}$ for $i < n$. 

That $\c^{+}_{\lambda}[G] \models \DC_{\aleph_{2}}$ is proved in Lemma \ref{cpgdc2lem}, using Lemma \ref{pmaxdc}. That each model $\c^{+}_{\lambda}[G, H \rest i]$ ($0 < i < n$) satisfies $\DC_{\omega_{2+i}} \rest \lambda^{\omega}$ is proved in \textsection\ref{dc3sec}.


\subsection{Reducing to $\DC_{\aleph_{m}}\rest \lambda^{\omega}$}\label{dcreducesec}

As we have defined it, $\DC_{\aleph_{m}}$ implies $\DC_{\aleph_{k}}$ for all $k \leq m$. Since $\ddagger$ implies $\DC$, Lemma \ref{dcreducelem} shows (in the case $m=0$) that $\DC$ holds in $\c^{+}_{\lambda}$.

\begin{lemma}\label{dcreducelem}\normalfont  Let $\bbP$ be a partial order in $\c^{+}_{\lambda}$, and let $I \subseteq \bbP$ be a $\c^{+}_{\lambda}$-generic filter.
Let $m<\omega$ be such that $\DC_{\aleph_{k}}$ holds in $\c^{+}_{\lambda}[I]$ for all $k < m$.
Suppose that, in $\c^{+}_{\lambda}[I]$, every $<\omega_{m}$-full tree  on $\lambda^{\omega}$ of height $\omega_{m}$ has a cofinal branch.
Then $\DC_{\aleph_{m}}$ holds in $\c^{+}_{\lambda}[I]$.
\end{lemma}

\begin{proof}
Fix a $<\omega_{m}$-full tree $T$ in $\c^{+}_{\lambda}[I]$. Fix an ordinal $\gamma$ such that every node of $T$ is the realization of a $\bbP$-name which is ordinal definable in $V_{\gamma}$ from some element of $\lambda^{\omega}$.
Given $(n,\delta, x) \in \omega \times \gamma \times \lambda^{\omega}$, let \[t_{n,\delta, x}\] be the set defined in $V_{\gamma}$ from $\delta$ and $x$ by the formula with G\"{o}del number $n$.

Let $T'$ be the tree of sequences $\langle x_{\alpha} : \alpha < \beta\rangle$ (for some $\beta < \omega_{m}$) for which there exists a sequence \[\langle y_{\alpha} : \alpha < \beta \rangle\] such that, for each $\eta < \beta$, \[y_{\eta} = t_{n,\delta, x_{\eta}, I},\] where $(n, \delta) \in (\omega, \gamma)$ is minimal such that  $t_{n,\delta,x_{\eta}}$ is a $\bbP$-name and
\[\langle y_{\alpha} : \alpha  < \eta \rangle^{\frown} \langle t_{n,\delta,x_{\eta},I}\rangle \in T.\]

Then $T'$ is also $<\omega_{m}$-full, and an $\omega_{m}$-chain through $T'$ induces one through $T$.
\end{proof}

\subsection{Proving $\DC_{\aleph_{2}}| \lambda^{\omega}$}\label{dcproveoutlinessec}

The following reflection argument is useful for establishing forms of $\DC$ in the models we consider. 

\begin{lemma}[$\ZF$]\label{dcsteplem} Suppose that $\tau$ and $\rho$ are infinite cardinals such that
there is no cofinal function $f \colon (\rho^{\omega})^{\tau} \to \rho^{+}$. 
If $\DC_{\tau^{+}} \restriction \rho^{\omega}$ holds, then 
$\DC_{\tau^{+}} \restriction (\rho^{+})^{\omega}$ also holds. 
\end{lemma}

\begin{proof}
	Let $T$ be a $\leq\tau$-full tree on $(\rho^{+})^{\omega}$ of height $\tau^{+}$. The hypotheses imply that there is a $\gamma \in [\rho, \rho^{+})$ such that every node of $T \cap (\gamma^{\omega})^{<\tau^{+}}$ has a proper extension in  $T \cap (\gamma^{\omega})^{<\tau^{+}}$. Let $b \colon \rho \to \gamma$ be a bijection, and let $T_{b}$ be the set of $\langle a_{\alpha} : \alpha < \beta\rangle \in (\rho^{\omega})^{\less\tau^{+}}$ for which $\langle b \circ a_{\alpha} : \alpha < \beta \rangle \in T$. Since $\DC_{\tau^{+}}\rest \rho^{\omega}$ holds, there is a cofinal branch $\langle a_{\alpha} : \alpha < \tau^{+}\rangle$ through $T_{b}$. Then $\langle b \circ a_{\alpha}
   : \alpha < \tau^{+} \rangle$ is a cofinal branch through $T$.
\end{proof}

Lemma \ref{dcsteplem} naturally enables induction arguments. 
We now prove the base case for one such argument. Recall that $\Theta^{M} = \omega_{3}^{M[G]}$ in the context of the lemma. 

\begin{lemma}\label{pmaxdc}
Suppose that $M$ is a model of $\AD^{+}$ containing $\bR$, and that $G \subseteq \pmax$ is $M$-generic. Then for all $\gamma < \Theta^{M}$, $M[G] \models \DC_{\infty} \rest \gamma^{\omega}$. If in addition $M \models ``\Theta$ is regular", then 
$M[G] \models \DC_{\aleph_{2}} \rest \omega_{3}^{\omega}$. 
\end{lemma}

\begin{proof}
The first conclusion of the lemma follows from the fact that forcing with $\pmax$ wellorders $\bR$. 
Suppose then that $M \models ``\Theta$ is regular". 
Then in $M[G]$ we have
\begin{itemize}
	\item $2^{\aleph_{0}} = \aleph_{2}$, 
	\item $\Theta^{M} = \omega_{3}$ is regular and 
	\item $\aleph_{2}^{\aleph_{1}} = \aleph_{2}$,
\end{itemize}
from which it follows that there is no cofinal function in $M[G]$ from $(\gamma^{\omega})^{\beta}$ to $\Theta^{M}$, for any $\gamma < \Theta^{M}$ and $\beta < \omega_{2}^{M}$. A reflection argument as in the proof of Lemma \ref{dcsteplem} then gives the second conclusion of the lemma. 
\end{proof}

In \textsection\ref{strongregsec} we establish the first hypothesis of Lemma \ref{dcsteplem} in $\c^{+}_{\lambda}[G]$ for $\tau = \omega_{1}$ and $\rho = \kappa_{i}$, for all $i < n$.
In conjunction with the two lemmas in this section, this gives that $\c^{+}_{\lambda}[G] \models \DC_{\aleph_{2}}$. 

\footnote{I don't think we need the rest of this, but leaving it here for easy access if needed. To show that $\DC_{\aleph_{2}} \rest \lambda^{\omega}$ holds in $\c^{+}_{\lambda}[G]$, we show that the following statements hold in $\c^{+}_{\lambda}[G]$:

\begin{itemize}
\item there is no cofinal map from $\omega_{2}$ to $\lambda$;

\item there is no cofinal map from $(\gamma^{\omega})^{\beta}$ to $\lambda$, for any $\gamma < \lambda$ and $\beta < \omega_{2}$ (it suffices to show this for $\gamma = \kappa$ and $\beta = \omega_{1}$).
\end{itemize}

The first of these follows from Lemma \ref{strongreglem} with $b$ as $\omega_{2} \times \pmax$. The second is shown in the proof of Lemma \ref{cpgdc2lem}, whose statement is just the desired statement that $\c^{+}_{\lambda}[G] \models \DC_{\aleph_{2}}$. These two facts imply that for every cardinal $\delta \leq \aleph_{2}$ and each $\delta$-full relation $R$ on $\lambda^{\omega}$ in $\c^{+}_{\lambda}[G]$, there exists a $\gamma < \lambda$ such that $R \cap \gamma^{\omega}$ is also $\delta$-full. Since $\lambda = \kappa^{+}$, it suffices then (once we have established the two facts above) to consider trees on $\kappa^{\omega}$.
To show that, in $\c^{+}_{\lambda}[G]$, every $\omega_{2}$-full tree on $\kappa^{\omega}$ has a cofinal branch, we use the fact (which follows from standard $\pmax$ arguments) that the following statements hold in $\c^{+}_{\lambda}[G]$:

\begin{itemize}
\item there is no cofinal map from $\omega_{2}$ to $\kappa$ (because $\kappa = \Theta^{\c^{+}_{\lambda}}$ is regular in $\c^{+}_{\lambda}$ and $\pmax \subseteq H(\aleph_{1})$);

\item there is no cofinal map from $(\gamma^{\omega})^{\beta}$ to $\kappa$, for any $\gamma < \kappa$ and $\beta < \omega_{2}$ (because $\kappa = \omega_{3}^{\c^{+}_{\lambda}[G]}$ and $\aleph_{2}^{\aleph_{1}} = \aleph_{2}$).
\end{itemize}
These facts imply that it suffices to consider $\omega_{2}$-full trees on $\gamma^{\omega}$ for any $\gamma < \kappa$.
Since each such $\gamma^{\omega}$ is a surjective image of the wellordered set $\breals$ in $\c^{+}_{\lambda}[G]$,
$\c^{+}_{\lambda}[G]$ satisfies the statement that each such tree has a cofinal branch.}

\section{Strong regularity of $\lambda$}\label{strongregsec}
As above, let $\kappa_{i}$ denote $(\kappa^{+i})^{\c^{+}_{\lambda}}$.
A cardinal $\rho$ is \emph{strongly regular} in $\c^{+}_{\lambda}$ if whenever $b$ is in $\c^{-}_{\rho}$ and $f \colon b \to \rho$ is in $\c^{+}_{\lambda}$ the range of $f$ is bounded in $\rho$, i.e. there is a $\gamma < \rho$ such that $f[b] \subseteq \gamma$.  
In this section we show that each of the cardinals $\k_{i}$, $i \leq n$, is strongly regular in  $\c^{+}_{\lambda}$.
We then derive several consequences, including the fact that $\DC_{\aleph_{2}}$ holds in $\c^{+}_{\lambda}[G]$.

\begin{lemma}\label{strongreglem}
	For each $i \leq n$, $\kappa_{i}$ is strongly regular in $\c^{+}_{\lambda}$. 
\end{lemma}

\begin{proof}
We work in $\c^{+}_{\lambda}$ and proceed by induction on $i$.
When $i=0$, the lemma follows from 
the facts that $\Theta$ is regular and each member of $\c^{-}_{\Theta}$ is a surjective image of the reals. 

Fix $i < n$ for which the lemma holds, and fix $b \in \c^{-}_{k_{i+1}}$ and $f \colon b \to \kappa_{i+1}$. 
Let $\beta<\kappa_{i+1}$ be such that $b\in \c^{-}_{\beta}$.
Since $\kappa_{i+1} = \kappa_{i}^{+}$, there exists a surjection $h\colon \kappa_{i}\to \beta$.
Let $B$ be the set of $y \in \kappa_{i}^\omega$ such that $b$ has a member definable in $\c^{-}_{\beta}$ from
$h \circ y$ and $\mH \rest \beta$.
Then $B$ induces a surjection $g \colon \kappa_{i}^{\omega} \to b$.

Since $\kappa_{i}$ is strongly regular, the ordertype of $f[g[\alpha^{\omega}]]$ is less than $\kappa_{i}$ for each $\alpha < \kappa_{i}$.
Since $\kappa_{i+1}$ is regular, $f[g[\alpha^{\omega}]]$ is a bounded subset of $\kappa_{i+1}$, for each $\alpha < \kappa_{i}$.
Again applying the regularity of $\kappa_{i+1}$, $f[b] = f[g[\kappa_{i}^{\omega}]]$ is bounded in $\kappa_{i+1}$.
\end{proof}

Since $\pmax$ is in $\c^{-}_{\kappa}$, Lemma \ref{strongreglem} has the following immediate consequence.
In particular, each $\kappa_{i}$, $i \leq n$, remains regular in $\c^{+}_{\lambda}[G]$.

\begin{lemma}\label{strongreglemcon}
For each $i \leq n$, each $b \in \c^{-}_{\kappa_{i}}[G]$ and each function $f \colon b \to \kappa_{i}$ in $\c^{+}_{\lambda}[G]$, the range of $f$ is bounded in $\kappa_{i}$. 
\end{lemma}


The following lemma will be used to prove that $\c^{+}_{\lambda}[G] \models \DC_{\aleph_{2}}$.






\begin{lemma}\label{cmpscllem}\normalfont 
For all $i\leq n$ and all $b\in \c^-_{\kappa_{i}}$, $\cP(b) \cap \c^{+}_{\lambda} \subseteq \c^{-}_{\kappa_{i}}$.
\end{lemma}

\begin{proof}
Again, we work in $\c^{+}_{\lambda}$.
When $i=0$, the lemma follows from Proposition \ref{Thetarealsrem} and the fact that each member of $\c^{-}_{\kappa}$ is a surjective image of the reals in $\c^{-}_{\kappa}$.

Fixing $i < n$ for which the lemma holds, we show that it holds for $i+1$.
Fix $b \in \c^{-}_{\kappa_{i+1}}$ and $\beta<\kappa_{i+1}$ such that
$b\in \c^{-}_{\beta}$.
Let
$h \colon \kappa_{i} \to \beta $ be a surjection in $\c_{\kappa_{i+1}}$.
Fix $a\in \cP(b)\cap \c^+_{\lambda}$, and let $B_{a}$ be the set of $(x,n) \in \kappa^\omega \times \omega$ such that some member of $a$ is definable over $\c^{-}_{\beta}$ from $h \circ x$ and $\mH \rest \beta$ via the formula with G\"{o}del number $n$.
Then $B_{a} \in \c^{+}_{\lambda}$.
Since $\cP(\kappa_{i}^{\omega}) \cap \c^{+}_{\lambda} \subseteq \c^{-}_{\kappa_{i+1}}$, we have $B_{a}\in \c^{-}_{\kappa_{i+1}}$, so $a\in \c^{-}_{\kappa_{i+1}}$. 
\end{proof}

We have enough in place to prove $\DC_{\aleph_{2}}$.

\begin{lemma}\label{cpgdc2lem}\normalfont  $\c^{+}_{\lambda}[G] \models \DC_{\aleph_{2}}$.
\end{lemma}

\begin{proof}
By Lemmas \ref{dcsteplem} and \ref{pmaxdc},  it suffices to shows that cofinal map from $\kappa_{i}^{\omega_{1}}$ to $\kappa_{i+1}$ in $\c^{+}_{\lambda}[G]$ for any $i < n$.
Fix such an $i$ and an $f \colon \kappa_{i}^{\omega_{1}} \to \kappa_{i+1}$ in $\c^{+}_{\lambda}[G]$.
By Lemma \ref{strongreglemcon}, each element of $\kappa_{i}^{\omega_{1}}$ in $\c^{+}_{\lambda}[G]$ has range bounded below $\kappa_{i+1}$.
By Lemma \ref{strongreglemcon} applied to $\kappa_{i+1}$, it suffices to show that for each $\gamma < \kappa_{i}$, $f[\gamma^{\omega_{1}}]$ is bounded below $\kappa_{i+1}$.
Fix such a $\gamma$.
Each element of $\gamma^{\omega_{1}}$ is the realization of a $\pmax$ name which coded by a subset of $\gamma \times \omega_{1} \times \pmax$, and which is therefore in $\c^{-}_{\kappa_{i}}$, by Lemma  \ref{cmpscllem}.
It follows that $\gamma^{\omega_{1}} \cap \c^{+}_{\lambda}$ is an element of $\c^{-}_{\kappa_{i+1}}$.
One more application of Lemma \ref{strongreglemcon} then completes the proof.
\end{proof}

The rest of this section is devoted to proving Lemma \ref{strongregghlem}, which, as discussed at the beginning of \textsection\ref{threadsec}, proves item (\ref{job4}) with the help of item (\ref{job3}).

\begin{remark}\label{obone}\normalfont By Proposition \ref{Thetarealsrem}, each set of reals in $\c^{+}_{\lambda}$ is in $\c^{-}_{\kappa}$.
Since $\kappa$ is regular, each condition in $(\pmax * \Add(\kappa, 1))^{\c^{+}_{\lambda}}$ is coded by a set of reals and is therefore in $\c^{-}_{\kappa}$.
It follows that $(\pmax * \Add(\kappa, 1))^{\c^{+}_{\lambda}}$ is in $\c^{-}_{\kappa_{1}}$. 
\end{remark}

\begin{remark}\label{obtwo}\normalfont 

Since $\kappa_{1}$ is regular, each $(\pmax * \Add(\kappa, 1))^{\c^{+}_{\lambda}}$-name for a condition in $\Add(\kappa_{1}, 1)$ is a subset of $(\pmax * \Add(\kappa, 1))^{\c^{+}_{\lambda}} \times \gamma$ for some $\gamma < \kappa_{1}$, and is therefore in $\c^{-}_{\kappa_{1}}$, by Lemma \ref{cmpscllem}.
Thus each condition in 
$(\pmax * \oast_{i \leq 1}\Add(\kappa^{+i}, 1))^{\c^{+}_{\lambda}}$ is in $\c^{-}_{\kappa_{1}}$, and therefore $(\pmax * \oast_{i \leq 1}\Add(\kappa^{+i}, 1))^{\c^{+}_{\lambda}}$ is in $\c^{-}_{\kappa_{2}}$.
The same arguments show that, for each $m \leq n$, each condition in \[(\pmax * \oast_{i \leq m}\Add(\kappa^{+i}, 1))^{\c^{+}_{\lambda}}\] is in $\c^{-}_{\kappa_{m}}$, and that, for each $m < n$, $(\pmax * \oast_{i \leq m}\Add(\kappa^{+i}, 1))^{\c^{+}_{\lambda}}$ is in $\c^{-}_{\kappa_{m+1}}$. 
\end{remark}

Remarks \ref{obone} and \ref{obtwo} give the following lemma, which implies that each $\kappa_{m+1}$ remains regular in $\c^{+}_{\lambda}[G,H_{0},\ldots,H_{m}]$. 


\begin{lemma}\label{strongregghlem} \normalfont 
	The following hold for all $m < p \leq n$ and all 
$b\in \c^{-}_{\kappa_{p}}[G,H_{0},\ldots,H_{m}]$.
\begin{enumerate}
\item $\cP(b) \cap \c^{+}_{\lambda} \subseteq \c^{-}_{\kappa_{p}}$; 
\item For each \[f \colon b \to \kappa_{p}\] in $\c_{\lambda}^{+}[G,H_{0},\ldots,H_{m}]$,  there exists a $\gamma<\kappa_{p}$ such that $f[b]\subseteq \gamma$.
\end{enumerate}
\end{lemma}

\begin{proof}
Fix a name $\tau$ in $\c^{-}_{\kappa_{p}}$ such that $b$ is the realization of $\tau$. For part (1), every subset of $b$ in $\c^{+}_{\lambda}$ is the realization of a name which can be coded by a subset of \[\tc(\{\tau\}) \times (\pmax * \oast_{i\leq m}\Add(\kappa^{+i}, 1)).\]
Part (1) then follows from Lemma \ref{strongreglemcon} and Remark \ref{obtwo}. Part (2) is similar, using Lemma \ref{strongreglem}. 
\end{proof} 




\section{$\omega_{1}$-closure in $V[G]$}

In this section we show that, in $V[G]$, $\c^{+}_{\lambda}[G]$ is closed under $\lambda$-sequences from
\[(\oast_{i \leq n}\Add(\kappa^{+i},1))^{\c^{+}_{\lambda}[G]}.\]
This is Lemma \ref{kaplamcllem} below, which follows from Lemma \ref{omega1 functions}.

\begin{lemma}\label{omega1 functions}\normalfont  In $V[G]$, for each $i \leq n$ and $b\in \c^-_{\kappa_{i}}[G]$,
$b^{\omega_1}\in \c^{-}_{\kappa_{i}}[G]$.
\end{lemma}

\begin{proof}
We prove the lemma first in the case where $b \in \c^{-}_{\kappa_{0}}[G]$, and then show that if it holds for $b \in \c^{-}_{\kappa_{i}}[G]$, for some $i < n$, then it holds for $b \in \c^{-}_{\kappa_{i+1}}$.

If $b \in \c^{-}_{\kappa_{0}}[G]$, then there exists a $\gamma < \kappa_{0}$ such that $b$ and all of its members are in $\c^{-}_{\gamma}[G]$. Since $\c^{-}_{\gamma}$ is a surjective image of $\bR$ in $\c^{-}_{\kappa_{0}}$, and $(\bR^{\omega_{1}})^{V[G]} \subseteq \c^{-}_{\gamma}[G]$ (by item (\ref{forsee}) from the list at the beginning of Section \ref{pmaxssec}), it follows that $(b^{\omega_{1}})^{V[G]} \subseteq \c^{-}_{\kappa_{0}}[G]$.

Fix $i < n$ for which the lemma holds. 
Since $\kappa_{i} < \Theta$, $\c^{-}_{\kappa_{i+1}}$ is a surjective image of $\breals$ in $V$.
Let $U \subseteq \breals$ be such that $\c^{-}_{\kappa_{i+1}}$ is a surjective image of $\breals$ in $L(U, \bR)$.
Since $\breals$ is wellordered in $L(U, \bR)[G]$, there exists in $L(U, \bR)[G]$ a function picking for each
$x \in \c^{-}_{\kappa_{i+1}}[G]$ a $\pmax$-name $\tau_{x} \in \c^{-}_{\kappa_{i+1}}$ such that $\tau_{x,G} = x$.
Since $\cP(\omega_{1}) \cap V[G] \subseteq L(\bR)[G]$, 
\[(\c^{-}_{\kappa_{i+1}}[G])^{\omega_{1}} \cap V[G]
\subseteq L(U, \bR)[G].\]


We work in $L(U, \bR)[G]$, which satisfies Choice.
Fix $b\in \c^-_{\kappa_{i+1}}[G]$, and let $\beta < \kappa_{i+1}$ be
such that $\Delta_{\kappa}, \tau_{b} \in \c^{-}_{\beta}$.
It follows that every member of $b$ is the realization of a name in $\c^{-}_{\beta}$.
We show first that $b^{\omega_1}\subseteq \c^{-}_{\kappa_{i+1}}[G]$.

Fix $f \in b^{\omega_{1}}$. Since Choice holds, there is an \[h_f \in (\c^{-}_{\beta})^{\omega_{1}}\] such that, for every $\alpha<\omega_1$,
$h_f(\alpha)$ is a $\pmax$-name in $\c^{-}_{\beta}$ such that $h_{f}(\alpha)_{G} = f(\alpha)$.
Fix  a function $c_{f} \colon \omega_{1} \to \omega$  and a sequence $\langle B_\alpha:\alpha<\omega_1\rangle$ such that
\begin{enumerate}
    \item[(1.1)] each
$B_\alpha$ is a nonempty subset of $\beta^{\omega}$, and

    \item[(1.2)] each $h_f(\alpha)$ is defined in
$\c^{-}_{\beta}$ from $\mH \rest \beta$ and each member of the corresponding $B_\alpha$,
via the formula with G\"{o}del number $c_{f}(\alpha)$.
\end{enumerate}
Since $\cP(\omega_{1}) \subseteq L_{\omega_{2}}(\bR)[G]$, $c_{f} \in \c_{\kappa_{i+1}}[G]$.

Let $h \colon \kappa_{i} \to \beta$ be a surjection in $\c^{-}_{\kappa_{i+1}}$ 
(which exists by Lemma \ref{cmpscllem}), and, for each $\alpha < \omega_1$, let
\[B'_\alpha = \{ x \in \kappa_{i}^{\omega} : h \circ x \in B_\alpha\}.\]
As there is no cofinal function from $\omega_{1}$ to $\kappa_{i}$ in 
$\c^{+}_{\lambda}[G]$, there is a $\gamma < \kappa_{i}$ such that $B'_\alpha \cap \gamma^{\omega}$ is nonempty
for each $\alpha < \omega_{1}$. Since $\gamma^{\omega_{1}} \subseteq \c^{-}_{\kappa_{i}}$ by the induction hypothesis, 
there is a sequence $\langle x_{\alpha} : \alpha < \omega_{1} \rangle \in (\gamma^{\omega})^{\omega_{1}} \cap \c^{-}_{\kappa_{1}}[G]$ such that each $x_{\alpha}$ is in the corresponding $B'_{\alpha}$. The sequence $\langle h \circ x_{\alpha} : \alpha < \omega_{1} \rangle$ and the function $c_{f}$ are both then in $\c^{-}_{\kappa_{i+1}}$. Along with $G$ and $\beta$ they can be used to define $f$
in $\c^{-}_{\kappa_{i+1}}$. 


Suppose now that $b^{\omega_1}\not \subseteq \c^{-}_{\alpha}[G]$ for any $\alpha < \kappa_{i+1}$. We then have a function $g\colon b^{\omega_1}\to \kappa_{i+1}$ that is unbounded in $\kappa_{i+1}$ with
$g\in \c_{\kappa_{i+1}}[G]$.
Using $\beta$ and $h$ as above, we can associate to each $f \in b^{\omega_{1}}$ the set of pairs $(c,x) \in \omega^{\omega_{1}} \times \kappa_{i}^{\omega}$ for which each value $f(\alpha)$ is defined in $\c^{-}_{\beta}$ from $\mH \rest \beta$ and $h \circ x$ via the formula with G\"{o}del number $c(\alpha)$. This induces a cofinal map from $\cP(\omega_{1}) \times \kappa_{i}^{\omega}$ to $\kappa_{i+1}$. Since each member of $\kappa_{i}^{\omega}$ has range bounded below $\kappa_{i}$ (by the countable closure of $\pmax$), and since there is no cofinal function from $\kappa_{i}$ to $\kappa_{i+1}$ in $\c^{+}_{\kappa_{i+1}}[G]$ by Lemma \ref{strongreglemcon}, this induces a cofinal function from $\cP(\omega_{1}) \times \gamma^{\omega}$ to $\kappa_{i+1}$ in $\c^{+}_{\kappa_{i+1}}$, for some $\gamma < \kappa_{i}$. 
By the induction hypothesis, $\cP(\omega_{1}) \times \gamma^{\omega}$ is in $\c^{-}_{\kappa_{i}}[G]$, contradicting Lemma \ref{strongreglemcon}.
\end{proof}



\begin{lemma}\label{kaplamcllem} \normalfont $V[G] \models ((\oast_{i\leq n}\Add(\kappa^{+i}, 1))^{\c^{+}_{\lambda}[G]})^{\omega_{1}} \subseteq \c^{+}_{\lambda}[G]$
\end{lemma}

\begin{proof}
As noted in Remark \ref{obtwo}, each element of $(\oast_{i\leq n}\Add(\kappa^{+i}, 1))^{\c^{+}_{\lambda}[G]}$
is in $\c^{-}_{\lambda}[G]$.
Since $\cof(\lambda) = \omega_{2}$ in $V[G]$, every element of
\[((\oast_{i\leq n}\Add(\kappa^{+i}, 1))^{\c^{+}_{\lambda}[G]})^{\omega_{1}}\] in $V[G]$ has range contained in some element of $\c^{-}_{\lambda}[G]$. The lemma then follows from Lemma \ref{omega1 functions}.
\end{proof}

\section{$\sf{DC}_{\aleph_{2+i}}$ in $\c^+_{\lambda}[G,H \rest i]$}\label{dc3sec}





Lemma \ref{dcomega_3lem} is the third item from the beginning of \textsection\ref{threadsec}, and completes the proof of Theorem \ref{mainthrm}.
We will fix an $i \leq n-1$ for which the lemma is assumed to hold, and prove that it holds also for $i + 1$.
The case $i=0$ is slightly different from the others, so we do it separately in Lemma \ref{cpgdc2lem}.
The following lemma is needed for the case $i=0$. 


\begin{lemma}\label{omega2 closure}\normalfont 
For each $j \leq n-1$ and $\rho < \kappa$ there are stationarily many $\eta < \kappa_{j+1}$ such that, in $\c^+_{\lambda}[G, H_{0}]$,
\[\c^{-}_{\eta}[G, H_{0}]^{\rho}\subseteq \c^{-}_{\eta}[G, H_{0}].\]
\end{lemma}

\begin{proof}
We work in $\c^{+}_{\lambda}[G, H_{0}]$. 
Let (*) be the statement that, for all $\alpha < \beta < \kappa_{j+1}$, if there exists a surjection $s \colon \kappa_{j} \to \alpha$ in $\c^{-}_{\beta}$, then \[\c^{-}_{\alpha}[G, H_0]^{\rho}\subseteq \c^{-}_{\beta}[G, H_0].\]
Then (*) implies that $\eta$ is as desired if it has cofinality greater than $\rho$ and has the property that for each $\alpha < \eta$ there is a $\beta < \eta$ such that $\c^{-}_{\beta}$ contains a surjection from $\kappa_{j}$ to $\alpha$.
Since $\kappa_{j+1}$ is regular by Lemma \ref{strongregghlem}, there are stationarily many such $\eta$. 

Formally we prove (*) by induction on $j$. However, since the arguments for the base case and induction step are almost the same, we give them both at the same time and make explicit the differences at the end.

Suppose that $j \leq n-1$ is given and (*), and therefore the lemma,
holds for all $k < j$. 
Fix $\alpha < \beta$ as given in (*), and let $s \colon \kappa \to \alpha$ be a surjection in $\c^{-}_{\beta}$.
Fix $f \colon \rho \to \c^{-}_{\alpha}[G, H_0]$.
For each $\gamma < \rho$, let $B_{\gamma}$ be the set of $x \in \alpha^{\omega}$ such that
$f(\gamma)$ is defined in $\c^{-}_{\alpha}[G, H_0]$ from $\mH \rest \alpha$, $G$, $H_0$ and $x$ via the formula with G\"{o}del code $x(0)$.
Since $\cof(\kappa_j) > \rho$, there is a $\delta < \kappa_j$ such that for all $\gamma < \rho$, \[B'_{\gamma} = \{ y \in \delta^{\omega} : s \circ y \in B_{\gamma}\}\] is nonempty.
In the case $j=0$, the sequence $\langle B'_{\gamma} : \gamma < \rho \rangle$ is coded by a set of reals in $\c^{+}_{\lambda}[G, H_0]$, so it is in $L(\Delta_{\kappa})[G]$. It follows that $f \in \c^{-}_{\beta}[G,H_0]$. In the case $j >0$, $\delta^{\omega}$ is an element of $\c^{-}_{\kappa_{j-1}}$ by the definition of $\c^{-}_{\kappa_{j-1}}$, so $\langle B'_{\gamma} : \gamma < \rho \rangle$ is an element of $\c^{-}_{\xi}[G, H_{0}]$, for some $\xi < \kappa_{j-1}$, by the induction hypothesis. In either case, 
$\langle B_{\gamma} : \gamma < \rho\rangle$ is in $\c^{-}_{\beta}$, so $f$ is as well. 
\end{proof}


\begin{lemma}\label{dcomega_3prelem} \normalfont $\c^+_{\lambda}[G, H_0]\models \sf{DC}_{\aleph_3}$.
\end{lemma}

\begin{proof}
The model $\c^{+}_{\lambda}[G, H_0]$ contains a wellordering of 
$\kappa^{\omega}$, so it satisfies $\DC_{\infty} \rest \kappa^{\omega}$. 
By Lemma \ref{dcsteplem} then, it suffices to show that, for each $j \leq n-1$, there is no cofinal function from $(\kappa_{j}^{\omega})^{\omega_{2}}$ to $\kappa_{j+1}$ in $\c^{+}_{\lambda}[G, H_0]$. For each such $j$, Lemma \ref{omega2 closure} implies that $(\kappa_{j}^{\omega})^{\omega_{2}}$ is an element of $\c^{-}_{\kappa_{j+1}}$. Part (2) of Lemma \ref{strongregghlem} implies that there is no cofinal function from $(\kappa_{j}^{\omega})^{\omega_{2}}$ to $\kappa_{j+1}$, as desired. 
\end{proof}



\begin{lemma}\label{dcomega_3lem} \normalfont 
For each $i \leq n$, $\c^+_{\lambda}[G, H \rest i]\models \sf{DC}_{\aleph_{2 + i}}$.
\end{lemma}

\begin{proof}
The cases $i=0$ and $i=1$ have been proved in Lemmas \ref{cpgdc2lem} and \ref{dcomega_3prelem}, so we may assume that $i > 1$, in which case $\aleph_{2 + i -1} = \kappa_{i-2}$. We work in $\c^{+}_{\lambda}[G, H \rest i]$. The set
$\kappa_{i-1}^{\omega}$ is wellordered, so $\DC_{\infty} \rest \kappa_{i-1}^{\omega}$ holds. 
By Lemma \ref{dcsteplem}, it suffices to show that, for each $j$ with $i \leq j \leq n-1$, there is no cofinal function from $(\kappa_{j}^{\omega})^{\kappa_{i-2}}$ to $\kappa_{j+1}$. Of course, $(\kappa_{j}^{\omega})^{\kappa_{i-2}}$ can equivalently be replaced with $\kappa_{j}^{\kappa_{i-2}}$. Again, 
$\kappa_{i-2}^{\kappa_{i-2}}$ is wellorded by $H_{i-1}$ in ordertype $\kappa_{i}$, and Lemma \ref{strongregghlem} implies that there is no cofinal function from $\kappa_{i}$ to $\kappa_{j+1}$. 
Now, for each $k$ in the interval $[i-2,j-i]$, the regularity of $\kappa_{k+1}$ and $\kappa_{j+1}$ implies that if there is a cofinal function from $\kappa_{k+1}^{\kappa_{i-2}}$ to $\kappa_{j+1}$, then there is a cofinal function from $\kappa_{k}^{\kappa_{i-2}}$ to $\kappa_{j+1}$. Since there is no such function in the case $k = i-2$, we are done. 
\end{proof}




\part{Constructing Nairian models}

\section{The Main Theorem on Nairian Models}

Theorem \ref{main theorem} is the main theorem of this part of the paper.  


\begin{theorem}\label{main theorem} 
Suppose $\V\models {\sf{ZFC}}$ is a hod premouse\footnote{Our notion of hod premouse is the one introduced in \cite{SteelCom}. We tacitly assume that all large cardinal notions are witnessed by the extenders on the extender sequence of a hod premouse.} and $\xi<\d<\eta$ are such that 
	\begin{enumerate}
		\item $\eta$ is a limit of Woodin cardinals of $\V$,
		\item $\d$ is a Woodin cardinal of $\V$,
		\item $\xi$ is the least inaccessible cardinal of $\V$ which is a limit of $<\d$-strong cardinals of $\V$, and
		\item $\xi$ is a limit of Woodin cardinals of $\V$.
	\end{enumerate}
	Let $g\subseteq Coll(\omega, <\eta)$ be $\V$-generic, and let $M$ be the derived model of $\V$ as computed by $g$.
    Letting 
	\begin{enumerate}[resume]
		\item $\Sigma\in \V$ be the iteration strategy of $\V$ indexed on the sequence of $\V$,
  
		\item $\mathcal{F}=\{ \Q:$ for some $\b<\eta$, $\Q\in \V[g_\b]$ is a $\Sigma^{g_\b}$-iterate of $\V|\eta$ such that $\lh(\T_{\V|\eta, \Q})<\eta$ and $\pi_{\V|\eta, \Q}$ is defined$\}$,\footnote{Here, $\Sigma^h$ is the generic interpretation of $\Sigma$ in $\V[h]$, $\T_{\V|\eta, \Q}$ is the normal $\V|\eta-$to$-\Q$ iteration tree, $\pi_{\V|\eta, \Q}=\pi^{\T_{\V|\eta, \Q}}$. This notation will be used throughout the rest of paper.}
  
		\item $\M_\infty$ be the direct limit of $\mathcal{F}$,\footnote{See \rrem{a rem about main thm} for more on our definition of $\M_\infty$.}
  
		\item $\pi_{\V|\eta, \infty}:\V|\eta\rightarrow \M_\infty$ be the iteration map, and
  
		\item $\l=\pi_{\V|\eta, \infty}(\xi)$, 
	\end{enumerate}
	then for every $n<\omega$, \[M\models \:\Join^n_{\l},\] and \[\text{$(\c_\l^-)^M\models {\sf{ZF}}+``\omega_1$ is a supercompact cardinal."}\]
\end{theorem}


\subsection{Notation, conventions, and preliminary results}\label{sec: notation}

We use the notation and terminology of \cite{SteelCom} where possible.\footnote{The introductory sections of \cite{SteelCom} summarize the results we need in this paper.}
All our hod pairs are lbr hod pairs \cite[Chapter 1.5, Definition 9.2.2]{SteelCom}, and we heavily rely on their basic theory as developed in \cite{SteelCom} and in \cite{MPSC}.
\cite[Theorem 1.4]{MPSC} will be very important, as full normalization is used throughout the sequel. 

\begin{remark} \normalfont All hod pairs are assumed to have iteration strategies with strong hull condensation, and hence full normalization (see \cite[Definition 0.1]{MPSC} and  \cite[Theorem 1.4]{MPSC}). 
We use the \textit{projectum free spaces} fine structure of \cite{SteelCom}. 
\end{remark}

\subsubsection{Fine structure and iteration trees.}
Our mice and hod mice use Jensen's indexing scheme.
Given an extender $E$ over a transitive model $M$, we let 
\begin{enumerate}
	\item $\pi^M_E: M\rightarrow Ult(M, E)$ be the ultrapower embedding,\footnote{We will often omit superscript $M$.}
	\item $\cp(E)$ be the critical point of $E$ and $\im(E)=\pi_E^M(\cp(E))$,\footnote{$\im(E)$ depends only on $E$. Indeed, $E$ determines $\powerset(\cp(E))\cap M$ and if $\powerset(\cp(E))\cap M=\powerset(\cp(E))\cap N$ then $\pi_E^M(\cp(E))=\pi_E^N(\cp(E))$.}
 
	\item $E$ be indexed at $\lh(E)=_{def} \pi_E((\cp(E)^+))$, and
 
	\item $\gen(E)$ be the natural length of $E$, i.e. the supremum of the ordinals of the form $\nu+1$, where $\nu$ is a generator of $E$.\footnote{See \cite[Chapter 2]{SteelCom}.}
\end{enumerate}
Our hod mice have only short extenders, meaning that $Ult(M, E)=\{ \pi_E^M(f)(a): f\in M, f:[\cp(E)]^{\card{a}}\rightarrow M, a\in [\im(E)]^{<\omega}\}$.

We follow a convention introduced in \cite{ANS01} and used in \cite{SteelCom}: A hod premouse $\M$ is a pair $(\M', k(\M))$, where $\M'$ is what one usually calls a hod premouse, i.e., $\M'=J_\a^{\vec{E}, S^\M}$, and $k(\M)\leq \omega$
is such that 
$\M'$ is $k(\M)$-sound.
We will abuse notation and write $\M$ for $\M'$ as well.
$\a(\M)$ stands for $\a$, and
$\fl{\M}$ denotes the universe of $\M$.

Let $\M$ be a premouse and $\b\leq {\sf{Ord}}\cap \M$.
We write $\M|\b$ for $(J_\b^{\vec{E}^\M}, \vec{E}^\M\rest \b, \in)$. We write $\M||\b$ for $(J_\b^{\vec{E}^\M}, \vec{E}^\M\rest \b, E_\b, \in)$.  We write $l(\M)=(\a(\M), k(\M))$. For $(\b, m)\leq l(\M)$, we write $\M|(\b, m)$ for $(\M|\b, m)$ and $\M||(\b, m)$ for $(\M||\b, m)$. Here we are following the notation of \cite{JSSS} and \cite{SchSt09}, which was later changed in \cite{SteelCom}.

We review some concepts from \cite[Chapter 2]{SteelCom}.
Suppose that $\M$ is a premouse and $n=k(\M)$. We say $E$ is an $\M$-\emph{extender} if $\powerset(\cp(E))^\M$ is measured by $E$ and $\cp(E)<\rho_n(\M)$. Given an $\M$-extender $E$ with $\cp(E)<\rho_n(\M)$, we set $Ult(\M, E)$ to be the decoding of $Ult_0(\M^n, E)$, where $\M^n$ is the $n$th reduct of $\M$. We also set $k(Ult(\M, E))=k(\M)$ (provided the resulting ultrapower is $k(\M)$-sound). 

Suppose next that $E$ is an extender such that for some $(\b, m)\leq_{lex} l(\M)$, $E$ is an $\M||(\b, m)$-extender. Letting $(\b, m)\leq l(\M)$ be largest such that $E$ is an $\M||(\b, m)$-extender, we let $Ult(\M, E)$ be the decoding of $Ult_0((\M||\b)^m, E)$ and $k(Ult(\M, E))=m$. If $(\b, m)<l(\M)$, we say that $E$ causes a \emph{drop}. We let $\pi_E^{\M}:\M||\b\rightarrow Ult(\M, E)$. 

Given a premouse $\M$ and ordinal $\l$, we write $\l\in \dom(\vec{E})^\M$ to indicate that there is an extender $E$ on the extender sequence of $\M$ such that $\lh(E)=\l$.
We may also let $E^\M_\l$ be the extender on the sequence of $\M$ such that $\lh(E_\l)=\l$.
If $P$ is a property, then we say $E\in \vec{E}^\M$ is the least that has property $P$ if $\lh(E)$ is the least ordinal $\a$ such that there is $F\in \vec{E}^\M$ with the property that $\lh(F)=\a$ and $F$ has the property $P$.

Our iteration trees are normal, i.e. the lengths of the extenders increase and each extender is applied to the earliest model it can be applied to (see \cite{OIMT} or \cite{SteelCom}).
An iteration tree $\T$ of length $\eta$ on a fine structural model $\M$ (a premouse, hod premouse, etc.) is a tuple \[(\M_\a, \iota_\a, \l_\a, \xi_\a, n_\a, \pi_{\a, \b}, \mathcal{D}, T: \a<\b<\eta)\] such that 
\begin{enumerate}
	\item $T$ is a tree order on $\eta$, 
	\item for $\a+1<\eta$, $T(\a+1)$ is the predecessor of $\a+1$ in the tree order given by $T$,
	\item $\M_0=\M$,
	\item  for each $\a<\eta$, $\iota_\a\leq {\sf{Ord}}\cap \M_\a$,
	\item if $\iota_\a\in \dom(\vec{E})^{\M_\a}$, then, setting $E_\a=E_{\iota_\a}^{\M_\a}$, $\l_\a=\im(E_\a)$ and \[\M_{\a+1}=Ult(\M_{T(\a+1)}||\xi_\a, E_\a),\] where $(\xi_\a, n_\a)\leq_{lex}l(\M_{T(\a+1)})$ is the largest such that $E_{\iota_\a}^{\M_\a}$ is an extender over $\M_{T(\a+1)}||(\xi_\a, n_\a)$,
    \item if $\iota_\a\not \in \dom(\vec{E}^{\M_\a})$, then $\M_{\a+1}=\M_\a$\footnote{We allow padding.} and $\l_{\a}=\iota_\a$,  
    \item for $\a+1<\b+1<\lh(\T)$, $\iota_\a<\iota_\b$ and $\l_\a\leq \l_\b$,
    \item for $\a+1<\lh(\T)$, $\b=T(\a+1)$ if and only if $\b$ is the least $\gg<\lh(\T)$ such that $\cp(E_\a)< \l_\b$ or $\a=\b$,
	\item $\mathcal{D}$ is the set of all those ordinals $\a+1$ such that $(\xi_\a, n_\a)<_{lex} l(\M_{T(\a+1)})$,
	
    \item if $\a T \b$ and $\mathcal{D}\cap [\a, \b)_T=\emptyset$, then $\pi_{\a, \b}: \M_\a\rightarrow \M_\b$ is the iteration embedding,
 
	\item if $\xi<\eta$ is a limit ordinal, then there is $\a T \xi$ such that $[\a, \xi)_T\cap \mathcal{D}=\emptyset$ and, setting $b=[\a, \xi)_T$, $\M_\xi$ is the direct limit of $(\M_\b, \pi_{\b, \gg}: \b<\gg \in b)$. 
\end{enumerate}
To relieve the letter $``T"$ of its duty and use it elsewhere, we will write $\T$ for $T$.
We will use superscript $\T$ to denote objects in $\T$.
Given an iteration tree $\T$, we let 
\begin{enumerate}[resume]
\item $\gen(\T)=\sup \{\l_\a^\T: \a<\lh(\T)\}$, and
\item $\ext_\a^\T=\M_\a^\T|\iota_\a^\T$.
\end{enumerate}

\begin{notation}[Concatenating iterations]\label{concatinating iterations}\normalfont
Suppose $\T$ is an iteration tree on $\M$ with last model $\N$, and $\U$ is an iteration tree on $\N$.
If $\T$ followed by $\U$ is a normal iteration tree on $\M$, then we denote the composite iteration by $\T^\frown \U$.
On the other hand, if $\T$ followed by $\U$ is \emph{not} a normal iteration tree on $\M$, then we write $\T\oplus \U$ for this iteration.
If $E\in \vec{E}^\N$, then we write $\T^{n\frown} \{E\}$ for the normal continuation of $\T$ by $E$. 
\end{notation}

We will use some notation 
from \cite{LSA}.
For example, if  \[\T= (\M_\a, \iota_\a, \l_\a, \xi_\a, n_\a, \pi_{\a, \b}, \mathcal{D}, T: \a<\b<\eta)\]
is an iteration tree and $\gg<\lh(\T)$, then we let 
\[\T_{\geq \gg}= (\M_\a, \iota_\a, \l_\a, \xi_\a, n_\a, \pi_{\a, \b}, \mathcal{D}, T: \gg\leq \a<\b<\eta).\]
In general, $\T_{\geq \gg}$ may not be anything reasonable.
However, it is possible that $\T_{\geq \gg}$ can be represented as an iteration tree on $\M_\gg^\T$ after re-enumerating it.

Similarly, given $\gg<\iota\leq\lh(\T)$, we will use $\T_{[\gg, \iota)}$ or other such notation to indicate the portion of $\T$ that is between $\gg$ and $\iota$. 

Continuing with $\T$, if $\N=\M_\a^\T$ for some $\a\leq \lh(\T)$, then we may write $\T_{\geq \N}$ instead of $\T_{\geq \a}$. In this case, we may also say that $\N$ is a \textit{node of} $\T$. We may say that $\U$ is a \textit{fragment} of $\T$ just in case for some $\gg<\iota\leq \lh(\T)$, $\U=\T_{[\gg, \iota)}$.

Suppose now that $\T$ is an iteration tree on $\M$ and $w=(\nu, \mu)$. We then say that $\T$ is \textit{based} on $w$ just in case $\T$ is \textit{strictly above} $\nu\footnote{For every $\a<\lh(\T)$, $\cp(E_\a^\T)>\nu$.}$ and \textit{below} $\mu$\footnote{For every $\a<\lh(\T)$, either $[0, \a)_\T\cap D^\T\neq \emptyset$ or $\lh(E_\a^\T)<\pi_{0, \a}^\T(\mu)$.}.

\subsubsection{Hod premice.}
We denote hod pairs using symbols $\mathfrak{p}$, $\mathfrak{q}$, $\mathfrak{s}$ and $\mathfrak{r}$. Typically, a hod pair $\hp$ is a pair $(\P, \Sigma)$ such that $\P$ is countable and $\Sigma$ is an $\omega_1$-iteration strategy with  strong-hull-condensation. These conventions may change, but if they do then either the new meanings will be clear from context or will be explicitly stated.
In the sequel, we will write $\hp=(\M^{\hp}, \Sigma^{\hp})$. 

\begin{notation}\label{iteration terminology}\normalfont
	Suppose $\hp$ is a hod pair. 
 \begin{enumerate}
 \item We say $\hq$ is an \textbf{iterate} of $\hp$ if $\M^{\hq}$ is a $\Sigma^{\hp}$-iterate of $\M^\hp$ and $\Sigma^\hq=\Sigma^\hp_{\M^\hq}$.\footnote{If $\Q$ is a $\Sigma$-iterate of $\P$, then $\Sigma_\Q$ is the strategy of $\Q$ that $\Sigma$ induces, so $\Sigma_\Q(\U)=\Sigma(\T\oplus \U)$, where $\T$ is the $\P$-to-$\Q$ iteration tree according to $\Sigma$ and $\T\oplus \U$ is the stack "$\T$ followed by $\U$". The iteration strategies considered in \cite{SteelCom} act on stacks, so our notation makes sense.} In this case, we let $\T_{\hp, \hq}$ be the $\M^\hp$-to-$\M^{\hq}$ iteration tree according to $\Sigma^{\hp}$.
 \item We say $\hq$ is a \textbf{complete iterate} of $\hp$ if $\hq$ is an iterate of $\hp$ and $\M^{\hq}$ is a complete $\Sigma^{\hp}$-iterate of $\M^\hp$ (i.e., the iteration embedding given by $\T_{\hp, \hq}$ is defined). 
 \item We say that $\T$ is an iteration tree on $\hp$ if $\T$ is an iteration tree on $\M^\hp$ according to $\Sigma^\hp$. If $\T$ is an iteration tree on $\hp$  and $\a<\lh(\T)$, then we let $\hm_\a^\T=(\M_\a^\T, \Sigma_{\M_\a^\T})$. 
  \item We let $\M_\infty(\hp)$ be the direct limit of all countable\footnote{We say $\hq$ is a countable iterate of $\hp$ if $\lh(\T_{\hp, \hq})$ is countable.} complete iterates of $\hp$.
	\item If $\hq$ is a complete iterate of $\hp$, then 
    \begin{enumerate}
        \item $\pi_{\hp, \hq}:\M^\hp\rightarrow \M^\hq$ is the iteration embedding given by $\T_{\hp, \hq}$ and $\pi_{\hq, \infty}:\M^{\hq}\rightarrow \M_\infty(\hp)$ is the iteration map according to $\Sigma^{\hq}$, and

        \item $\T_{\hq, \infty}$ is the normal $\M^{\hq}$-to-$\M_\infty(\hp)$ iteration. 
    \end{enumerate}
  \item Suppose $\a\leq {\sf{Ord}}\cap \M^\hp$. Then $\hp|\a=(\M^\hp|\a, \Sigma^\hp_{\M^\hp|\a})$ and $\hp||\a=(\M^\hp||\a, \Sigma^\hp_{\M^\hp||\a})$.
 \end{enumerate}
\end{notation}

The following remark concerns our notation used in \rthm{main theorem}. 

\begin{remark}\label{a rem about main thm}\normalfont
Suppose that $\hp$ is a hod pair.
Comparison is what usually guarantees that $\M_\infty(\hp)$ is defined.
Here, however, we require only a weaker form of comparison, saying that if $\hq$ and $\hr$ are two (countable) complete iterates of $\hp$, then there is a (countable) complete iterate $\hs$ of $\hp$ such that  $\hs$ is also a countable complete iterate of both $\hq$ and $\hr$. Usually $\M_\infty$ is defined in models of ${\sf{AD^+}}$, so by ``countable iterate" we mean countable in the sense of the relevant determinacy  model.
But in the statement of \rthm{main theorem}, $\M_\infty$ is defined using $<\eta$ iterations.
First, notice that $\eta=\omega_1^M$, where $M$ is as in \rthm{main theorem}.
Second, notice that because $\Sigma$ has full normalization, the system $\mathcal{F}$ is indeed directed.\footnote{See \rprop{strategy coherence} and also \cite[Theorem 6.4 and Lemma 6.5]{blue2024ad}.}
\end{remark}

Suppose that $V\models {\sf{AD^{+}}}+{\sf{HPC}}+V=L(\powerset(\bR))$, and let $\theta_\a<\Theta$ be a member of the Solovay sequence.
The results of \cite{SteelCom} imply that there is a hod pair $\hp$ such that $\M_\infty(\hp)\subseteq \H$ with $V_{\theta_\a}^\H=\lfloor \M_\infty(\hp)|\theta_\a\rfloor$.

\begin{terminology}\label{catches}\normalfont Continuing with our hod pair $\hp$, let $\hq$ be a complete iterate of $\hp$. We say $\hq$ \textbf{catches} $\a$ if $\a\in \rge(\pi_{\hq, \infty})$. If $\hq$ catches $\a$, then we let $\a_\hq$ be the $\pi_{\hq, \infty}$-preimage of $\a$. Similarly, if $\hq$ is a complete iterate of $\hp$, $x\in \M^\hq$, and $\hr$ is a complete iterate of $\hq$, then we let $x_\hr=\pi_{\hq, \hr}(x)$.
\end{terminology}

Suppose $\P$ is a hod premouse.
We let $S^\P$ be the strategy indexed on the sequence of $\P$.
Setting $\Sigma=S^\P$ and fixing a $\P$-generic $g$, in many situations $\Sigma$ can be extended to act on iterations in the generic extension $\P[g]$.
For example, if $\eta$ is a regular cardinal of $\P$ and $g$ is a generic for a poset whose size is below a Woodin cardinal that is greater than $\eta$, then $\Sigma_{\P|\eta}$ can be uniquely extended to an iteration strategy in $\P[g]$.
It is appropriate to call this strategy $\Sigma^g_{\P|\eta}$.
If $\P$ is a hod premouse and $\eta$ is an inaccessible limit of Woodin cardinals, then $S^\P_{\P|\eta}$ can be uniquely extended to an $\omega_1$-strategy in $\P[g]$, where $g\subseteq Coll(\omega, <\eta)$ is $\P$-generic.
We will suppress reference to the generic when it will not cause confusion.
See \cite[Chapter 11.1]{SteelCom} for more on interpreting iteration strategies in generic extensions.

\subsubsection{Least-extender-disagreement coiterations and minimal copies.} Suppose that $u=(\hp, \hq)$ is a pair of hod pairs. Then $(\X, \Y)$ is the  \textit{least-extender-disagreement coiteration} of $u$ if $(\X, \Y)$ is obtained by iterating away the least disagreeing extenders.\footnote{ As is done in the usual comparison argument for mice, see \cite{OIMT}.}
Suppose next that $\hp$ is a hod pair and $\hq$ and $\hr$ are complete iterates of $\hp$.
Then the least-extender-disagreement-coiteration of $(\hq, \hr)$ is \textit{successful}, i.e., the least-extender-disagreement-coiteration  of $(\hq, \hr)$ outputs a pair $(\X, \Y)$ such that $\X$ and $\Y$ have a common last model $\hs$. Moreover, for each $\a<\lh(\X)$, exactly one of $E_\a^\X$ and $E_\a^\Y$ is not $\emptyset$. Indeed, for each $\a\leq \lh(\X)$, letting 
\begin{enumerate}
	\item $\xi$ be the least such that $\M_\a^\X|\xi=\M_\a^\Y|\xi$,
	\item $\U$ be the iteration tree on $\M^\hp$ according to $\Sigma^\hp$ that is built by comparing $\M_\a^\X|\xi$ with $\M^\hp$, and
	\item $\P'$ be the last model of $\U$,
\end{enumerate}
if $E_\a^\X$ is not $\emptyset$, then $E_\a^\X\in \vec{E}^{\P'}$ and $\lh(E_\a^\X)=\xi$.\footnote{For more on the least-extender-disagreement-coiteration, see \cite[Theorem 6.4]{blue2024ad}.}

Suppose that $M, N$ are two transitive models, $\nu$ is a regular cardinal of both $M$ and $N$ such that $V_\nu^M=V_\nu^N$, and $\T$ is an iteration of $V_\nu^M$.
We let $\U=_{def}\T^N$ be the unique iteration of $N$ such that $U=T$ and, for every $\a<\lh(\T)$, $E_\a^\U=E_\a^\T$. 

\begin{definition}[Minimal copies]\label{min copies} \normalfont
	Suppose that $\M$ is a fine structural model and $F$ is an $\M$-extender.
    Set $\N=Ult(\M, F)$ and $\pi=\pi_F^\M$. Suppose that $\T$ is a normal iteration of $\M$ above $\cp(F)$ (i.e., all extenders used in $\T$ have critical point $>\cp(F)$).
    We can then (attempt to) \textit{copy} $\T$ onto $\N$ in a minimal fashion reminiscent of the full normalization procedure.
    This minimal-copying procedure produces a normal iteration $\U$ of $\N$ with the following properties.
	\begin{enumerate}
		\item $\lh(\U)=\lh(\T)$, $U=T$., and for every $\a<\lh(\T)$, $\M_\a^\U=Ult(\M_\a^\T, F)$. 
		\item For $\a<\lh(\T)$, let $\sigma_\a=\pi_F^{\M_\a^\T}$. Then $E_\a^\U=\bigcup_{x\in \ext^\T_\a}\sigma_{\a+1}(x\cap E_\a^\T)$.
	\end{enumerate}
	We then say that $\U$ is the \textbf{minimal $(\M, F)$-copy} or simply the \textbf{minimal $F$-copy} of $\T$, and $\sigma_\a$ $(\a<\lh(\T))$ are the \emph{minimal copy maps}.
    When $\M$ is clear from context we will omit it from the notation.
\end{definition}

The minimal $(\M, F)$-copy of $\T$ may not exist, as a priori there is nothing that implies $E_\a^\U\in \vec{E}^{\M_\a^\U}$. However, minimal copies do come up in the full normalization process, and facts like $E_\a^\U\in \vec{E}^{\M_\a^\U}$ can be established assuming our iteration strategy has nice properties (see \cite[Remark 6.16]{SteelCom}).
A key step towards full normalization is the fact that $Ult(\ext^\T_\a, F)\insegeq Ult(\M_\a^\T, F)$, which can be proved just like the fact that $E_\a^\U\in \vec{E}^{\M_\a^\U}$.
The following theorem is essentially the full normalization theorem (\cite[Theorem 1.4]{MPSC}), and can be established using material from \cite[Chapter 6.1]{SteelCom}.\footnote{For more on full normalization, see \cite[Definition 3.7]{schlutzenberg2021normalization} and \cite{siskind2022normalization}.}

\begin{theorem}\label{min copy thm} 
Suppose that $\hp=(\P, \Sigma)$ is a hod pair and $\eta$ is a regular cardinal of $\P$.
Suppose that $\T$ is according to  $\Sigma$ and is strictly above $\eta$,\footnote{I.e. for all $\a<\lh(\T)$, $\eta<\cp(E_\a^\T)$.} $\Q$ is the last model of $\T$, and the iteration embedding $\pi^\T:\P\rightarrow \Q$ is defined.\footnote{I.e., there is no drop along the main branch of $\T$.} Suppose that $\R$ is a $\Sigma_\Q$-iterate of $\Q$ such that, setting $\hq=(\Q, \Sigma_\Q)$ and $\hr=(\R, \Sigma_\R)$, $\gen(\T_{\hq, \hr})\subseteq \pi_{\hq, \hr}(\eta)$. Let $\xi<\lh(\T_{\hp, \hr})$ be the least such that $\M_\xi^{\T_{\hp, \hr}}|\pi_{\hq, \hr}(\eta)=\R|\pi_{\hq, \hr}(\eta)$.
Let $F$ be the long extender derived from $\pi_{\hq, \hr}\rest (\Q|\eta)$, and let $\U$ be the minimal $(\P, F)$-copy of $\T$. Then $\T_{\hp, \hr}=\T_{\hp, \hr}\rest (\xi+1)^\frown \U$.
(In particular, $\M_\xi^{\T_{\hp, \hr}}=Ult(\P, F)$.) 
	
Moreover, for each $\iota<\lh(\U)$, $\M_\iota^\U=Ult(\M_\iota^\T, F)$, and $E_\iota^\U$ is the last active extender of $Ult(\ext^\T_\iota, F)$.
\end{theorem}

We also require that the following consequence of full normalization.\footnote{As pointed out to us by Steel, while \rprop{strategy coherence} can be proved using full normalization, one has to prove such a coherence property in order to develop the theory of hod mice. See \cite[Definition 5.2.4, 2.7.7 and Theorem 5.2.5]{SteelCom}.} The property established in \rprop{strategy coherence} is called \emph{strategy coherence} in \cite{SteelCom}.

\begin{proposition}\label{strategy coherence}\normalfont Suppose that $\hp=(\P, \Sigma)$ is a hod pair such that $\P\models \sf{ZFC-Powerset}$ and $\T$ is a normal iteration tree on $\P$ according to $\Sigma$ with a last model.
Set $\a=\lh(\T)-1$ and suppose $\eta\leq {\sf{Ord}} \cap \M_\a^\T$ and $\b\leq \a$ are such that $\M_\b^\T|\eta=\M_\a^\T|\eta$.
Set $\N=\M_\b^\T|\eta$.\footnote{The proposition is still true for $\N=\M_\b^\T||\eta$, assuming $\M_\b^\T||\eta=\M_\a^\T||\eta$.} 
Then $(\Sigma_{\M_\b^\T})_\N=(\Sigma_{\M_\a^\T})_\N$.\footnote{It follows from the results of \cite[Definition 5.4.4 and Theorem 5.4.5]{SteelCom} that $\Sigma_{\M|\nu}$ is defined to be the $id$-pullback of $\Sigma_\M$.}
\end{proposition}

\begin{proof}
Set $\Q=\M^\T_\b$ and $\R=\M^\T_\a$.
Let $\U$ be a normal iteration tree of limit length that is according to both $\Sigma_{\Q}$ and $\Sigma_{\R}$.
We assume additionally that $\eta$ is an inaccessible cardinal of both $\Q$ and $\R$ and neither of $\Q$ nor $\R$ project across $\eta$.
The general case is not significantly different but needs some fine structure.

Let $\U_\Q$ be the $id$-copy of $\U$ onto $\Q$ and $\U_\R$ be the $id$-copy of $\U$ onto $\R$.
We have that $\U_\Q$ and $\U_\R$ have the same tree order and use the same extenders\footnote{This is where we use the fact that $\eta$ is an inaccessible cardinal in both $\Q$ and $\R$ and neither projects across $\eta$.
In general, we cannot conclude that $\U_\Q$ and $\U_\R$ have the same extenders. However, we can simply apply $\U$ to both $\Q$ and $\R$, producing iterations trees $\S_\Q$ and $\S_\R$ such that $\S_\Q$ and $\S_\R$ have the same tree structure as $\U$ and use exactly the same extenders.
We will then have that $\S_\Q$ and $\S_\R$ are pseudo-hulls of $\S_\Q$ and $\U_\R$, and so the proof can still be implemented. Here $\S_\Q$ will have the following property: for every $\zeta<\lh(\U)$, $\M_\zeta^{\S_\Q}\insegeq \M_\zeta^{\U_\Q}$.
The proof of this fact is very much like the proofs of \cite[Claim 6.1.7 and Claim 6.1.8]{SteelCom}.}.
Now let $\X_0$ be the embedding normalization of $\T_{\leq \b}\oplus \U_\Q$ and $\X_1$ be the embedding normalization of $\T\oplus \U_\R$.
Since $\U$ has limit length and $\Q|\eta=\R|\eta$, $\X_0=\X_1=_{def}\X$. Furthermore, $\Sigma_{\Q}(\U_\Q)$ and $\Sigma_{\R}(\U_\R)$ are completely determined by $\Sigma(\X)$ via exactly the same procedure (see \cite[Section 6.6]{SteelCom} and in particular, \cite[Corollary 6.6.20]{SteelCom}). Therefore, $\Sigma_{\Q}(\U_\Q)=\Sigma_{\R}(\U_\R)$.
\end{proof}

\subsubsection{The extender algebra}
If $M$ is a transitive model and $\d$ is a Woodin cardinal of $M$, let ${\sf{EA}}^M_\d$ denote the extender algebra of $V_\d^M$ at $\d$ with countably many generators.
If $\xi<\d$ and $w=[\xi, \d]$, let ${\sf{EA}}^M_w$ be the extender algebra of $V_\d^M$ that uses only extenders with critical points above $\xi$.
If $M$ has a distinguished extender sequence, then we assume that the extender algebra uses extenders from that sequence.

\subsubsection{Summary of \cite{CCM}}\label{sec: chang model}

The proof of  \rthm{main theorem} uses \cite[Theorem 1.3]{CCM} and \cite[Proposition 2.2]{CCM}, which we summarize here. 

Suppose $\V\models \sf{ZFC}$ is an lbr hod premouse and $\eta$ is an inaccessible limit of Woodin cardinals of $\V$.
Suppose $g\subseteq Coll(\omega, <\eta)$ is $\V$-generic.
Set $\P=\V|(\eta^+)^\V$.
It is shown in \cite{SteelCom} that 
$S^\V$ has a unique canonical extension in generic extensions of  $\V$.
We set $\Sigma=S^\V$ and $\hp=(\P, \Sigma)$.
If $h$ is $\V$-generic, then we let $\hp^h=(\P, \Sigma^h)$, where $\Sigma^h$ is the unique extension of $\Sigma$ in $\V[h]$.

\begin{notation}\normalfont
	Continuing with $\V, \P, \hp$, and $g$ as above:
	\begin{itemize}
		\item For $\a<\eta$, $g_\a=g\cap Coll(\omega, <\a)$.
		\item For $\a<\eta$, $\mathcal{F}^{g_\a}_\hp$ is the set of $\hq^g$, where $\hq$ is a complete iterate of $\hp^{g_\a}$ such that $\T_{\hp^{g_\a}, \hq}\in V[g_\a]$, $\lh(\T_{\hp^{g_\a}, \hq})< \eta$, $\T_{\hp^{g_\a}, \hq}$ is based on $\P|\eta$, and $\pi_{\hq^{g_\a}, \hq}(\eta)=\eta$.\footnote{It is not hard to see that this condition follows from the previous conditions.}
		\item  For $\a<\eta$, $\mathcal{F}^g_\hp=\bigcup_{\a<\eta}\mathcal{F}^{g_\a}_\hp$.
		\item For $\a<\eta$, $\hq\in \mathcal{F}^{g_\a}(\hp)$, and $\b\in [\a, \eta)$,  $\mathcal{F}^{g_\b}_\hq$ is the set of $\hr^g$ where $\hr$ is a complete iterate of $\hq^{g_\b}$ such that $\T_{\hq^{g_\b}, \hr}\in V[g_\b]$ and $\lh(\T_{\hq^{g_\b}, \hr})<\eta$. 
		\item For $(\a, \hq, \b)$ as in the above clause, $\mathcal{F}^{g}(\hq)=\bigcup_{\b\in (\a, \eta)}\mathcal{F}^{g_\b}(\hq)$.
	\end{itemize}
 We will usually just write $\hq$ for $\hq^{g_\b}$. 
\end{notation}

Because $\Sigma^g$ has full normalization, given $\hr, \hs\in \mathcal{F}^g_\hq$, we have $\hw\in \mathcal{F}^g_\hq$ such that $\hw$ is a common iterate of $\hr$ and $\hs$. Moreover, $\pi_{\hr, \hw}\circ \pi_{\hq, \hr}=\pi_{\hs, \hw}\circ \pi_{\hq, \hs}$. It follows that $(\mathcal{F}^g_\hq, \pi_{\hr, \hs}: \hr, \hs\in \mathcal{F}^g_\hq)$ is a directed system. We let $\M_\infty(\hq)$ be the direct limit of $(\mathcal{F}^g_\hq, \pi_{\hr, \hs}: \hr, \hs\in \mathcal{F}^g_\hq)$. Let $\pi_{\hr, \infty}:\R\rightarrow \M_\infty(\hq)$ be the direct limit embedding. Set $\eta_\infty=\pi_{\hq, \infty}(\eta)$.

Suppose that $\hr\in \mathcal{I}^g(\hp)$. The interval $w=(\nu^w, \d^w)$ is a \emph{window} of $\hr$ (or of $\M^{\hr}$) if there are no Woodin cardinals of $\M^{\hr}$ in the interval $w$ and $\d^w$ is a Woodin cardinal of $\M^{\hr}$. If $w, w'$ are two windows of $\M^\hr$ then we write $w<_W w'$ just in case $\d^w<\nu^{w'}$. We say $\hq$ is a \textit{window-based iterate} of $\M^\hr$ if there exist $\iota<\eta$, 
an $<_W$-increasing sequence of windows $(w_\a=(\nu_\a, \d_\a): \a<\eta)$ of $\M^\hr$, and a sequence $(\hq_\a: \a<\eta)\subseteq \mathcal{F}^{g_\iota}_\hr$ (in $V[g_\iota]$) such that
\begin{enumerate}
        \item $\M^\hr\in V[g_\iota]$,
	\item $\hq_0\in \mathcal{F}^{g_\iota}_\hr$ and $\T_{\hr, \hq_0}$ is based on $\M^\hr|\nu_0$,
	\item $\hq_{\a+1}\in \mathcal{F}^{g_\iota}_{\hq_\a}$,
	\item $\T_{\hq_\a, \hq_{\a+1}}$ is based on $\pi_{\hr, \hq_\a}(w_\a)$,
	\item for limit ordinals $\l$, $\M^{\hq_\l}$ is the direct limit of $(\M^{\hq_\a}, \pi_{\hq_\a, \hq_\b}: \a<\b<\l)$,
	\item $\M^\hq$ is the direct limit of $(\hq_\a, \pi_{\hq_\a, \hq_\b}: \a<\b<\eta)$.
\end{enumerate}
\begin{definition}\label{gen iterate}\normalfont Given $\hr\in \mathcal{I}^g(\hp)$, We say that $\hq\in \mathcal{F}^g_\hr$ is a \textbf{genericity iterate} of $\hr$ if it is a window-based iterate of $\hr$ as witnessed by $(w_\a: \a<\eta)$ such that
	\begin{enumerate}
		\item if $x\in \bR_g$, then for some $\a<\eta$, $x$ is generic for the extender algebra ${\sf{EA}}^{\M^\hq}_{\pi_{\hp, \hq}(w_\a)}$,\footnote{This is the extender algebra of $\hq$ at $\d^{\pi_{\hp, \hq}(w_\a)}$ using extenders whose critical points are $>\eta^{\pi_{\hr, \hq}(w_\a)}$.}
		\item for each $\a<\eta$, $w_\a\in \rge(\pi_{\hr, \hq})$.
	\end{enumerate} 
\end{definition}
Suppose that $\hq=(\Q, \Lambda)$ is a genericity iterate of $\hr=(\R, \Psi)$ with $\hr\in \mathcal{I}^g(\hp)$.
Then there is $h\subseteq Coll(\omega, <\eta)$ with $h\in \V[g]$ such that $h$ is $\Q$-generic and $\bR^{\Q[h]}=\bR_g$. We call such an $h$ a \textit{maximal generic}.
Let $\V_\hq$ be the last model of the iteration tree $\T$ that is obtained by copying $\T_{\hp, \hq}$ onto $\V$ via the identity embedding, and set (in $\V[g]$)
\begin{center}
	$C(\eta, g)=L(\M_\infty(\hp), \bigcup_{\xi<\eta_\infty}\powerset_{\omega_1}(\M_\infty(\hp)|\xi), \Gamma^\infty_g, \bR_g)$
\end{center}

Theorem \ref{invariance} summarizes what we need from \cite{CCM}.

\begin{theorem}\label{invariance}
Suppose $\hq$ is a genericity iterate of $\hp$ and $h$ is a maximal $\hq$-generic. Then $\M_\infty(\hq)=(\M_\infty(\hq))^{\M^{\hq}[h]}$, $\M_\infty(\hq)=\M_\infty(\hp)$ and $C(\eta, g)=C(\eta, h)^{\V_\hq(\bR_g)}$.
\end{theorem}

$C(\eta, h)^{\V_\hq(\bR_g)}$ is computed in $\V_\hq(\bR_g)$ relative to $\hq$.
More precisely, if $h$ is a maximal $\hq$-generic, then $C(\eta, h)^{\V_\hq(\bR_g)}$ is just $L(\M_\infty(\hq), \bigcup_{\xi<\eta_\infty}\powerset_{\omega_1}(\M_\infty(\hq)|\xi), \Gamma^\infty_h, \bR_h)$ as computed in $\V_\hq[h]$.

\section{Bounded iterations and small cardinals}

\subsection{Splitting iteration trees}\label{splitting trees sec} 

\begin{definition}\label{mu bounded iterations}\normalfont
	Suppose $\Q$ is an lbr premouse or some type of hybrid premouse and $\mu\in \Q$. 
    Then an iteration tree $\T$ is \textbf{$\mu$-bounded} if, assuming that $\T$ has a successor length and $\pi^\T(\mu)$ is defined,\footnote{The embeddings in an iteration tree are partial embeddings, and so by ``$\pi^\T(\nu)$ is defined" we mean that for every $\b_0<\b_1$ on the main branch of $\T$, $\pi^\T_{0, \b_0}(\nu)$ is in the domain of $\pi^{\T}_{\b_0, \b_1}$.} $\gen(\T)\subseteq \pi^\T(\mu)$.
 If $\hq$ is a hod pair, then $\hr$ is a $\mu$-\textbf{bounded iterate} of $\hq$ if $\T_{\hq, \hr}$ is $\mu$-bounded.
\end{definition}

\begin{figure}
\begin{tikzpicture}[scale=1.5]

    \draw (0,0) -- (0,3);
    \node[left,below] at (0,0) {$\mathfrak{q}$};
    \draw (2,0) -- (2,3);
    \node[left,below] at (2,0) {$\mathfrak{w}$};
    \draw (4,0) -- (4,3);
    \node[left,below] at (4,0) {$\mathfrak{r}$};

    \draw (0,1) -- (-0.2,1);
    \node[left] at (-0.2,1) {$\nu$};
    \draw (0,2) -- (-0.2,2);
    \node[left] at (-0.2,2) {$\mu$};

    \draw (2,1.7) -- (1.8,1.7);
    \node[right] at (2,1.7) {$\pi_{\mathfrak{q},\mathfrak{r}}(\nu)$};
    \draw (4,1.7) -- (3.8,1.7);
    \node[right] at (4,1.7) {$\pi_{\mathfrak{q},\mathfrak{r}}(\nu)$};
    
    \draw (4,2.7) -- (3.8,2.7);
    \node[right] at (4,2.7) {$\pi_{\mathfrak{q},\mathfrak{r}}(\mu)$};
    
    \draw[dashed,->] (0,1) -- (1.8,1.7);
    
    \draw[dashed,->] (0,2) -- (3.8,2.7);

    \node[below=6mm] at (1,0) {$\mathcal{X}$};
    
    \node[below=6mm] at (3,0) {$\mathcal{Y}$};
    \end{tikzpicture}
    \caption{Lemma \ref{splitting lemma}. $\mathcal{Y}$ is strictly above $\pi_{\mathfrak{q},\mathfrak{w}}(\nu)$.}
\end{figure}

Lemma \ref{splitting lemma} is the simple key idea behind many of the proofs to come.

\begin{lemma}\label{splitting lemma} 
	Suppose that $\hq=(\Q, \Lambda)$ is a hod pair and that $\mu$ is a cardinal of $\Q$ such that, letting 
	\begin{center}
		$\nu=\sup(\{\nu'<\mu: \Q\models ``\nu'$ is a Woodin cardinal$"\})$,
	\end{center}
	$\nu<\mu$. Suppose that $\hr=(\R, \Psi)$ is an iterate of $\hq$ such that $\pi_{
 \hq, \hr}(\nu)$ is defined and $\T_{\hq, \hr}$ is $\mu$-bounded. Then $\T_{\hq, \hr}$ can be split into left and right components $\X$ and $\Y$ such that 
	\begin{enumerate}
		\item the main branch of $\X$ does not drop and $\gen(\X)\subseteq \pi_{\Q, \R}(\nu)$,
		\item if $\hw$ is the last model of $\X$, then $\pi_{\hq, \hw}(\nu)= \pi_{\hq, \hr}(\nu)$, 
		\item $\Y$ is a normal iteration tree on $\hw$ that is strictly above $\pi_{\hq, \hw}(\nu)$, and
		\item $\T_{\hq, \hr}=\X^\frown \Y$.
	\end{enumerate}
\end{lemma}
\begin{proof} Suppose $\gg$ is the least $\xi$ such that $\M_\xi^{\T_{\hq, \hr}}|\pi_{\hq, \hr}(\nu)=\R|\pi_{\hq, \hr}(\nu)$.
Set $\X=(\T_{\hq, \hr})_{\leq \gg}$.
It follows that $\pi_{\hq, \hr}(\nu)=\pi_{0, \gg}^{\T_{\hq, \hr}}(\nu)$, and therefore $\gg$ is on the main branch of $\T_{\hq, \hr}$. Moreover, $(\T_{\hq, \hr})_{\geq \gg}$ is strictly above $\pi_{\hq, \hr}(\nu)$.\footnote{This follows from the facts that (1) in $\Q$, $\nu$ is either a Woodin cardinal or a limit of Woodin cardinals and (2) $\Q$ has no Woodin cardinals in the interval $(\nu, \mu)$.
So if $E$ is an extender used in $(\T_{\hq, \hr})_{\geq \gg}$, then $\cp(E)>\nu$, as otherwise $\hr$ would have a Woodin cardinal in the interval $(\pi_{\hq, \hr}(\nu), \pi_{\hq, \hr}(\mu))$.
Here we use the fact that $\gen(\T_{\hq, \hr})\subseteq \pi_{\hq, \hr}(\mu)$.} Therefore, $\X$ and $\Y=_{def}(\T_{\hq, \hr})_{\geq \gg}$ are as desired.
\end{proof}

\begin{notation}\label{nu infty}\normalfont Suppose $\hq=(\Q, \Lambda)$ is a hod pair and $\nu\in \Q$.
Let 
\begin{enumerate}
\item $\T^\nu_{\hq}=(\T_{\hq, \infty})_{\leq \xi}$, where $\xi$ is the least $\xi'$ such that $\M_{\xi'}^{\T_{\hq, \infty}}||\pi_{\hq, \infty}(\nu)=\M_\infty(\hq)||\pi_{\hq, \infty}(\nu)$, and
\item $\hq^\nu$ be the last model of $\T^\nu_{\hq}$ and $\Q^\nu=\M^{\hq^\nu}$.
\end{enumerate}
\end{notation}
In clause 1 of \rnot{nu infty}, $\T^\nu_{\hq}$ is the longest initial segment $\U$ of $\T_{\hq, \infty}$ such that $\gen(\U)\subseteq \nu$. In most applications of \rlem{xi is on the main branch}, we will have that $\nu\not \in \dom(\vec{E}^\Q)$, and so in \rnot{nu infty} we could change the equality $\M_{\xi'}^{\T_{\hq, \infty}}||\pi_{\hq, \infty}(\nu)=\M_\infty(\hq)||\pi_{\hq, \infty}(\nu)$ to $\M_{\xi'}^{\T_{\hq, \infty}}|\pi_{\hq, \infty}(\nu)=\M_\infty(\hq)|\pi_{\hq, \infty}(\nu)$\footnote{Also, in the context of \rlem{xi is on the main branch}, the equality $\M_{\xi'}^{\T_{\hq, \infty}}|\pi_{\hq, \infty}(\nu)=\M_\infty(\hq)|\pi_{\hq, \infty}(\nu)$ suffices. This is because if the lemma fails then it can be shown that $\cf(\pi_{\hq, \infty}(\nu))=\omega$ which is clearly false.}.

\begin{lemma}\label{xi is on the main branch}
Suppose $(\hq,  \nu, \xi)$ are as in \rnot{nu infty}. Then $\xi$ is on the main branch of $\T_{\hq, \infty}$.
\end{lemma}
\begin{proof} Suppose $\xi$ is not on the main branch of $\T_{\hq, \infty}$. It follows that for some $\b$,\\\\
	(1.1) $\cp(E_\b^{\T_{\hq, \infty}})<\pi_{\hq, \infty}(\nu)$, \\
	(1.2) $\T_{\hq, \infty}(\b+1)<\xi$, \\
	(1.3) $\b+1$ is on the main branch of $\T_{\hq, \infty}$, and \\
	(1.4) $\lh(E_\b^{\T_{\hq, \infty}})\geq \pi_{\hq, \infty}(\nu)$.\\\\
	Then if $\b'=\T_{\hq, \infty}(\b+1)$ and $\b''+1=\lh(\T_{\hq, \infty})$, then 
	$\pi_{\b', \b''}^{\T_{\hq, \infty}}(\pi_{0, \b'}^{\T_{\hq, \infty}}(\nu))>\pi_{\hq, \infty}(\nu)$, which is a contradiction.
\end{proof}

\begin{remark}\label{remark about infty notation} \normalfont
	It will be convenient to also write $\hq^\nu$ for $\nu\in \M_\infty(\hq)$ such that $\hq$ catches $\nu$. In this case, we set $\hq^\nu=\hq^{\nu_\hq}$. 
\end{remark}


\begin{figure}
\begin{tikzpicture}[every node/.style={midway}]
  \matrix[column sep={8em,between origins}, row sep={5em}] at (0,0) {
    \node(Qnu) {$\mathcal{Q}^{\nu}_{\infty}$}  ; & \node(Rnu) {$\mathcal{R}^{\nu_{\mathcal{R}}}_{\infty}$}; \\
    \node(Q) {$\mathcal{Q}$}; & \node (R) {$\mathcal{R}$};\\
  };
  \draw[->] (Q) -- (R) node[below]  {$\mathcal{T}_{\mathcal{Q},\mathcal{R}}$};
  \draw[->] (Q) -- (Qnu) node[left]  {$\mathcal{T}^{\nu}_{\mathcal{Q},\infty}$};
  \draw[->] (R) -- (Rnu) node[right] {$\mathcal{T}_{\mathcal{R},\infty}^{\nu_{\mathcal{R}}}$};
  \draw[->] (Q) -- (Rnu) node[above] {$\mathcal{X}$};
  \draw[->] (Qnu) -- (Rnu) node[above] {$\pi^{\mathcal{X}}$};
\end{tikzpicture}
\caption{Lemma \ref{bound iterations factor}.
$\mathcal{X}$ is the full normalization of $(\T_{\hq, \hr})^\frown \T^{\nu_\hr}_{\hr, \infty}$, and $\cp(\pi^\X_{\xi, \zeta})>\nu_\infty$.}
\end{figure}

\begin{lemma}\label{bound iterations factor}
Suppose $\hq=(\Q, \Lambda)$ is a hod pair and $\nu$ is a $\Q$-cardinal. Suppose $\hr=(\R, \Psi)$ is a complete $\nu$-bounded iterate of $\hq$. Then $\hq^\nu=\hr^{\nu_\hr}$.
\end{lemma}
\begin{proof} The proof follows from the fact that $\gen(\T_{\hq, \hr})\subseteq \nu_\hr$, which implies that \[\R=\{\pi_{\hq, \hr}(f)(a): a\in [\nu_\hr]^{<\omega}\wedge f\in \Q\}.\] Notice that $\pi_{\hq, \infty}(\nu)=\pi_{\hr, \infty}(\nu_\hr)$ and, setting $\nu_\infty=\pi_{\hq, \infty}(\nu)$, $\gen(\T^\nu_{\hq})\subseteq \nu_\infty$ and $\gen(\T^{\nu_\hr}_{\hr})\subseteq \nu_\infty$.

	Let $\X$ be the full normalization of $(\T_{\hq, \hr})^\frown \T^{\nu_\hr}_{\hr}$.
    Then if $\xi$ is such that $\T^\nu_{\hq}=(\T_{\hq})_{\leq \xi+1}$, then $\T^\nu_{\hq}=\X_{\leq \xi+1}$. 
	
	We claim that $\xi$ is on the main branch of $\X$.
    For since $\pi^\X(\nu)=\nu_\infty$, it follows from the argument of \rlem{xi is on the main branch} that $\xi$ is on the main branch of $\X$.
    We then have that if $\zeta+1=\lh(\X)$, then $\cp(\pi^\X_{\xi, \zeta})>\nu_\infty$. 
	
	We now claim that $\pi^\X_{\xi, \zeta}$ is onto.
    Let $y\in \hr^{\nu_\hr}_\infty$. Because $\gen(\T^{\nu_\hr}_{\hr})\subseteq \nu_\infty$, we have that $y=\pi^{\nu_\hr}_{\hr}(g)(a)$, where $a\in [\nu_\infty]^{<\omega}$ and $g\in \R$.
    Because $\gen(\T_{\hq, \hr})\subseteq \nu_\hr$, we have some $h\in \Q$ and $b\in [\nu_\infty]^{<\omega}$ such that $y=\pi^{\nu_\hr}_{\hr}(\pi_{\hq, \hr}(h))(b)$.
    Hence if $x=\pi^\nu_{\hq}(h)(b)$, then $\pi^\X_{\xi, \zeta}(x)=y$.
\end{proof}

\rcor{bound iterations factor 2} follows from \rlem{bound iterations factor} (which implies that $\hq^\nu=\hr^{\nu_\hr}$) and full normalization as encapsulated in \rlem{min copy thm}.

\begin{corollary}\label{bound iterations factor 2}
Suppose that $\hq=(\Q, \Lambda)$ is a hod pair and that $\mu$ is a cardinal of $\Q$ such that, letting 
	\begin{center}
		$\nu=\sup(\{\nu'<\mu: \Q\models ``\nu'$ is a Woodin cardinal$"\})$,
	\end{center}
$\nu<\mu$. Suppose $\hr=(\R, \Psi)$ is a complete $\mu$-bounded iterate of $\hq$. Then $\hr^{\nu_\hr}$ is a normal iterate of $\hq^\nu$ via a normal tree $\U$ which is strictly above $\pi_{\hq}(\nu)$.
	
	More precisely, setting $\T_{\hq, \hr}=\X^\frown \Y$, where $(\X, \Y)$ is as in \rlem{splitting lemma}, and letting $F$ be the long extender derived from $\pi_{\hq, \infty}\rest \Q|(\nu^+)^\Q$, $\U$ is the minimal $(\Q, F)$-copy of $\U$.\footnote{More precisely, $F$ consists of pairs $(a, A)$ such that $a\in [\pi_{\hq, \infty}(\nu)]^{<\omega}$, $A\in \Q$ and $a\in \pi_{\hq, \infty}(A)$.}
\end{corollary}

\subsection{Small cardinals}\label{small cardinals subsec} The splitting of iterations trees into left and right components as described in \rlem{splitting lemma} can be slightly generalized at a cost.
In \rlem{splitting lemma} we did not make any smallness assumption on $\Q$.
Here we will introduce a splitting assuming $\mu$ is \textit{small}.
Small cardinals will appear throughout our analysis of Nairian models derived from hod mice.

The key ideas used in this subsection are now standard in the inner model theoretic literature. 

\begin{terminology}\label{small cardinal}\normalfont Suppose $\hq=(\Q, \Lambda)$ is a hod pair and $\d$ is a cutpoint Woodin cardinal of $\Q$. 
\begin{enumerate}
\item We say $\d$ \textbf{ends an lsa block} of $\Q$ if the least ${<}\d$-strong cardinal of $\Q$ is a limit of Woodin cardinals of $\Q$.\footnote{This terminology is justified by \cite{LSA}.}

\item Given $\nu<\d$, we say $\nu$ is in the \textbf{$\d$-block} of $\Q$ if, letting $\k$ be the least ${<}\d$-strong cardinal of $\Q$, $\nu\in [\k, \d]$. 

\item Assuming $\d$ ends an lsa block of $\Q$, given $\nu$ in the $\d$-block of $\Q$, we say $\nu$ is \textbf{big} in $\Q$ (relative to $\d$\footnote{Usually $\d$ will be clear from the context.}) if there is a $\tau\leq \nu$ such that $\Q\models ``\tau$ is a measurable cardinal that is a limit of ${<}\d$-strong cardinals". 

\item For $\nu<\d$ as above, we say $\nu$ is \textbf{small} in $\Q$ if it is not big in $\Q$. 

\item Given a $\Q$-cardinal $\nu$, we say $\nu$ is \textbf{properly overlapped} in $\Q$ if there is a $\d'$ such that $\d'$ is a Woodin cardinal of $\Q$, $\d'$ ends an lsa block of $\Q$, $\nu$ is in the $\d'$-block of $\Q$, and for all $\k\leq \nu$, $\Q|\nu\models ``\k$ is a strong cardinal" if and only if $\Q\models ``\k$ is a ${<}\d$-strong cardinal."\footnote{Our convention is that the large cardinal notions are witnessed by the extenders on the sequence of the relevant fine structural models.}
\item Given a $\Q$-cardinal $\nu$, we say $\nu$ is a \textbf{proper cutpoint} in $\Q$ if $\nu$ is properly overlapped and is not a critical point of a total extender of $\Q$.
\end{enumerate}
\end{terminology}

Recall \rter{catches}.
Lemma \ref{main lemma on small ordinals} isolates the main property of small ordinals that we will use.

\begin{lemma}\label{main lemma on small ordinals}
Suppose $\hq=(\Q, \Lambda)$ is a hod pair and $\d$ ends an lsa block of $\Q$. Suppose 
\begin{enumerate}[label=(\alph*)]
\item $\mu$ is in the $\d$-block of $\Q$ and is a properly overlapped small regular cardinal of $\Q$,
\item $\hr=(\R, \Psi)$ is a complete iterate of $\hq$, 
\item $\nu\leq \mu_\hr$ is a ${<}\d_\hr$-strong cardinal of $\R$ or a limit of such cardinals, 
\item $\zeta$ is the least such that $\M_\zeta^{\T_{\hq, \hr}}|\nu=\R|\nu$, and 
\item $\zeta'$ is the least such that $\M_{\zeta'}^{\T_{\hq, \hr}}|\mu_\hr=\R|\mu_\hr$. 
\end{enumerate} 
Set $\S=\M_\zeta^{\T_{\hq, \hr}}$, $\S'=\M_{\zeta'}^{\T_{\hq, \hr}}$, $\hs=(\S, \Lambda_\S)$ and $\hs'=(\S', \Lambda_{\S'})$. Then
\begin{enumerate}
\item $\mu_{\hr}=\mu_{\hs'}$,
\item $(\T_{\hq, \hr})_{\geq \zeta}$ is a normal iteration tree on $\S$ according to $\Lambda_\S$, 
\item $(\T_{\hq, \hr})_{\geq \zeta'}$ is a normal iteration tree on $\S'$ according to $\Lambda_{\S'}$,
\item both $\zeta$ and $\zeta'$ are on the main branch of $\T_{\hq, \hr}$,
\item $\gen(\T_{\hq, \hs})\subseteq \nu$,
\item $\gen(\T_{\hq, \hs'})\subseteq \mu_{\hs'}$,
\item $(\T_{\hq, \hr})_{\geq \zeta}$ is strictly above $\nu$, and
\item $(\T_{\hq, \hr})_{\geq \zeta'}$ is strictly above $\mu_{\hr}$.
\end{enumerate}
\end{lemma}

\begin{proof}
We assume that $\nu<\mu_{\hr}$; the other case is easier.
Some of the clauses above have either been already verified or are immediate consequences of the others.
Clauses 1, 2 and 3 follow from clauses 7 and 8.
Part of clause 4, namely that $\zeta'$ is on the main branch of $\T_{\hq, \hr}$, is essentially \rlem{xi is on the main branch}.
However, clause 4 follows from clauses 7 and 8.
Clauses 5 and 6 follow from the definitions of $\zeta$ and $\zeta'$.
That leaves clauses 7 and 8.
Since their proofs are very similar, we demonstrate clause 7, which is the harder of the two. 

Suppose that there is an extender $E$ used in $(\T_{\hq, \hr})_{\geq \zeta}$ with $\cp(E)\leq \nu$.
Let $\a\in [\zeta, \lh(\T_{\hq, \hr}))$ be such that $E=E_\a^{\T_{\hq, \hr}}$. To save ink, we drop $\T_{\hq, \hr}$ from superscripts.
We have that\\\\
(1.1) $\cp(E_\a)\leq \nu<\im(E_\a)$, and\\
(1.2) $\M_\a|\im(E_\a)=\R|\im(E_\a)$.\\\\
It follows from clause (c) that\footnote{$\im(E_\a)$ is an inaccessible cardinal in $\R$ and $\nu<\im(E_\a)$ is a strong cardinal or a limit of strong cardinals of $\R$. See also (1.2).}\\\\
(2) there is $\eta\in [\cp(E_\a), \nu]$ such that $\M_\a|\im(E_\a)\models ``\eta$ is a strong cardinal."\\\\
(2) and (1.1)-(1.2) imply that\\\\
(3.1) $\M_\a|\cp(E_\a)\models ``$there is a class of strong cardinals," and\\ 
(3.2) if $\cp(E_\a)<\nu$, then $\M_\a|\nu\models ``\cp(E_\a)$ is a strong cardinal."\\\\
Since $\cp(E_\a)$ is a measurable cardinal of $\M_\a$ and $\cp(E_\a)\leq \nu$, clause (c) implies that $\R\models ``\mu_\hr$ is big."
\end{proof}

\rcor{hull property cor} follows straightforwardly from \rlem{main lemma on small ordinals}.
The terminology it introduces is inspired by the hull and definability properties used in the theory of the core model (see \cite{CMIP}).

\begin{corollary}\label{hull property cor}
Suppose $\hq=(\Q, \Lambda)$ is a hod pair, $\hr=(\R, \Lambda_\R)$ is a complete iterate of $\hq$, $\d$ ends an lsa block of $\Q$, and $\mu$ is such that either
\begin{enumerate}
\item $\mu\in \rge(\pi_{\hq, \hr}\rest \d)$, or
\item $\mu$ is in the $\pi_{\hq, \hr}(\d)$-block of $\R$, $\mu$ is small in $\R$, and $\mu$ is either a $<\pi_{\hq, \hr}(\d)$-strong cardinal of $\R$ or a limit of such cardinals.
\end{enumerate}
Then the following properties hold.
\begin{enumerate}
\item \textbf{Hull Property:} $\powerset(\mu)\cap \R\subseteq Hull^{\R|\pi_{\hq, \hr}(\d)}(\rge(\pi_{\hq|\d_\hq, \infty})\cup \mu)$.
\item \textbf{Definability Property:} $\mu\in Hull^{\R}(\rge(\pi_{\hq|\d_\hq, \infty})\cup \mu)$.
\end{enumerate}
\end{corollary}

The hull and definability properties imply the next corollary via standard arguments.

\begin{corollary}\label{no extender disagreement}
Suppose
\begin{enumerate}
\item $\hp=(\P, \Sigma)$ is a hod pair, 
\item $\d$ is a Woodin cardinal of $\P$ that ends an lsa block, 
\item $\hq=(\Q, \Sigma_\Q)$ and $\hr=(\R, \Sigma_\R)$ are two complete iterates of $\hp$,
\item $\nu$ is in the $\d$-block of $\P$, $\nu$ is small in $\P$, and $\nu$ is either a $<\d$-strong cardinal of $\P$ or a limit of such cardinals, and
\item $\nu_\hq=\nu_\hr$ and $\Q|\nu_\hq=\R|\nu_\hr$. 
\end{enumerate}
Then the least-extender-disagreement-coiteration of $(\hq, \hr)$ is strictly above $\nu_\hq$. 
\end{corollary} 
\begin{proof} Let $\hs=(\S, \Sigma_\S)$ be the common complete iterate of $\hq$ and $\hr$ obtained via the least-extender-disagreement-coiteration of $(\hq, \hr)$.
We first show that $\cp(\pi_{\hq, \hs})>\nu_\hq$ and $\cp(\pi_{\hr, \hs})> \nu_\hr$.
Notice that if $\cp(\pi_{\hq, \hs})\leq \nu_\hq$, then $\cp(\pi_{\hq, \hs})$ is a $<\d_\hq$-strong cardinal of $\Q$.
Similarly, if $\cp(\pi_{\hr, \hs})\leq \nu_\hr$, then $\cp(\pi_{\hr, \hs})$ is a $<\d_\hr$-strong cardinal of $\R$. 

Assume now that $\cp(\pi_{\hq, \hs})\leq \nu_{\hq}$.
The definability property implies that $\cp(\pi_{\hq, \hs})=\cp(\pi_{\hr, \hs})$, and moreover, that if $\zeta=\cp(\pi_{\hq, \hs})$, then $\pi_{\hq, \hs}(\zeta)=\pi_{\hr, \hs}(\zeta)$.
The hull property then implies that $\pi_{\hq, \hs}\rest (\Q|(\zeta^+)^\Q)=\pi_{\hr, \hs}\rest (\R|(\zeta^+)^\R)$.
Therefore, the first extender used on the main branch of $\T_{\hq, \hs}$ is compatible with the first extender used on the main branch of $\T_{\hr, \hs}$.

We now have that $\cp(\pi_{\hq, \hs})>\nu_\hq$ and $\cp(\pi_{\hr, \hs})> \nu_\hr$.
It follows that $\nu_\hq$ is a $<\d_{\hs}$-strong cardinal of $\S$. Because $\nu$ is small in $\P$, it follows from \rlem{main lemma on small ordinals} that $\T_{\hq, \hs}$ is strictly above $\nu_\hq$.
A similar argument shows that $\T_{\hr, \hs}$ is strictly above $\nu_\hr$.
\end{proof}

Lemma \ref{no overlapping} is another consequence of the hull and definability properties.
It essentially says that in least-extender-disagreement-coiterations, small strong cardinals can only be used on one side of the coiteration.

\begin{lemma}\label{no overlapping}
Suppose 
\begin{enumerate}
\item $\hp=(\P, \Sigma)$ is a hod pair, 
\item $\d$ is a Woodin cardinal of $\P$ that ends an lsa block, 
\item $\hq=(\Q, \Sigma_\Q)$ and $\hr=(\R, \Sigma_\R)$ are two complete iterates of $\hp$,
\item $\hs=(\S, \Sigma_\S)$ is the common iterate of $\hq$ and $\hr$ obtained via the least-extender-disagreement-coiteration of $(\hq, \hr)$, and
\item $\zeta$ is a small ${<}\pi_{\hp, \hs}(\d)$-strong cardinal of $\S$ such that if $\a<\lh(\T_{\hq, \hs})$ is the least with the property that $\zeta\in \rge(\pi^{\T_{\hq, \hs}}_{\a})$, then letting $\zeta_{\hq}$ be the $\pi^{\T_{\hq, \hs}}_\a$-preimage of $\zeta$, $\sup(\pi^{\T_{\hq, \hs}}_\a[\zeta_{\hq}])<\zeta$.\footnote{I.e., somewhere along the main branch of $\T_{\hq, \hs}$ an extender with critical point the image of $\zeta_\hq$ is used.
More precisely, there is $\b\in [\a, \lh(\T_{\hq, \hs}))$ such that $\b$ is on the main branch of $\T_{\hq, \hs}$ and $\cp(E_\b^{\T_{\hq, \hs}})=\pi_{\a, \b}^{\T_{\hq, \hs}}(\zeta_\hq)$.} 
\end{enumerate}
Let now $\b$ be the least such that $\zeta\in \rge(\pi^{\T_{\hr, \hs}}_\b)$, and let $\zeta_\hr$ be the $\pi^{\T_{\hr, \hs}}_\b$-preimage of $\zeta$. Then $\sup(\pi^{\T_{\hr, \hs}}_\b[\zeta_{\hr}])=\zeta$.\footnote{Nowhere along the main branch (and since $\zeta$ is small, anywhere else) of $\T_{\hr, \hs}$ an extender of critical point the image of $\zeta_\hr$ is used.}
\end{lemma}
\begin{proof}
Assume towards a contradiction that our claim is false, and that $\zeta$ is the least for which it fails.
Notice that there are $\iota_{\hq}<\lh(\T_{\hq, \hs})$ and $\iota_{\hr}<\lh(\T_{\hr, \hs})$ such that\\\\
(1.1) $\iota_{\hq}$ is on the main branch of $\T_{\hq, \hs}$ and $\a\leq \iota_{\hq}$,\\
(1.2) $\iota_{\hr}$ is on the main branch of $\T_{\hr, \hs}$ and $\b\leq \iota_{\hr}$,\\
(1.3) letting $E$ be the extender used at $\iota_{\hq}$ on the main branch of $\T_{\hq, \hs}$, $\cp(E)=\pi_{\a, \iota_{\hq}}^{\T_{\hq, \hs}}(\zeta_\hq)$, and\\
(1.4) letting $F$ be the extender used at $\iota_{\hr}$ on the main branch of $\T_{\hr, \hs}$, $\cp(F)=\pi_{\b, \iota_{\hr}}^{\T_{\hr, \hs}}(\zeta_\hr)$.\\\\
Assume without loss of generality that $\iota_\hq$ and $\iota_\hr$ are the least satisfying the above requirements and $\iota_{\hq}\leq \iota_{\hr}$.

We say $\gg$ is a $(\hq, \zeta)$\textit{-point} if 
\begin{enumerate}[label=(\alph*)]
\item $\gg$ is on the main branch of $\T_{\hq, \hs}$, 
\item $\iota_\hq\leq \gg$, and 
\item if $H$ is the extender used on the main branch of $\T_{\hq, \hs}$ at $\gg$, then $\cp(H)=\pi_{\a, \gg}^{\T_{\hq, \hs}}(\zeta_{\hq})$.
\end{enumerate} 
We say $\gg$ is a $(\hq, \hr, \zeta)$\textit{-point}
if 
\begin{enumerate}[label=(\alph*)]
\item $\gg$ is on the main branch of $\T_{\hq, \hs}$, 
\item $\gg$ is a successor ordinal, and 
\item if $\gg'$ is the $\T_{\hq, \hs}$-predecessor of $\gg$, then $\gg'$ is a $(\hq, \zeta)$-point and $\pi_{\a, \gg}^{\T_{\hq, \hs}}(\zeta_{\hq})< \pi_{\b, \iota_\hr}^{\T_{\hr, \hs}}(\zeta_{\hr})$.
\end{enumerate}
Let $\nu=\sup \{ \gg<\iota_{\hr}: \gg$ is a $(\hq, \hr, \zeta)$-point$\}$.

We now have that $\nu$ is on the main branch of $\T_{\hq, \hs}$ and $\pi_{\a, \nu}^{\T_{\hq, \hs}}(\zeta_{\hq})\leq \pi_{\b, \iota_\hr}^{\T_{\hr, \hs}}(\zeta_{\hr})$. 
\rcor{no extender disagreement} implies that if $\pi_{\a, \nu}^{\T_{\hq, \hs}}(\zeta_{\hq})=\pi_{\b, \iota_\hr}^{\T_{\hr, \hs}}(\zeta_{\hr})$, then $(\T_{\hq, \hs})_{\geq \nu}$ and $(\T_{\hr, \hs})_{\geq \nu}$\footnote{Recall that our least-extender-disagreement-coiterations allow padding.} are strictly above $\pi_{\b, \iota_\hr}^{\T_{\hr, \hs}}(\zeta_{\hr})$.
It follows that there cannot be $F$ as in (1.4), so we must have that\\\\
(2) $\pi_{\a, \nu}^{\T_{\hq, \hs}}(\zeta_{\hq})<\pi_{\b, \iota_\hr}^{\T_{\hr, \hs}}(\zeta_{\hr})$.\\\\
Then $\nu$ must be a $(\hq, \zeta)$-point.
Using \rcor{no extender disagreement} and the definition of $\nu$, we get that if $\nu'$ is the least on the main branch of $\T_{\hq, \hs}$ such that $\nu'>\nu$, then\\\\
(3) $\cp(\pi_{\nu, \nu'}^{\T_{\hq, \hs}})=\pi_{\a, \nu}^{\T_{\hq, \hs}}(\zeta_\hq)$ and $\pi_{\a, \nu'}^{\T_{\hq, \hs}}(\zeta_\hq)>\pi_{\b, \iota_\hr}^{\T_{\hr, \hs}}(\zeta_\hr)$.\\\\
Since $\S|\pi_{\a, \nu'}^{\T_{\hq, \hs}}(\zeta_\hq)=\M_{\nu'}^{\T_{\hq, \hs}}|\pi_{\a, \nu'}^{\T_{\hq, \hs}}(\zeta_\hq)$, we have\\\\
(4.1) (see \rcl{4.1}) $\M_{\iota_\hr}^{\T_{\hr, \hs}}|\pi_{\b, \iota_\hr}^{\T_{\hr, \hs}}(\zeta_\hr)\models ``\pi_{\a, \nu}^{\T_{\hq, \hs}}(\zeta_{\hq})$ is a strong cardinal,"\\
(4.2) (and hence) $\S\models ``\pi_{\a, \nu}^{\T_{\hq, \hs}}(\zeta_{\hq})$ is a ${<}\d_\hs$-strong cardinal," and hence\\
(4.3) $\M_{\nu'}^{\T_{\hq, \hs}}|\pi_{\a, \nu'}^{\T_{\hq, \hs}}(\zeta_\hq)\models ``\pi_{\a, \nu}^{\T_{\hq, \hs}}(\zeta_{\hq})$ is a strong cardinal."\\\\
But (4.3) implies that\\\\
(5) $\M_{\nu}^{\T_{\hq, \hs}}\models ``\pi_{\a, \nu}^{\T_{\hq, \hs}}(\zeta_\hq)$ is not a small cardinal."\\\\
Clearly (5) contradicts our hypothesis.

(4.1) easily follows from Claim \ref{4.1}.
[Claim \ref{4.1} and the initial segment condition \cite[p.~22]{SteelCom} imply that if $H$ is as in the claim, then for some $\gg<\im(H)$, $\gg>\pi_{\b, \iota_\hr}^{\T_{\hr, \hs}}(\zeta_\hr)$ and $H|\gg\in \vec{E}^{\M_{\nu'}^{\T_{\hp, \hq}}}$. (4.1) then follows because $\pi_{\b, \iota_\hr}^{\T_{\hr, \hs}}(\zeta_\hr)$ is a cardinal of $\M_{\nu'}^{\T_{\hp, \hq}}$.]

\begin{claim}\label{4.1}\normalfont Let $H$ be the extender used on the main branch of $\T_{\hq, \hs}$ at $\nu$. Then $\pi_{\b, \iota_\hr}^{\T_{\hr, \hs}}(\zeta_\hr)$ is a generator of $H$. 
\end{claim}
\begin{proof}
Set $\hw=\hm_\nu^{\T_{\hq, \hs}}$, $\hw'=\hm_{\nu'}^{\T_{\hq, \hs}}$, $\hw=(\W, \Sigma_\W)$ and $\hw'=(\W', \Sigma_{\W'})$.
We use $i$ for the embeddings given by $\T_{\hq, \hs}$ and $j$ for the embeddings given by $\T_{\hr, \hs}$. Let $\zeta_{\hw}=i_{\a, \nu}(\zeta_\hq)$ and $\zeta_{\hw'}=\iota_{\a, \nu'}(\zeta_\hq)$.
Towards a contradiction, assume that the claim is false.
Then $j_{\b, \iota_\hr}(\zeta_\hr)=i_{\nu, \nu'}(f)(a)$, where $f\in \W$ and $a\in [j_{\b, \iota_\hr}(\zeta_\hr)]^{<\omega}$.
It follows from \rlem{main lemma on small ordinals} that\\\\
(6.1) $\hs$ is an iterate of $\hw$ and $\T_{\hw, \hs}$ is above $\zeta_\hw$,\\
(6.2) $\hs$ is an iterate of $\hw'$ and $\T_{\hw', \hs}$ is above $\zeta_{\hw'}$,\\
(6.3) $(\T_{\hq, \hs})_{[\nu, \nu']}\inseg \T_{\hw, \hs}$, and\\
(6.4) $\hw'$ is on the main branch of $\T_{\hw, \hs}$, and $H$ is the first extender used on the main branch of $\T_{\hw, \hs}$.\\\\
Since $\cp(\pi_{\hw', \hs})>j_{\b, \iota_\hr}(\zeta_\hr)$, we have that\\\\
(7) $j_{\b, \iota_\hr}(\zeta_\hr)=\pi_{\hw, \hs}(f)(a)$.\\\\
The Hull Property from \rcor{hull property cor} implies that $f=\pi_{\hp, \hw}(h)(b)$, where $h\in \P$ and $b\in [i_{\iota_\hq, \nu}(\zeta_\hq)]^{<\omega}$.
Notice next that the following statements hold:\\\\
(8.1) $\T_{\hp, \hw}$ is the full normalization of $\T_{\hp, \hq}\oplus \T_{\hq, \hw}$,\\
(8.2) $\T_{\hp, \hs}$ is the full normalization of $\T_{\hp, \hw}\oplus \T_{\hw, \hs}$,\\
(8.3) if $\U$ is the longest initial segment of $\T_{\hp, \hw}$ such that $\gen(\U)\subseteq \zeta_\hw$ and $\hx=(\X, \Sigma_\X)$ is the last model of $\U$, then $\hw$ is a complete iterate of $\hx$ and $\T_{\hx, \hw}$ is strictly above $\zeta_\hw$, and\\
(8.4) if $\hx$ is as in (8.3), then $\hx$ is on the main branch of $\T_{\hp, \hs}$, $\hs$ is a complete iterate of $\hx$, $\cp(\pi_{\hx, \hs})=\zeta_\hw$, and $\pi_{\hx, \hs}\rest (\X|\zeta_{\hw}^+)^\X=\pi_{\hw, \hs}\rest \W|(\zeta_\hw^+)^\W$.\footnote{This is because $\T_{\hx, \hs}$ is the full normalization of $\T_{\hx, \hw}\oplus \T_{\hw, \hs}$ and $\pi_{\hx, \hs}=\pi_{\hw, \hs}\circ \pi_{\hx, \hw}$.}\\\\
It follows that for some $g\in \P$ and $c\in [j_{\b, \iota_\hr}(\zeta_\hr)]^{<\omega}$,\\\\
(9) $j_{\b, \iota_\hr}(\zeta_\hr)=\pi_{\hp, \hs}(g)(c)$.\\\\
Reverse engineering, and letting $\hk=\hm_{\iota_\hr}^{\T_{\hr, \hs}}$, we get that\\\\
(10.1) $\hs$ is a complete iterate of $\hk$ and $\T_{\hk, \hs}$ is above $j_{\b, \iota_\hr}(\zeta_\hr)$,\\
(10.2) $\cp(\pi_{\hk, \hs})=j_{\b, \iota_\hr}(\zeta_\hr)$ (see (1.4)),\\
(10.3) $\T_{\hp, \hk}$ is the full normalization of $\T_{\hp, \hr}\oplus \T_{\hr, \hk}$,\\
(10.4) $\T_{\hp, \hs}$ is the full normalization of $\T_{\hp, \hk}\oplus \T_{\hk, \hs}$, and\\
(10.5 ) $j_{\b, \iota_\hr}(\zeta_\hr)=\pi_{\hk, \hs}(\pi_{\hp, \hk}(g))(c)\in \rge(\pi_{\hk, \hs})$.\footnote{We make this conclusion by using an analysis similar to (8.3) and (8.4) above.}\\\\
(10.5) implies that $\cp(\pi_{\hk, \hs})\not = j_{\b, \iota_\hr}(\zeta_\hr)$, contradicting (10.2).
\end{proof}
This proves \rlem{no overlapping}.
\end{proof}

\subsection{Bounding preimages and closure points}
Lemma \ref{hull property} bounds pre-images of points in the direct limit.

\begin{lemma}\label{hull property}
Suppose $\hq=(\Q, \Lambda)$ is a hod pair, $\d$ ends an lsa block of $\Q$, and $\mu$ is a small properly overlapped cardinal of $\Q$ that is in the $\d$-block of $\Q$.
Suppose $A\in \powerset(\pi_{\hq, \infty}(\mu))\cap \M_\infty(\hq)$ and $\iota<\pi_{\hq, \infty}(\mu)$.
Then there is a complete iterate $\hr$ of $\hq$ such that $\hr$ catches both $A$ and $\iota$ and $\gen(\T_{\hq, \hr})\subseteq \mu_\hr$.
Moreover, if $\iota$ is either a ${<}\pi_{\hq, \infty}(\mu)$-strong cardinal of $\M_\infty(\hq)$ or a limit of such cardinals, then there is a complete iterate $\hr$ of $\hq$ catching $\iota$ and such that $\gen(\T_{\hq, \hr})\subseteq \iota_{\hr}$.  
\end{lemma}
\begin{proof} Given any complete iterate $\hr'$ of $\hq$ that catches $(A, \iota)$, letting $\U$ be the longest initial segment of $\T_{\hq, \hr'}$ such that $\gen(\U)\subseteq \mu_{\hr'}$ and letting $\hr$ be the last model of $\U$, it follows from \rlem{main lemma on small ordinals} that $\hr$ is as desired.  

Suppose now that $\iota$ is either a ${<}\pi_{\hq, \infty}(\mu)$-strong cardinal of $\M_\infty(\hq)$ or a limit of such cardinals.
\rlem{main lemma on small ordinals} implies that if $\hr'$ is a complete iterate of $\hq$ that catches $\iota$ and $\zeta$ is the least such that $\M_\zeta^{\T_{\hq, \hr'}}|\iota_{\hr'}=\M^{\hr'}|\iota_{\hr'}$, then $(\T_{\hq, \hr'})_{\geq \zeta}$ is strictly above $\iota_{\hr'}$.
It follows that if $\hr$ is the last model of $(\T_{\hq, \hr'})_{\leq \zeta}$, then $\hr$ catches $\iota$ and $\iota_{\hr}=\iota_{\hr'}$ (this is because $\hr'$ is a complete iterate of $\hr$ and $\cp(\pi_{\hr, \hr'})> \iota_{\hr'}$).
Since $\gen(\T_{\hq, \hr})\subseteq \iota_{\hr}$, $\hr$ is as desired.
\end{proof}

Lemma \ref{closed points equiv} characterizes closure points. 

\begin{terminology}\label{closure term}\normalfont
Suppose $\hp=(\P, \Sigma)$ is a hod pair, $\d$ is a Woodin cardinal of $\P$ that ends an lsa block of $\P$, and $\nu$ is a ${<}\d$-strong cardinal of $\P$. Let $\tau=(\nu^+)^\P$, $\hq=(\Q, \Sigma_\Q)$ be a complete iterate of $\hp$, and $X_\hq=\pi_{\hq, \infty}[\Q|\tau_\hq]$.
We say $\gg<\nu^\infty$ is $(\hq, \nu)$-\textbf{closed}, or $(X_\hq, \nu)$-\textbf{closed}, if   
\begin{center}
$\gg=\sup(\nu^\infty \cap Hull^{\M_\infty(\hp)|\tau^\infty}(X_\hq\cup \gg))$.
\end{center} 
We will often drop $\nu$ from the notation and say that $\gg$ is $X_\hq$-closed to mean that $\gg$ is $\hq$-closed.\footnote{The idea here is that being $\hq$-closed is a property that can be defined using $(X_\hq, \nu, \gg)$, and that $\hq$ is not needed.}
\end{terminology}

Suppose $(\hp, \nu, \d, \tau)$ are as in \rter{closure term}.
Notice that $\nu^\infty=\pi_{\hp, \infty}(\nu)$ is a limit of $(\hq, \nu)$-closed points and that if $\gg$ is $(\hq, \nu)$-closed, then if $\M_\gg$ is the transitive collapse of $Hull^{\M_\infty(\hq)|\tau^\infty}(X_\hq\cup \gg)$ and $\sigma: \M_\gg\rightarrow \M_\infty(\hq)|\tau^\infty$ is the uncollapse map then $\cp(\sigma)=\gg$.
It follows that $\M_\infty(\hq)\models ``\gg$ is properly overlapped."
In fact, we have the following characterization of $\hq$-closed points.
Given a complete iterate $\hq$ of $\hp$, we let $w_\hq=(\tau_\hq, \d_\hq)$ (this notation depends on $\tau$, but we will only use it below, and so we omit $\tau$).

\begin{lemma}\label{closed points equiv}
Suppose $(\hp, \nu, \d, \tau)$ are as in \rter{closure term}. Suppose $\hp'$ is a complete iterate of $\hp$. Then $\gg$ is $(\hp', \nu)$-closed if and only if  there is a complete iterate $\hr=(\R, \Sigma_\R)$ of $\hp'$ and an iterate $\hs=(\S, \Sigma_\S)$ of $\hr$ such that 
\begin{enumerate}
\item $\T_{\hr, \hs}$ is based on $w_\hr=(\tau_\hr, \d_\hr)$,
\item $\gen(\T_{\hp, \hr})\leq \nu_\hr$, and
\item there is an extender $E\in \vec{E}^{\S}$ such that $\cp(E)=\nu_\hr$, $\gen(\T_{\hr, \hs})\leq\im(E)$, and $\pi_{\hr_E, \infty}(\nu_\hr)=\gg$, where $\hr_E=(Ult(\R, E), \Sigma_{Ult(\R, E)})$.
\end{enumerate}
In particular, $\nu^\infty$ is a limit of $(\hp', \nu)$-closed points.  
\end{lemma}
\begin{proof} Set $X=X_{\hp'}$.\footnote{This notation depends on $\nu$ but we omit it as it is fixed.}
We first prove the right-to-left direction.
Suppose that $\gg=\pi_{\hr_E, \infty}(\nu_\hr)$ for some $(\hr, \hs, E)$ as above.
We claim that $\gg$ is $X$-closed.
Towards a contradiction, assume that $\gg<\sup(\nu^\infty\cap Hull^{\M_\infty(\hp)|\tau^\infty}(X\cup \gg))$, and let 
\begin{center}
$\gg'\in \nu^\infty\cap Hull^{\M_\infty(\hp)|\tau^\infty}(X\cup \gg)$
\end{center}
be such that $\gg'\geq \gg$.
Let $s\in X^{<\omega}$, $t\in \gg^{<\omega}$, and $\sigma$ be a term such that $\gg'=\sigma^{\M_\infty(\hp)|\tau^\infty}(s, t)$.
Set $\R_E=Ult(\R, E)$, and let $\hq=(\Q, \Sigma_\Q)$ be a complete iterate of $\hr_E=(\R_E, \Sigma_{\R_E})$ such that $t \in \rge(\pi_{\hq, \infty})$ (and hence, $\gg'\in  \rge(\pi_{\hq, \infty})$).
Setting
\begin{center}
$U=\{\sigma^{\R_E|\tau_{\R_E}}(s_{\hr_E}, v): v\in (\nu_\hr)^{<\omega}\}\cap \nu_{\hr_E}$,
\end{center} 
we have $U\subseteq \nu_\hr$, as $U\subseteq \rge(\pi_E^\R)$ (because $\pi_{\hr, \hr_E}(\nu_\hr)=\nu_{\hr_E}$).
Notice also that $U\in \R_E$.
However, since $\gg_\hq=\pi_{\hr_E, \hq}(\nu_\hr)$, we have that 
\begin{center}
$\Q\models \exists \iota\in [\pi_{\hr_E, \hq}(\nu_\hr), \nu_{\hq}) \exists v \in (\pi_{\hr_E, \hq}(\nu_\hr))^{<\omega} ( \iota=\sigma^{\Q|\tau_{\hq}}(s_{\hq}, v))$.
\end{center}
Thus, $\Q\models \pi_{\hr_E, \hq}(U)\cap [\pi_{\hr_E, \hq}(\nu_\hr), \nu_\hq)\not=\emptyset$, contradiction.

Suppose next that $\gg$ is an $X$-closed point.
Let $\hq$ be a complete iterate of $\hp'$ such that $\gg\in \rge(\pi_{\hq, \infty})$.
We must have that $\gg_\hq<\nu_\hq$.
For $\b<\lh(\T_{\hp', \hq})$, let $\hp_\b=\hm_\b^{\T_{\hp', \hq}}$.
We assume that for no $\b+1<\lh(\T_{\hp', \hq})$ such that $[0, \b)_{\T_{\hp', \hq}}\cap D^{\T_{\hp', \hq}}=\emptyset$, $\gg\in \rge(\pi_{\hp_\b, \infty})$.
It follows that if $\b+1=\lh(\T_{\hp, \hq})$, then $\b=\b'+1$, and if $\zeta$ is the $\T_{\hp, \hq}$-predecessor of $\b$, then $\cp(E_{\b'}^{\T_{\hp, \hq}})=\nu_{\hp_\zeta}$ (we use smallness of $\nu$ here, see \rlem{main lemma on small ordinals}).
Let $\R=\M_\zeta^{\T_{\hp, \hq}}$, $\S=\M_{\b'}^{\T_{\hp, \hq}}$, and $E=E_{\b'}^{\T_{\hp, \hq}}$.
Let $\hr=(\R, \Sigma_\R)$ and $\hs=(\S, \Sigma_\S)$.

There are two cases.
If $\gg_\hq=\nu_\hr$, then there is nothing to prove, as $(\hr, \hs, E)$ are as desired.
Suppose then that $\gg_\hq\not =\nu_\hr$.
As $\gg\not \in \rge(\pi_{\hr, \infty})$, we have $\gg_\hq\in (\nu_\hr, \nu_\hq)$.
Because $\gg$ is $X$-closed, it must be that $\gg_\hq<\gen(E)$, as otherwise $\gg_\hq=\pi_{\hr, \hq}(f)(s)$, where $s\in \gg_\hq^{<\omega}$ and $f\in \pi_{\hq, \infty}^{-1}[X]$.
It follows that $\gg_\hq$ must be a generator of $E$ and hence,
\begin{center}
$\M_\infty(\hp)\models ``\gg$ is overlapped by an extender with critical point $\pi_{\hq, \infty}(\nu_\hr)"$.
\end{center}
Because $\gg$ is properly overlapped in $\M_\infty(\hp)$, we have that $\M_\infty(\hp)\models ``\pi_{\hq, \infty}(\nu_\hr)$ is ${<}\d^\infty$-strong cardinal."
This contradicts the smallness of $\nu_\hr$ in $\R$, as we have that $\R\models ``\nu_\hr$ is a limit of ${<}\d_\hr$-strong cardinals." 
\end{proof}

Essentially the same proof establishes Lemma \ref{catching closed points}.

\begin{lemma}\label{catching closed points}
Suppose $(\hp, \nu, \d, \tau)$ are as in \rter{closure term}.
Suppose that $\hp'$ is a complete iterate of $\hp$ and $X=X_{\hp'}$. Suppose $\gg$ is a $(\hp', \nu)$-closed point and $\hq=(\Q, \Sigma_\Q)$ is a complete iterate of $\hp'$ such that $\sup(\pi_{\hq, \infty}[\nu_\hq])\leq \gg$.
Let $\hs=(\S, \Sigma_\S)$ be a complete iterate of $\hq$ such that $\gg\in \rge(\pi_{\hs, \infty})$.
Let $\R$ be the least node on the main branch of $\T_{\hq, \hs}$ such that $\R|\gg_\hs=\S|\gg_\hs$, and let $\iota<\lh(\T_{\hq, \hs})$ be such that $\R=\M_\iota^{\T_{\hq, \hs}}$. Let $\hr=\hm_\iota^{\T_{\hq, \hs}}$.

Then $\gg_\hs=\nu_\hr$, and there is $\b<\lh(\T_{\hq, \hs})$ such that $\b+1$ is on the main branch of $\T_{\hq, \hs}$, $\T_{\hq, \hs}(\b+1)=\iota$, and $\cp(E_\b^{\T_{\hq, \hs}})=\nu_\hr$.
\end{lemma}

\section{Nairian models derived from hod mice}\label{sec: kom models}

In this section, we
\begin{enumerate}
    \item prove a useful result about the Vop\v{e}nka algebra (\rlem{vopenka algebra}) and use it to establish a result showing that the different Nairian models agree on the powerset operation (see \rthm{cating the powerset}, whose main proof ideas we use several times),

    \item construct a Nairian model of ${\sf{ZF}}+``\omega_1$ is a supercompact cardinal" (see \rthm{first prop} and \rcor{zf in c}),

    \item show that strong cardinals of $\H$ are successor and regular cardinals of the Nairian model (see  \rthm{first prop1} and \rthm{strongs are successors}), and

    \item show that the cofinality of strong cardinals of $\H$ is big in the parent determinacy model (\rthm{cof is kappa}).
\end{enumerate}


\subsection{Additional notation for \textsection\ref{sec: kom models}}\label{notation for kom model}
We use the following notation in addition to \rnot{iteration terminology} and \rter{catches} throughout this section.
We will be working inside the derived model of some hod mouse.
We assume as our working hypothesis that $\V\models {\sf{ZFC}}$ is a hod premouse\footnote{We tacitly assume that all large cardinal notions are witnessed by the extenders on the extender sequence of hod premouse.} and $\xi\leq \d<\eta$ are such that 
	\begin{enumerate}
		\item $\eta$ is an inaccessible limit of Woodin cardinals of $\V$,
		\item $\d$ is a Woodin cardinal of $\V$,
		\item $\xi$ is a limit of Woodin cardinals of $\V$ such that there is $\k\leq \xi$ with the property that $\k$ is $<\d$-strong cardinal.
	\end{enumerate}
	Let $g\subseteq Coll(\omega, <\eta)$ be $\V$-generic and let $M$ be the derived model of $\V$ as computed by $g$. Set
	\begin{enumerate}
		\item $\P=\V|(\eta^+)^\V$,
		\item $\Sigma\in \V$ be the iteration strategy of $\P$ indexed on the sequence of $\V$,
        \item $\hp=(\P, \Sigma)$,
		\item $\mathcal{F}=\mathcal{F}^g_\hp$ (see \rsubsec{sec: chang model}),
		\item $\mH=\M_\infty(\hp)$,
        \item $\Omega=\Sigma_{\mH}$,  
		\item $\xi^\infty=\pi_{\hp, \infty}(\xi)$ and $\d^\infty=\pi_{\hp, \infty}(\d)$.
	\end{enumerate}
	We let $\c^-=(\c^-_{\xi^\infty})^M$, $\c=(\c_{\xi^\infty})^M$ and $\c^+=(\c_{\xi^\infty}^+)^\M$.
    We treat $\hp$ as a hod pair in $\V$ and in $\V[g]$. Observe that if $\a<\eta$ then $\hp|\a\in M$. 

\subsection{Bounding the Solovay sequence}
We study the Solovay sequences of various Nairian models under the working hypothesis stated in \rnot{notation for kom model}.


\begin{definition}\label{x-solovay sequence}\normalfont
	Suppose $X_0, X_1$ are any two sets. We then let $\theta_0(X_0, X_1)$ be the supremum of ordinals $\b$ such that there is an $OD(X_1)$\footnote{This means that $X_1$ can be used as a parameter.} surjection $f: X_0\rightarrow \b$.
\end{definition}

While we will apply \rdef{x-solovay sequence} inside determinacy models, the definition also makes sense in $\ZFC$.

\begin{lemma}\label{chain condition for vopenka1}
	Suppose $\nu<\l$, $\hq=(\Q, \Lambda)$ is a complete iterate of $\hp$ such that $\nu\in \rge(\pi_{\hq, \infty})$, and $\mu$ is the least Woodin cardinal of $\mH$ that is greater than $\nu$.
    Set $Z=\pi_{\hq, \infty}[\Q|\nu_\hq]$.
    Then \[M\models ``\mu \text{ is not a surjective image of an $OD(Z)$ function with domain }\powerset_{\omega_1}(\nu)."\]
    Thus, \[M\models \theta_0(\powerset_{\omega_1}(\nu), Z)\leq \mu.\]
\end{lemma} 
\begin{proof}
Assume not and suppose $f:\powerset_{\omega_1}(\nu)\rightarrow \mu$ is an $OD(Z)^M$ surjection.
We furthermore assume that $f$ is the least such $OD(Z)^M$ function. Let $h\subseteq Coll(\omega, \Q|\nu_\hq)$ be $\Q$-generic.
	
	Let $\mathcal{A}\subseteq (Coll(\omega, \Q|\nu_\hq)*{\sf{EA}}_{\mu_\hq})^\Q$ be a maximal antichain consisting of pairs $(p_0, p_1)\in \mathcal{A}$ such that there is $\b<\mu_\hq$ with the property that $(p_0, p_1)$ forces the following statement:\\\\
	($*_\b$) If $k*(x_0, x_1)$ is the generic, then in $\Q[k*(x_0, x_1)]$, $(x_0, x_1)\in \bR^2$ and, letting for $m\in 2$, $Y_m=\{ \pi_{\hq, \infty}(k(i)): i\in x_m\},$\footnote{Our intended reader knows the arguments why sets like $Y_m$ are definable over $\Q$.
    Such arguments are part of the basic hod analysis.
    In particular, the material explained in \rsubsec{sec: chang model} implies that $\Q$ \textit{knows} $\pi_{\hq, \infty}\rest (\Q|\nu_\hq)$, which is the portion of $\pi_{\hq, \infty}$ that is needed to define $Y_0$ and $Y_1$.
    This is because we can find a complete iterate $\hq'$ of $\hq$ such that $\T_{\hq, \hq'}$ is above $\d_\hq$, and $M$ can be realized as the derived model of $\M^{\hq'}$ at $\eta_{\hq'}$. But then, since $\pi_{\hq', \infty}\in \M^{\hq'}$ and $\cp(\pi_{\hq, \hq'})>\mu_\hq$, we have that $\pi_{\hq, \infty}\rest (\Q|\mu_\hq)\in \rge(\pi_{\hq, \hq'})$.}
    \begin{itemize}
        \item[(a)]  there is an $OD(Y_0)^M$ surjection $h:\powerset_{\omega_1}(\nu)\rightarrow \mu$, and

        \item[(b)] if $f_{Y_0}:\powerset_{\omega_1}(\nu)\rightarrow \mu$ is the least $OD(Y_0)^M$ surjection then $f_{Y_0}(Y_1)=\pi_{\hq, \infty}(\b)$.
    \end{itemize}

	For $p\in \mathcal{A}$, let $\b_p< \mu_\hq$ be the unique $\b$ as above.
    We have that $\card{\mathcal{A}}^\Q<\mu_\hq$, since the extender algebra is $\mu_\hq$-cc.
	
	It is not hard to show that $\{\b_p:p\in \mathcal{A}\}=\mu_\hq$. Indeed, fix $\b<\mu_\hq$. Let $\hr=(\R, \Psi)$ be a complete iterate of $\hq$  such that for some $\R$-generic $n\subseteq Coll(\omega, \R|\nu_\R)$ and some $\R[n]$-generic $(x_0, x_1)\subseteq {\sf{EA}}_{\mu_\R}$, setting for $m\in 2$, $Y_m=\{\pi_{\hr, \infty}(n(i)): i\in x_m\}$, $Y_0=Z$, and $\pi_{\hq, \infty}(\b)=f(Y_1)$.
    It readily follows that for some $p\in \pi_{\hq, \hr}(\mathcal{A})$, $\pi_{\hq, \hr}(\b)=(\b_p)^\R$.
    Hence, for some $p\in \mathcal{A}$, $\b=\b_p$.
    Contradiction.
\end{proof}

The following corollary also follows from the results of \cite{MPSC}. In particular see \cite[Theorem 5.3]{MPSC}. 

\begin{corollary}\label{corollary to lemma}
Suppose $\nu<\xi^\infty$ and $\mu$ is the least Woodin cardinal of $\mH$ that is greater than $\nu$. Then for any $X\in \powerset_{\omega_1}(\nu)$, $\theta_0(\powerset_{\omega_1}(\nu), X)\leq \mu$. 
\end{corollary}
\begin{proof}
Fix an $X\in \powerset_{\omega_1}(\nu)$.
We want to see that $\theta_0(\powerset_{\omega_1}(\nu), X)\leq \mu$. We can fix a complete iterate $\hq=(\Q, \Lambda)$ of $\hp$ such that $X\subseteq \pi_{\hq, \infty}[\Q|\nu_\hq]$.
Let $Z=\pi_{\hq, \infty}[\Q|\nu_\hq]$, and fix a pair of reals $(u_0, u_1)$ such that $u_0$ codes a function $f:\omega\rightarrow \Q|\nu_\hq$ and $X=\{\pi_{\hq, \infty}(f(i)): i\in u_1\}$.
Since $X$ is ordinal definable from $(Z, u_0, u_1)$,\footnote{Notice that $Z\prec \mH|\nu$, and so $Z$ determines $\pi_{\hq, \infty}\rest (\Q|\nu_\hq)$, as this is the inverse of the transitive collapse of $Z$.} it is enough to prove that $\theta_0(\powerset_{\omega_1}(\nu), (Z, u_0, u_1))\leq \mu$, or just that for any real $u$, $\theta_0(\powerset_{\omega_1}(\nu), (Z, u))\leq \mu$. 
	
	To see this, suppose not and let $p$ be a finite set of ordinals and $\phi$ be a formula such that, letting $g:\powerset_{\omega_1}(\nu)\rightarrow \mu$ be given by
	\begin{center}
		$g(Y)=\b \iff \phi[Z, u, p, Y, \b]$,
	\end{center}
	$\rge(g)=\mu$.
    Define $h:\powerset_{\omega_1}(\nu)\rightarrow \mu$ by setting $h(Y)=\b$ if and only if, letting $Y\cap \omega=r$, $\phi[Z, r, p, Y, \b]$ holds.
    Then $h$ is ordinal definable from $Z$ and $\rge(h)=\mu$, contradicting \rlem{chain condition for vopenka1}.
\end{proof}

\begin{lemma}[Vopenka Algebra]\label{vopenka algebra}
Suppose $\nu<\xi^\infty$ and $X\in \nu^\omega$.
Let $\mu$ be the least Woodin cardinal of $\mH$ such that $\nu<\mu$. Then $X$ is generic over $\mH$ for a poset of size $\mu$. 
\end{lemma}
\begin{proof}
	Consider the Vopenka algebra for adding a member of $\powerset_{\omega_1}(\nu)$ to $\mH$. 
    \rlem{chain condition for vopenka1} implies that this algebra has the $\mu$-cc in $\mH$.
    Indeed, if $\mathcal{A}\in \mH$ is an antichain in the aforementioned Vopenka algebra such that $\card{\mathcal{A}}^\mH=\mu$, then we can find an $OD$ sequence $(A_i: i<\mu)$ such that $A_i\subseteq \powerset_{\omega_1}(\nu)$ and, for $i\not =i'$, $A_i\cap A_{i'}=\emptyset$.
    Define $h:\powerset_{\omega_1}(\nu)\rightarrow \mu$ by letting $h(Y)$ be the least $i$ such that $Y\in A_i$, if there is such an $i$, and otherwise set $h(Y)=0$.
    Then $h$ is an $OD$-surjection, contradicting \rlem{chain condition for vopenka1}.
	
	Because $\mu$ is an inaccessible cardinal in $\mH$, we can find $\mathbb{P}\in \mH$ of size at most $\mu$ with $X\in \mH[G]$ for some $G\subseteq \mathbb{P}$. 
    Indeed, work in $\mathcal{H}$.
    Let $\mathbb{Q}$ be the Vopenka algebra. 
    Let $f:\omega\rightarrow X$ be an enumeration of $X$, and let $\sigma$ be a $\mathbb{Q}$-name for $f$ such that for some sequence of maximal $\mathbb{Q}$-antichains $(A_n: n<\omega)$, $\sigma$ consists of tuples $(p, \check{n}, \check{\alpha})$ such that $p\in A_n$. 
    Let $\mathbb{P}_0=\cup_{n<\omega} A_n$. 
    We can then inductively construct $(\mathbb{P}_\a:\a\leq \mu)$ such that
    \begin{enumerate}
        \item for each $\a<\mu$, $\card{\mathbb{P}_\a}<\mu$ and $\mathbb{P}_\a$ is a sub-poset of $\mathbb{Q}$,

        \item for $\a<\b$, $\mathbb{P}_\a\subseteq \mathbb{P}_\b$,

        \item for $\zeta\leq \mu$ a limit ordinal, $\mathbb{P}_\zeta=\bigcup_{\a<\zeta}\mathbb{P}_\a$, and

        \item if $A\subseteq \mathbb{P}_\mu$ is a maximal antichain, then $A$ is also a maximal antichain in $\mathbb{Q}$.
    \end{enumerate}
    Then $\mathbb{P}_\mu$ is the desired $\mathbb{P}$. 
    In fact, if $G\subseteq \mathbb{Q}$ is $\mH$-generic and is such that $\sigma_G=f$, then $G'=G\cap \mathbb{P}$ is also $\mH$-generic and $\sigma_{G'}=f$.
\end{proof}

\subsection{The agreement of the powersets}
The following lemma is due to Woodin, although it is easier to prove in our context.

\begin{lemma}\label{hod[x] is hodx}
Suppose $\b<\xi^\infty$ and $X\in \powerset_{\omega_1}(\b)$.
Then $\mH[X]=({\sf{HOD}}_X)^M$. 
In particular, $V_{\xi^\infty}^{\mH[X]}=V_{\xi^\infty}^{({\sf{HOD}}_X)^M}$.
\end{lemma}
\begin{proof} 
Recall that $M$ can be embedded into the derived model of $\H$ at $\Theta$,\footnote{This is a general theorem of Woodin, but in our context it follows from $\H$ analysis. However, see \cite{sargsyan2021ideals}.} and moreover, there is such an embedding $j^*$ such that $j=_{def}j^*\rest \H\in \H$.\footnote{$j$ is really just the $\mH$-to-$j^*(\mH)$ iteration embedding. $j^*$ is then given by mapping the homogenity systems via $j$.
More precisely, if $B\subseteq \bR$ is a set in $M$ and $\bar{\mu}$ is a homogeneity system such that $B=S_{\bar{\mu}}^M$, then for $s\in \omega^{<\omega}$, letting $k_s=\pi_{\bar{\mu}_s}\rest \mH$, $j^*(B)$ is the set of $y$ such that the direct limit of $(j(k_{y\rest m}): m\in \omega)$ is well-founded.}  
	
	Clearly $\mH[X]\subseteq ({\sf{HOD}}_X)^M$.
    Let $A$ be an $OD[X]$-definable set of ordinals.
    Notice that $j^*(X)=j^*[X]=j[X]$, and so $j[X]\in \mH[X]$. Therefore, $j^*(A)\in \mH[X]$.
    $A$ is definable in $\mH[X]$ via the equivalence $\b\in A\iff j(\b)\in j^*(A)$. 
\end{proof}

\rthm{cating the powerset} is the first major step towards the proof of \rthm{main theorem}.
It shows that different Nairian models compute the powerset operation correctly. 
(Below, we use $\mH$ for $\H$.)

\begin{theorem}\label{cating the powerset}
For every $\b<\xi^\infty$, \[\powerset(\b^\omega)\cap \c^+=\powerset(\b^\omega)\cap \c=\powerset(\b^\omega)\cap \c^-.\]
\end{theorem}
\begin{proof}
	Fix $A\subseteq \powerset(\b^\omega)\cap \c^+$. Let $\gg<\xi^\infty$ be such that $A$ is $OD(X)$ for some $X\in \powerset_{\omega_1}(\gg)$. We can assume without losing generality that $\gg\geq \b$. Let $(j, j^*)$ be as in the proof of \rlem{hod[x] is hodx}.
	
	Now $\H_X$ is an infinity-Borel code for $A$, i.e., there is a formula $\phi$ and a finite sequence of ordinals $s$ such that 
    \[Y\in A\iff \H_X[Y]\models \phi[s, X, Y, j\rest (\mH|\Theta)].\]
    To see this, let $s\in {\sf{Ord}}^{<\omega}$ be such that $A$ is definable from $(X, s)$, and let $\psi(u_0, u_1, u_2)$  be the formula defining $A$ from $(X, s)$ over $M$. Thus, $Y\in A\iff M\models \psi[s, X, Y]$.  We then have that for any $Y\in \powerset_{\omega_1}(\b)$,
	\begin{center}
		$Y\in A \iff \H_X[Y]\models ``\psi[j(s), j[X], j[Y]]$ holds in the derived model at $\Theta$."
	\end{center}
	The equivalence holds because, since $X$ and $Y$ are in $\powerset_{\omega_1}(\gg)$, we have $j^*(X)=j[X]$ and $j^*(Y)=j[Y]$.
    We can now find a formula $\phi$ such that 
	\[Y\in A \iff \mH[X][Y]\models \phi[s, X, Y, j\rest (\mH|\Theta^M)].\]
	Let $\mu$ be the least Woodin cardinal of $\mH$ that is above $\gg$.
    It follows from \rlem{vopenka algebra} that if $Y\in \powerset_{\omega_1}(\b)$, then $Y$ can be added to $\mH[X]$ by a poset of size at most $\mu$. Working in $M$, let $\zeta=((\Theta^+))^{\H}$ and $i: \M\rightarrow H_\zeta^{\H}$ be such that 
	\begin{enumerate}
		\item $(i, \M)\in {\H}$ and $\{s, j\rest (\mH|\Theta^M)\}\in \rge(i)$, 
		\item $\cp(i)>(\mu^+)^{\H}$, and 
		\item $\card{\M}^\H=(\mu^+)^\H$.
	\end{enumerate}
	Then every $Y\in \powerset_{\omega_1}(\b)$ can be added to $\M$ by a poset of size at most $\mu$, and for each such $Y$, $i$ extends to $i_Y:\M[Y]\rightarrow H_\zeta^{\H}[Y]$. Let $k=i^{-1}(j\rest (\mH|\Theta^M)$ and $t=i^{-1}(s)$.
	It now follows that 
	\[Y\in A \iff \M[X][Y]\models \phi[t, X, Y, k].\]
	Therefore, $A\in \c^-$.
\end{proof}


\begin{corollary}\label{strategy is in the Nairian}
For $\a<\b<\xi^\infty$, let $\Omega^{\a, \b}=\Omega_{\mH|\a}\rest \c^-_{\b}$. Then the function $(\a, \b)\mapsto\Omega^{\a, \b}$ is definable over $(\c^-, \mH|\xi^\infty, \in)$. 
\end{corollary}
\begin{proof} This is an application of generic interpretability (see \cite{SteelCom} and \cite{HMMSC}). Fix $\a<\b<\xi$.
Define $\Omega^{\a, \b}$ uniformly in $(\a, \b)$ over $(\c^-, \mH|\xi^\infty, \in)$.
Let $\K=\mH|\xi^\infty$.
Let $\T\in \c^-$ be a normal iteration tree according to $\Omega^{\a, \b}$. We then set $\Omega^{\a, \b}(\T)=b$ if and only if $\K[\T]\models S^{\K}(\T)=b$.\footnote{$S^\M$ is the internal strategy predicate of $\M$. We use $S^\M$ also for its interpretation onto generic extensions.}
Our definition makes sense because $\T$ is ordinal definable from some $X\in \powerset_{\omega_1}(\xi^\infty)$, and so $\T\in \K[X]$ implying that $\K[\T]$ makes sense. 
\end{proof}

\rthm{cating the powerset} can be used to show e.g.~that in the full Nairian model, the successor of the first $\omega$-many strongs is not a strong cardinal, see \cite{blue202komitas}.

\subsection{Regularity of $\xi^\infty$}


\begin{theorem}\label{lambda is regular}
Suppose $\xi$ is an inaccessible cardinal in $\P$.
Then $\c^+\models \cf(\xi^\infty)=\xi^\infty$.
\end{theorem} 

\begin{proof}
	Suppose $f:\b\rightarrow \xi^\infty$ is an increasing cofinal function such that $f\in \c^+$.
    We want to see that $\b=\xi^\infty$. 
    Suppose that $\b<\xi^\infty$.
    Fix some $\b'<\xi^\infty$ and $X\in \powerset_{\omega_1}(\M_\infty|\b')$ such that $f$ is ordinal definable in $M$ from $X$. 
	Without losing generality, we may assume that $M\models ``f$ is the $OD(X)$-least surjection."
	
	Fix $\hq'\in \mathcal{F}$ such that $\hq'$ catches $\b, \b'$ and $X\subseteq \pi_{\hq', \infty}[\Q'|\b'_{\hq'}]$ where $\hq'=(\Q', \Lambda')$.
    Let $\zeta_0<\zeta'_1<\xi_{\hq'}$ be two Woodin cardinals of $\Q'$ such that $\b'_{\hq'}<\zeta_0$ (and hence, $X\subseteq \pi_{\hq', \infty}[\Q'|\zeta_0]$). 
    Setting $Z= \pi_{\hq', \infty}\rest (\Q'|\zeta_0)$, let $(w_0, w_1)\in \bR^2$ be a $Z$-\textit{code} of $X$. 
    More precisely, $w_0$ codes a surjection $k: \omega\rightarrow \Q'|\zeta_0$ and $X=\{Z(k(i)): i\in w_1\}$. 
    Let $\hq''=(\Q'', \Lambda'')$ be a complete iterate of $\hq'$ obtained by iterating in the interval $(\zeta_0, \zeta_1')$ so that  $(w_0, w_1)$ is generic over $\hq''$ for the extender algebra ${\sf{EA}}^{\Q''}_{[\zeta_0, \pi_{\hq', \hq''}(\zeta'_1)]}$.\footnote{This is the extender algebra at $\pi_{\hq', \hq''}(\zeta_1')$ that uses extenders with critical points $>\zeta_0$.}
	Finally, set $\zeta_1=\pi_{\hq', \hq''}(\zeta_1')$, and let $\hq=(\Q, \Lambda)$ be a genericity iterate of $\hq''$ obtained by iterating above $\zeta_1$ (see \rdef{gen iterate}).
	
	\begin{lemma}\label{x in q}
		$X\in \Q[w_0, w_1]$ and $\Q[w_0, w_1]\models ``\xi_\hq$ is an inaccessible limit of Woodin cardinals". 
	\end{lemma} 
	\begin{proof}
    The second claim is a consequence of the fact that $(w_0, w_1)$ is added to $\Q$ by a forcing whose $\Q$-cardinality is $<\xi_\hq$.
    The first is an immediate consequence of \rprop{invariance} and the fact that \[\pi_{\hq', \infty}\rest (\Q'|\zeta_0)=\pi_{\hq, \infty}\rest (\Q|\zeta_0).\]
	\end{proof}
	
	It follows from \rprop{invariance} that the derived model of $\Q$ at $\eta_\Q$ is $M$.
    It then follows from \rlem{x in q} that $f\in \Q[w_0, w_1]$. 
	
	Now work inside $\Q$.
    Let $u=(\zeta_0, \zeta_1)$, and let $D$ be the set of pairs $(p, \a)$ such that 
	\begin{enumerate}
		\item $p\in \sf{EA}_{u}^\Q$,
		\item $\a<\b_\hq$, and
		\item $p$ forces that the generic pair $(x_0, x_1)\in \bR^2$ is such that $x_0$ codes a surjection $h_{x_0}:\omega\rightarrow \Q|\zeta_0$.  
	\end{enumerate}
	Define $k: D\rightarrow \xi_\hq$ as follows. 
    Given $(p, \a)\in D$, set $k(p, \a)=\gg$ just in case $p$ forces that if $(x_0, x_1)$ is the generic pair then, letting
    \begin{itemize}
        \item[(a)] $U=\{\pi_{\hq, \infty}(h_{x_0}(i)): i\in x_1\}$, and

        \item[(b)] $f_U:\b\rightarrow \xi^\infty$ be such that $M\models ``f_U:\b\rightarrow \xi^\infty$ is the least $OD(U)$ increasing and cofinal function",
    \end{itemize}
    $\pi_{\hq, \infty}(\gg)=f_U(\pi_{\hq, \infty}(\a))$.
    If $p$ does not force the above, then set $k(p, \a)=0$.
    
	We claim that $k$ is cofinal. To see this, fix $\gg<\xi_\hq$. We want to show that there is $(p, \a)\in D$ such that $k(p, \a)\geq \gg$. Let $\hr=(\R, \Psi)$ be a complete iterate of $\hq$ such that for some $\a'<\b_\hr$, $\pi_{\hq, \infty}(\gg)\leq f(\pi_{\hr, \infty}(\a'))$. 
	\begin{lemma}\label{b in m}
		$\R\models ``\pi_{\hq, \hr}(\gg)\leq \sup(\rge(\pi_{\hq, \hr}(k)))."$
	\end{lemma}
 
	\begin{proof}
	Fix $(x_0, x_1)$ such that $x_0$ codes a surjection $n: \omega\rightarrow \R|\pi_{\hq, \hr}(\zeta_0)$ and $X=\{ \pi_{\hr, \infty}(n(i)): i\in x_1\}$. 
    Let $\hs=(\S, \Phi)$ be a complete iterate of $\hr$ based on the window $\pi_{\hq, \hr}(u)$ such that the pair $(x_0, x_1)$ is generic over $\S$ for $\sf{EA}_{\pi_{\hq, \hs}(u)}^{\S}$.
    We can find a complete iterate $\hs'=(\S', \Psi')$ of $\hs$ that is obtained by iterating above $\d_{\hs}$ and is a genericity iterate of $\P$ (see \rdef{gen iterate}).
    We then have that
    \begin{enumerate}
        \item $\pi_{\hq, \hs}(\gg)=\pi_{\hq, \hs'}(\gg)$ and $\pi_{\hq, \hs}(k)=\pi_{\hq, \hs'}(k)$, and

        \item $\S'\models ``\pi_{\hq, \hs'}(\gg)\leq \sup(\rge(\pi_{\hq, \hs'}(k)))"$.
    \end{enumerate}
	Therefore, $\S\models ``\pi_{\hq, \hs}(\gg)\leq \sup(\rge(\pi_{\hq, \hs}(k))),"$ which implies \[\R\models ``\pi_{\hq, \hr}(\gg)\leq \sup(\rge(\pi_{\hq, \hr}(k)))."\]
	\end{proof}
	
	But pulling back to $\Q$, we have that $\Q\models ``\gg\leq \sup(\rge(k))"$.
    Thus, $k$ is cofinal.
    But $\Q\models ``\xi_\hq$ is an inaccessible cardinal." Contradiction.
\end{proof}

\begin{corollary}\label{towards zf} 
Suppose $\xi$ is an inaccessible cardinal in $\P$.
Then $\c^-\models \ZF-{\sf{Powerset}}$.
\end{corollary}

\begin{proof} 
It is enough to prove that $\c^-\models {\sf{Replacement}}$. 
To show this, we will in fact show that if $B\in \c^-$ and $f: B\rightarrow \c^-$ is in $\c^+$, then $\rge(f)\in \c^-$.
This easily follows from the regularity of $\xi^\infty$. 
Indeed, because $B\in \c^-$, for some $\b<\xi^\infty$, we have a surjection $h: \powerset_{\omega_1}(\b)\rightarrow B$ with $h\in \c^-$.\footnote{First find $\b<\xi^\infty$ such that $B\in \c^-_\b$.
Next let  $(\a_i: i<\omega)$ be the first $\omega$ indiscernibles of $\c_\b$.
Notice that whenever $y\in \c^-_\b$, there is an $n<\omega$ and $Z\in \b^\omega$ such that $y$ is definable from $((\a_i: i\leq n), Z, \mH|\b)$ in $\c_\b$.
For $n, m\in \omega$, let $h_{n, m}:\b^\omega\rightarrow B$ be such that letting $s_n=(\a_0,..., \a_n)$ and $\phi_m$ be the formula coded by $m$, for every $U\in \dom(h_{n, m})$, $h_{n, m}(U)=a$ if and only if $a$ is the unique $z$ such that $\c_\b\models \phi_m[z, U, s_n, \mH|\b]\wedge z\in B$.
We can now code the sequence $(h_{n, m}: (n, m)\in \omega^2)$ into one surjection $h:\b^\omega\rightarrow B$.}
Let $k:\powerset_{\omega_1}(\b)\rightarrow \xi^\infty$ be such that for each $Y\in \powerset_{\omega_1}(\b)$, $k(Y)$ is the least ordinal $\tau$ such that $f(h(Y))\in \c^-_\tau$.
Fix some $\gg<\xi^\infty$ and $X\in \powerset_{\omega_1}(\gg)$ such that in $M$, $h$ is ordinal definable from $X$, and  consider the prewellordering $\leq^*$ of $\powerset_{\omega_1}(\b)$ given by $Y\leq^* Z$ if and only if $k(X)\leq k(Y)$.
We have that the length of $\leq^*$ is less than $\xi^\infty$ by \rcor{corollary to lemma}.
Because $\xi^\infty$ is regular in $\c^+$, it follows that for some $\a<\xi^\infty$, we have $\rge(f)\subseteq \c^-_\a$.

Now fix a surjection $g:\powerset_{\omega_1}(\a)\rightarrow \c^-_\a$ with $g\in \c^-$, and consider the set $A$ consisting of those $Z\in \a^\omega$ such that, letting $Z_0=\{ Z(2i): i\in \omega\}$ and $Z_1=\{ Z(2i+1): i\in \omega\}$, $f(h(Z_0))=h(Z_1)$.
Then for some $\a'<\xi^\infty$, $A$ is ordinal definable from some $U\in \powerset_{\omega_1}(\a')$, and hence $A\in \c^-$ (see \rthm{cating the powerset}).
Therefore, $\rge(f)\in \c^-$.
\end{proof}

\subsection{Producing cardinals in the Nairian model}
Our next aim is to show that the Nairian model has non-trivial cardinal structure above $\Theta$. The reader may benefit from reviewing Definition \ref{mu bounded iterations} and Remark \ref{remark about infty notation}.

\begin{lemma}\label{main lemmA}
Suppose 
\begin{enumerate}
\item $\hr=(\R, \Psi)$ is a hod pair such that $\M_\infty(\hr)=\mH|\d^\infty$,
\item $\nu$ is a $<\d^\infty$-strong cardinal of $\mH$ or a limit of such cardinals,
\item $\nu$ is small,
\item $\hr$ catches $\nu$,
\item for some $\b\in (\d^\infty, \Theta)$, $\Delta$ is the set of all sets of reals that are Wadge below $\b$ (in $M$),
\item $\a=(\Theta^+)^{L(\Delta)}$, and
\item $\psi: N\rightarrow L_\a(\Delta)$ is such that $N$ is countable and transitive, $\psi$ is elementary, and $\{\hr, \nu\}\in \rge(\psi)$.
\end{enumerate}
Let 
\begin{enumerate}[resume]
\item $\hr_N=\psi^{-1}(\hr)$,
\item $\hr^\nu$ and $\R^\nu$ be defined as in Remark \ref{remark about infty notation},
\item $\nu_N=\psi^{-1}(\nu)$,
\item $\R^{\nu_N}=\psi^{-1}(\R^\nu)$,
\item $\pi^{\nu_N}_{\hr_N}=\psi^{-1}(\pi^\nu_{\hr})$, and
\item $\hr'=(\R^{\nu_N}, \Psi_{\R^{\nu_N}})$.
\end{enumerate}
Then $\T_{\hr, \hr'}$ is $\nu$-bounded, $(\hr')^\nu=\hr^\nu$, and $\pi^{\nu}_{\hr'}=\psi\rest \hr^{\nu_N}$.
\end{lemma}

\begin{proof} 
We will use subscript $N$ to denote the $\psi$-preimages of objects in the range of $\psi$. 
It follows from hull condensation that $\Psi_N=\Psi\rest N$.
We let $p:\R^\nu\rightarrow \mH|\d^\infty$ be the iteration embedding according to $\Psi_{\hr^\nu}$ and $p_N=\psi^{-1}(p)$.\footnote{See \rlem{main lemma on small ordinals}. We use smallness of $\nu$ to conclude that $\mH|\d^\infty$ is an iterate of $\hr^\nu$.} 
We have that $p_N: \R^{\nu_N}\rightarrow \mH_N|\d^\infty_N$ is the iteration embedding according to $\Psi_{\R^{\nu_N}}$. 
Let $\hh=(\mH_N|\d^\infty_N, \Psi_{\mH_N|\d^\infty_N})$.\footnote{Notice that $\Psi_{\mH_N|\d^\infty_N}$ extends $(\Omega_N)_{\mH_N|\d^\infty_N}$.} Notice that $\psi\rest \mH_N=\pi_{\hh, \infty}$ and hence $\nu_N=\nu_\hh$.

We have that $\gen(\T_{\hr, \hr'})\subseteq \nu_N$. Therefore, to show that $\T_{\hr, \hr'}$ is $\nu$-bounded, it is enough to show that $\nu_N=\nu_{\hr'}$. Notice that $\pi_{\hr', \infty}=\pi_{\hh, \infty}\circ p_N$, and since $\nu_N=\nu_\hh$ and $\cp(p_N)>\nu_N$,\footnote{See \rlem{xi is on the main branch}. This also follows from \rlem{main lemma on small ordinals}.} we get that $\pi_{\hr', \infty}(\nu_N)=\pi_{\hh, \infty}(\nu_N)=\nu$. 
Therefore, $\nu_{\hr'}=\nu_N$.

Since $\T_{\hr, \hr'}$ is $\nu$-bounded, \rlem{bound iterations factor} implies that $(\hr')^\nu=\hr^\nu$.  

We now want to show that $\pi^{\nu}_{\hr'}=\psi\rest \hr^{\nu_N}$.
We have that $\cp(p)>\nu$ and  $\pi_{\hr, \infty}=p\circ \pi^\nu_{\hr}$. 
Suppose now $x\in \hr^{\nu_N}$. 
Fix $f\in \R$ and $a\in [\nu_N]^{<\omega}$ such that $x=\pi^{\nu_N}_{\hr_N}(f)(a)$. 
Thus, $x=\pi^{\nu_N}_{\hr_N}(f)(p_N(a))$ and the following holds:
\begin{center}
\begin{align*}
\psi(x) &= \psi(\pi^{\nu_N}_{\hr_N}(f))(\psi(p_N(a)))\\
        &= \pi^{\nu}_{\hr}(f)(\pi_{\hh, \infty}(p_N(a)))\ \ \ (\text{because}\ \  \psi\rest \mH_N=\pi_{\hh, \infty})\\
        &= \pi^{\nu}_{\hr}(f)(p\circ \pi^\nu_{\hr'}(a))\ \ \ (\text{because}\ \pi_{\hh, \infty}\circ p_N=\pi_{\hr', \infty} \text{and}\ \pi_{\hr', \infty}=p\circ \pi^\nu_{\hr'})\\
        &= \pi^{\nu}_{\hr}(f)(\pi^\nu_{\hr'}(a))\ \ \ (\text{because}\ \cp(p)>\nu)\\
        &= \pi^\nu_{\hr'}(\pi^{\nu_N}_{\hr_N}(f))(\pi^\nu_{\hr'}(a))\\
        &= \pi^\nu_{\hr'}(\pi^{\nu_N}_{\hr_N}(f)(a))\\
        &= \pi^\nu_{\hr'}(x).
        \end{align*}
        \end{center}
\end{proof}

\begin{theorem}\label{first prop}
Suppose $\nu<\xi^\infty$ is a ${<}\d^\infty$-strong cardinal in $\mH$ or a limit of ${<}\d^\infty$-strong cardinals of $\mH$.
Let $\l$ be the least ${<}\d^\infty$-strong cardinal of $\mH$ that is strictly bigger than $\nu$.
Assume $\nu$ (and hence $\l$) is small in $\P$.
Then $\powerset(\nu^\omega)\cap \c=\powerset(\nu^\omega)\cap \c_\l$.
\end{theorem}

\begin{proof} 
Fix $A\subseteq \powerset(\nu^\omega)$ with $A\in \c$ and $\gg<\xi^\infty$ such that for some $X\in \powerset_{\omega_1}(\gg)$ and $a\in \bR$, $A$ is ordinal definable from $X, a$ (in $M$)\footnote{Notice that we could take $X\in \gamma^\omega$, in which case we do not need the real parameter $a$. However, it is easier to work with $X\in \powerset_{\omega_1}(\gg)$.}.
Fix such an $X, a$ as well.
Using \rlem{hod[x] is hodx} and \rthm{cating the powerset}, we can find $s\in [\xi^\infty]^{<\omega}$ such that $A$ is ordinal definable from $X, a$ and $s$, and since $X$ is a bounded countable subset of $\xi^\infty$, we can without losing generality assume that $s\subseteq X$.
Fix a formula $\psi$  such that
\begin{center}
$Z\in A$ if and only if $\c\models \psi[Z, X, a]$.
\end{center}
It follows from \rlem{hod[x] is hodx} and \rthm{cating the powerset} that it is enough to show that $A$ is ordinal definable from some $X'\in \l^\omega$. 

Next fix a complete iterate $\hq=(\Q, \Sigma_\Q)$ of $\hp|\d$ such that $X\cup\{\gg, \nu\}\subseteq \rge(\pi_{\hq, \infty})$.
Let $X_\hq=\pi_{\hq, \infty}^{-1}[X]$, $\tau=(\nu^+)^\mH$, and $X^\tau=\pi^\tau_{\hq}[X_\hq]$.
Let $F$ be the long extender derived from $\pi^\tau_{\hq}\rest (\Q|\tau_\hq)$.
More precisely, $F$ consists of pairs $(b, B)$ such that $b\in [\nu]^{<\omega}$, $B\in \Q|\tau_\hq$ and $b\in \pi^\tau_{\hq}(B)$. Notice that\\\\
(1) $\{\Q^\tau, \Q, F\}\in \c_\l$.\\\\
Indeed, we have $\Q\in \c_\l$ because it is a countable transitive set and $\c_\l$ contains all the reals.
Next, notice that $\pi^\tau_{\hq}\rest (\Q|\tau_\hq)\in \c_\l$, which implies that $F\in \c_\l$.
Finally, notice that $\Q^\tau=Ult(\Q, F)$.
Notice also that\\\\
(2) $\pi_F^\Q=\pi^\tau_{\hq}$.\\\\
Let $\b=\pi_F(\gg_\hq)$, and let $\mu$ be the least Woodin cardinal of $\Q^\tau$ that is bigger than $\b$.
Let $\zeta=(\mu^+)^{\Q^\tau}$.
Notice next that\\\\
(3) for some $\b'<\l$ and some $X'\in (\b')^\omega$, $\Sigma_{\Q^\tau|\zeta}$ is ordinal definable in $M$ from $X'$.\\\\
Indeed, we have some $\b'<\l$\footnote{We use the fact that $\l$ is a strong cardinal to get such a $\b'$.} such that $\mH|\b'$ is a complete $\Sigma_{\Q^\tau|\zeta}$-iterate of $\Q^\tau|\zeta$.\footnote{See \rlem{main lemma on small ordinals}. This uses the smallness of $\tau$. However, one doesn't need the smallness of $\tau$ to show (3). We know that $\tau$ is on the main branch of $\T_{\hq, \infty}$ (see \rlem{xi is on the main branch}). From here we can conclude the existence of a $\b'$ and a $k$ such that $\Sigma_{\Q^\tau|\zeta}$ is the $k$-pullback of $\Omega_{\mH|\b'}$.}
Let now $k: \Q^\tau|\zeta\rightarrow \mH|\b'$ be the iteration embedding according to $\Sigma_{\Q^\tau|\zeta}$.
Because $\hq$ catches $\nu$, we have that $\cp(k)>\nu$ (see \rlem{xi is on the main branch}).
We want to argue that $k\in \c_\l$.
Notice that $\zeta\in \rge(\pi^\tau_{\hq})$. 
Let then $j=k\circ (\pi^{\tau}_{\hq} \rest (\Q|\zeta_\hq))$. 
Notice that $j\in \c_\l$, as it is (essentially) a countable subset of $\mH|\b'$.
But now we can compute $k$ from $j$.
Indeed, given $x\in \Q^\tau|\zeta$, we have some $f\in \Q|\zeta_\hq$ such that $x=\pi^\tau_{\hq}(f)(a)$ where $a\in [\nu]^{<\omega}$. Hence, $k(x)=k(\pi^\tau_{\hq}(f))(k(a))=j(f)(a)$.
Hence, $k\in \c_\l$.
Since $\Sigma_{\mH|\b}$ is ordinal definable in $M$ and $\Sigma_{\Q^\tau|\zeta}$ is the $k$-pullback of $\Sigma_{\mH|\b}$, (3) follows.
Fix now $X'$ as in (3).

We define a term relation in $\Q$.
Let $\sigma\in \Q$ be a $Coll(\omega, \mu_\hq)$-name that is a standard name for a subset of $\bR^4$ such that whenever $g\subseteq Coll(\omega, \mu_\hq)$ is generic over $\Q$, $\sigma_g$ consists of a tuple $(u_0, u_1, v_0, v_1, w)$ such that\\\\
(4.1) $u_0$ is a real that codes a function $g_{u_0}:\omega\rightarrow \Q|\tau_\hq$,\\
(4.2) $v_0$ is a real that codes a function $g_{v_0}: \omega \rightarrow \Q|\gg_\hq$,\\
(4.3) $u_1\subseteq \omega$, $v_1\subseteq \omega$ and $w\in \bR$, and\\
(4.4) letting $Z=\{\pi_{\hq, \infty}(g_{u_0}(i)): i \in u_1\}$ and $W=\{\pi_{\hq, \infty}(g_{v_1}(i)): i \in u_1\}$, $\c\models \psi[Z, W, w]$.\\\\
The following claims can be easily established using the results summarized in \rsubsec{sec: chang model}.
We leave their standard proofs to the reader.
Given a transitive model $N$ and $\iota$ a limit of Woodin cardinals of $N$, we will write $D^{N, \iota}\models \phi[...]$ to mean that whenever $h\subseteq Coll(\omega, <\iota)$ is $N$-generic and $D$ is the derived model of $N$ as computed by $h$,\footnote{So $D=(L(Hom^*))^{N(\bR^*)}$, where $\bR^*=\cup_{\iota'<\iota}\bR^{N[h\cap Coll(\omega, <\iota')]}$.} $D\models \phi[...]$.
Similarly, we will write $c=b^{D^{N, \iota}}$ just in case $b$ is definable in $D$ and $c$ is the object defined by the definition of $b$.
Clause 4 of \rcl{sigma is def in der model} is meaningful because we can define $(\pi_{\hq, \infty})^{D^{\Q^+, \eta}}$ inside $\Q$.
This is a consequence of generic interpretability (see \cite{SteelCom}).

\begin{claim}\label{sigma is def in der model}\normalfont
$\sigma\in \Q$\footnote{Because $\sigma$ is ordinal definable in $M$ from $\hq$.} and, letting $\hq^+=(\Q^+, \Sigma_{\Q^+})$ be the complete iterate of $\hp$ obtained by copying $\T_{\hp|\d, \hq}$ onto $\P$ via $id$, the following equivalence holds in $\Q^+$. 
$(p, \dot{b})\in \sigma$ if and only if $p\in Coll(\omega, \mu_\hq)$, $\dot{b}=(\dot{u_0}, \dot{u}_1, \dot{v}_0, \dot{v}_1, \dot{w})$ and $p$ forces the following statement:
\begin{enumerate}[label=(\alph*)]
\item $\dot{u}_0$ is a real that codes a function $g_{\dot{u}_0}:\omega\rightarrow \check{\Q}|\check{\tau_\hq}$,
\item $\dot{v}_0$ is a real that codes a function $g_{\dot{v}_0}: \omega \rightarrow \check{\Q}|\check{\gg_\hq}$,
\item $\dot{u}_1\subseteq \omega$, $\dot{v}_1\subseteq \omega$ and $\dot{w}\in \dot{\bR}$, and
\item it is forced by $Coll(\omega, <\iota)$ that, letting $\pi=(\pi_{\hq, \infty})^{D^{\Q^+, \eta}}$, $\epsilon=\pi(\check{\xi_\hq})$, $\dot{Z}=\{\pi(g_{\dot{u}_0}(i)): i \in \dot{u}_1\}$ and $\dot{W}=\{\pi(g_{\dot{v}_0}(i)): i \in \dot{v}_1\}$, \[\c^{D^{\Q^+, \eta}}_{\epsilon}\models \psi[\dot{Z}, \dot{W}, \dot{w}].\]
\end{enumerate}
\end{claim}

\begin{claim}\label{sigma is def in der model1}\normalfont
Let $\hq^+=(\Q^+, \Sigma_{\Q^+})$ be as in \rcl{sigma is def in der model}, and suppose $\hs=(\S, \Sigma_\S)$ is a complete iterate of $\hq$.
Suppose $g\subseteq Coll(\omega, \mu_\hs)$ is generic over $\S$ and $(u_0, u_1, v_0, v_1, w)\in (\pi_{\hq^+, \hs}(\sigma))_g$.
Then the following conditions are satisfied:
\begin{enumerate}[label=(\alph*)]
\item $u_0$ is a real that codes a function $g_{u_0}:\omega\rightarrow \S|\tau_\hs$,
\item $v_0$ is a real that codes a function $g_{v_0}: \omega \rightarrow \S|\gg_\hs$,
\item $u_1\subseteq \omega$, $v_1\subseteq \omega$, and $w\in \bR^{\S[g]}$, and
\item letting $Z=\{\pi_{\hs, \infty}(g_{u_0}(i)): i \in u_1\}$ and $W=\{\pi_{\hs, \infty}(g_{v_1}(i)): i \in u_1\}$, \[\c\models \psi[Z, W, w].\]
\end{enumerate}
\end{claim}

Our goal now is to show that, letting $\sigma^\tau=\pi^\tau_{\hq}(\sigma)$, $\sigma^\tau$ can be used to identify $A$ in $\c_\l$.
We observe that (1) and (2) imply that $\sigma^\tau\in \c_\l$.
Lemma \ref{sigma gives a} demonstrates exactly what we want. 
Let $\eta_0$ be the least Woodin cardinal of $\P$ that is bigger than $\d$, and let $\a=((\eta_0^\infty)^+)^{M}$, where $\eta_0^\infty=\pi_{\hp, \infty}(\eta_0)$.
Let $U\in M$ be the supercompactness measure on $\powerset_{\omega_1}(L_\a(\Delta))$, where $\Delta$ is the collection of all sets of reals whose Wadge rank is $<\eta_0^\infty$.
It follows from a result of Woodin \cite{Wo20} that $U$ is the unique supercompactness measure on $\powerset_{\omega_1}(L_\a(\Delta))$. Hence,\\\\
(5) $U$ is ordinal definable in $M$.\\\\
Clause 2 of \rlem{sigma gives a} defines $A$ from the tuple $(\Q^\tau, X^\tau, \sigma^\tau, \b, \Sigma_{\Q^\tau|\zeta}, U)$.
It then follows from (1) and (3) that $A\in \c_\l$.
Therefore \rlem{sigma gives a} is all that we need to prove.

\begin{lemma}\label{sigma gives a}
Suppose $Z\in \tau^\omega$.
Then the following are equivalent.
\begin{enumerate}
\item $Z\in A$.
\item For almost all $Y\in U$, 
\begin{enumerate}
\item $Y\prec L_\a(\Delta)$,
\item $\{\Q^\tau, \sigma^\tau, \b\}\in Y$, and
\item letting 
\begin{enumerate}
\item $\pi_Y: N_Y\rightarrow L_\a(\Delta)$ be the inverse of the transitive collapse of $Y$, 
\item $\b_Y=\pi^{-1}_Y(\b)$, 
\item for $i\in \omega$, $\mu_Y=\pi_Y^{-1}(\mu)$, 
\item $\Q_Y^\tau=\pi_Y^{-1}(\Q^\tau)$, $\sigma_Y=\pi_Y^{-1}(\sigma^\tau)$, $\tau_Y=\pi_Y^{-1}(\tau)$, $\nu_Y=\pi_Y^{-1}(\nu)$ and $\zeta_Y=\pi_Y^{-1}(\zeta)$,
\item $\Psi_Y$ be the $\pi_Y$-pullback of $\Sigma_{\Q^\tau|\zeta}$ and 
\item $\hr_Y=(\Q_Y^\tau|\zeta_Y,\Psi_Y)$,
\end{enumerate}
there is a complete iterate $\hs$ of $\hr_Y$ such that $\T_{\hr_Y, \hs}$ is above $\b_Y$, $\T_{\hr_Y, \hs}$ is below $\mu_Y$, and there is an $\M^\hs$-generic $g\subseteq Coll(\omega, \mu_\hs)$ such that for some $(u_0, u_1, v_0, v_1, w)\in (\pi_{\hr_Y, \hs}(\sigma_Y))_g$, 
\[Z=\{\pi_Y(g_{u_0}(m)): m\in u_1\}, w=a \text{ and }X^\tau=\{\pi_Y(g_{v_0}(i)): i\in v_1\}.\]
\end{enumerate}
\end{enumerate}
\end{lemma}
\begin{proof}
We start by showing that clause 1 implies clause 2.
Suppose that $Z\in A$. 
Fix any $Y\prec L_\a(\Delta)$ such that $\{\hq, Z, \tau, \sigma^\tau, \b\}\in Y$.
The set of such $Y$ is in $U$, and so it suffices to show that $Y$ satisfies the conditions listed in clause 2.

We have that $Z\subseteq Y$.
Let $q:\Q^\tau\rightarrow \mH$ be the iteration embedding via $\Sigma_{\Q^\tau}$, $\hq_Y=\pi_Y^{-1}(\hq)$, and  $\hr_Y^+=(\hq_Y)^{\tau_Y}_{\infty}$\footnote{$(\hq_Y)^{\tau_Y}_{\infty}$ is defined in $N_Y$. We have that $\Q^\tau_Y=\M^{\hr_Y^+}$. See Notation \ref{iteration terminology}.}.
\rlem{main lemmA} implies that\\\\
(6) $\hr_Y=\hr_Y^+|\zeta_Y$, $(\hr^+_Y)^\tau=\hq^\tau$, $\pi_{\hr^+_Y, \infty}=q\circ \pi_{\hr^+_Y}^\tau$, and $\pi_Y\rest \Q_Y^\tau=\pi^\tau_{\hr^+_Y}$.\\\\
Fix now $(u_0, u_1, v_0, v_1)$ such that\\\\
(7.1) $u_0$ is a real coding a function $g_{u_0}: \omega\rightarrow \Q^\tau_Y|\tau_Y$,\\
(7.2) $u_1\subseteq \omega$ is such that $Z=\{ \pi_Y(g_{u_0}(m)): m\in u_1\}$,\\
(7.3) $v_0$ is a real coding a function $g_{v_0}: \omega\rightarrow \Q^\tau_Y|\b_Y$, and\\
(7.4) $v_1\subseteq \omega$ is such that $X^\tau=\{\pi_Y(g_{v_0}(m)): m\in v_1\}$.\\\\
Let $\hs$ be a complete iterate of $\hr_Y$ such that $\T_{\hr_Y, \hs}$ is above $\b_Y$, $\T_{\hr_Y, \hs}$ is below $\mu_Y$, and there is an $\M^{\hs}$-generic $g\subseteq Coll(\omega, \mu_\hs)$ such that $(u_0, u_1, v_0, v_1, a)\in \M^{\hs^+}[g]$.
Let $\hs^+$ be the complete iterate of $\hr^+_Y$ obtained by copying $\T_{\hr_Y, \hs}$ onto $\hr_Y^+$ via $id$.
Notice that\footnote{We have that $\M^{\hr_Y^+}=\Q^\tau_Y$.}\\\\
(8) $\pi_{\hs^+, \infty}\rest (\M^{\hs^+}|\b_Y)=\pi_{\hr_Y^+, \infty}\rest (\M^{\hr_Y^+}|\b_Y)$,\\\\
which is a consequence of the fact that $\T_{\hr_Y, \hs}$ is above $\b_Y$. 
It follows from (6), (7.1)-(7.4), and (8) that\\\\
(9) $(u_0, u_1, v_0, v_1, a)\in (\pi_{\hr_Y, \hs}(\sigma_Y))_g$.\\\\
To see that (9) holds, notice that (6) implies that $\sigma_Y=\pi_{\hq, \hr^+_Y}(\sigma)$, $\tau_Y=\tau_{\hr^+_Y}$ and $\b_Y=\b_{\hr^+_Y}$.
Using (8) we get that\\\\
(10.1) $\cp(q)>\tau$, \\
(10.2) $\pi_{\hs^+, \infty}\rest (\Q^\tau_Y|\b_Y)=\pi_{\hr^+_Y, \infty}\rest (\Q^\tau_Y|\b_Y)=q\circ \pi^\tau_{\hr^+_Y} \rest (\Q^\tau_Y|\b_Y)$, and \\
(10.3) $\pi_{\hq, \infty}=q\circ \pi_{\hq}^\tau$ and $q[X^\tau]=X$.\\\\
Setting $\M^{\hs^+}=\S$, we now have that 
\\\\
(11.1) $u_0$ is a real coding a function $g_{u_0}: \omega\rightarrow \S|\tau_{\hs^+}$,\\
(11.2) $u_1\subseteq \omega$ is such that $Z=\{ \pi_{\hs^+, \infty}(g_{u_0}(m)): m\in u_1\}$,\\
(11.3) $v_0$ is a real coding a function $g_{v_0}: \omega\rightarrow \S|\b_{\hs^+}$,\\
(11.4) $v_1\subseteq \omega$ is such that $X=\{\pi_{\hs^+, \infty}(g_{v_0}(m)): m\in v_1\}$.\\\\
It follows from the definition of $\sigma$ that $(u_0, u_1, v_0, v_1, a)\in (\pi_{\hr^+_Y, \hs^+}(\sigma_Y))_g$.
(9) now follows from the fact that  $\hs^+|\zeta_{\hs^+}=\hs$ and $\pi_{\hr_Y, \hs}=\pi_{\hr^+_Y, \hs^+}\rest (\hq^\tau_Y|\zeta_Y)$.
We have shown that clause 1 implies clause 2.

We now show that clause 2 implies clause 1 keeping the notation introduced so far.
Assume then that the conditions listed under clause 2 hold true for some $Z\in \nu^\omega$.
There is a $U$-measure one set of $Y$ such that $\{\hq, Z, \tau, \sigma^\tau, \b\}\in Y$. 
So let $Y$ be such that the conditions listed under clause 2 are true and $\{\hq, Z, \tau, \sigma^\tau, \b\}\in Y$. 
Let $\hq_Y=\pi_Y^{-1}(\hq)$ and $\hr_Y^+$ be defined as above. 
We then get, using the argument spelled out above, that $(\hq_Y)^{\tau_Y}_{\infty}=\hq^\tau_Y$. 
Let $\hs, g, (u_0, u_1, v_0, v_2, a)$ be as in clause 2. 
Just like above, we define $\hs^+$ to be the result of copying $\T_{\hr_Y, \hs}$ onto $\hr_Y^+$ via $id$. 
We still have (6), and hence (11.1)-(11.4) are still true. 
It then follows from the definition of $\sigma$, \rsubsec{sec: chang model}, \rcl{sigma is def in der model} and \rcl{sigma is def in der model1} that $\c\models \phi[Z, X, a]$ and hence, $Z\in A$.
\end{proof}
\end{proof}

\begin{corollary}\label{zf in c}
Suppose $\xi$ is an inaccessible cardinal of $\P$ that is also a limit of $<\d$-strong cardinals of $\P$. 
Then $\c^-\models ``{\sf{ZF}}+``\omega_1$ is a supercompact cardinal."
\end{corollary}

\begin{proof}
We first prove that $\c^-\models {\sf{ZF}}$.
By \rcor{towards zf}, it suffices to verify that $\c^-$ satisfies the Powerset Axiom.
Suppose $X\in \c^-$.
We need to see that there is $Y\in \c^-$ such that $\c^-\models \powerset(X)=Y$.
First fix some $\b<\xi^\infty$ such that $X\in \c_\b$.
It then follows that there is a surjection $F:\b^\omega\rightarrow X$ with $F\in \c^-$.
\rthm{first prop} implies that $\powerset(\b^\omega)\in \c^-$, and
hence $\powerset(X)=\{ F[a]: a\in \powerset(\b^\omega)\}\in \c^-$.

To see that $\c^-\models ``\omega_1$ is a supercompact cardinal," we let $\mu_\a\in M$ be the $\omega_1$-supercompactness measure on $\powerset_{\omega_1}(\a^\omega)$, for $\a<\xi^\infty$.
It is enough to show that for each $\a<\xi^\infty$, $\mu_\a\cap \c^-\in \c^-$. 
Fix such an $\a<\xi^\infty$, and let $\l$ be the third least $<\d^\infty$-strong cardinal of $\mH$ that is greater than $\a$.
We have that $\powerset(\a^\omega)\cap \c^-\in \c^-_\l$.
It follows from a result of Woodin that $\mu_\a$ is ordinal definable (see \cite{Wo20}), and hence, fixing a surjection $G:\l^\omega\rightarrow \powerset(\a^\omega)\cap \c^-_\l$ with $G\in \c^-$, we have that the set $A=\{ z\in \l^\omega: G(z)\in \mu_\a\}$ is ordinal definable from $G$. 
Hence $A\in \c^-$ (see \rthm{cating the powerset}).
\end{proof}

We need the following version of \rthm{first prop} which appears as \cite[Theorem 6.3(a)]{MPSC}.

\begin{theorem}\label{easier first prop}
If $\nu$ is the first ${<}\d^\infty$-strong cardinal of $\mH$, then $\powerset(\omega^\omega)\cap \c_\nu=\powerset(\omega^\omega)\cap \c$.
\end{theorem}

The proof of \rthm{easier first prop} is very similar to, but much easier than, the proof of \rthm{first prop}.
We do not need to consider the Skolem hulls as we did in \rlem{sigma gives a}.
Nor do we need to push $\hq$ up to level $\tau$. 
We only need to consider $\hq$ itself and define $\gg$, $\b$, $\mu$, and $\zeta$ relative to $\hq$ via the same definition. 
The set $A$ that we need to consider is simply a set of reals, and so the term relation $\sigma$ that we consider only needs to add a code for the parameter $W$, and the pair $(u_0, u_1)$ can be substituted with just a single $u\in \bR$.
\rlem{sigma gives a} is easier now as we can simply show that $A$ can be defined using $\hq|\zeta$ and a pair of reals $(v_0, v_1)$ coding the parameter $X$ using the method illustrated in (7.3) and (7.4) of the proof of \rthm{first prop}.
We leave the details to the reader.

\subsection{Strong cardinals are regular cardinals}
We next show that all small strong cardinals $<\xi^\infty$ are regular cardinals in $\c$.

\begin{lemma}\label{doesnt matter}
Suppose $\nu$ is in the $\d$-block of $\P$ and is a $<\d$-strong cardinal of $\P$.
Suppose $\a\in ((\nu^+)^\P, \d)$ and $\hq$ is a complete iterate of $\hp$ such that $\T_{\hp, \hq}$ is based on $(\a, \d)$.
Let $E \in \vec{E}^\Q$ be such that $\lh(E)\in (\a, \d_\hq)$ and $\cp(E)=\nu$. 
Set $\P_E=Ult(\P, E)$, $\Q_E=Ult(\Q, E)$, $\hp_E=(\P_E, \Sigma_{\P_E})$, and $\hq_E=(\Q_E, \Sigma_{\Q_E})$.
Then $\pi_{\hp_E, \infty}\rest (\nu^+)^\P=\pi_{\hq_E, \infty}\rest (\nu^+)^\P$.
\end{lemma}

\begin{proof} 
Let $\U=(\T_{\hp, \hq})^{n\frown} E$\footnote{See \rnot{concatinating iterations}.} and $\X$ be the full normalization of $\T_{\hp, \hq} \oplus E$. We then have that $\U\insegeq \X$ and $\X=\U^\frown \Y$, where $\Y$ is the minimal $E$-copy of $\T_{\hp, \hq}$. 
Because $\a>\nu$, we have that $\Y$ is above $\lh(E)$. 
Thus $\hq_E$ is a complete iterate of $\hp_E$ and $\Y=\T_{\hp_E, \hq_E}$. 
Therefore, $\pi_{\hp_E, \infty}=\pi_{\hq_E, \infty}\circ \pi_{\hp_E, \hq_E}$. 
Since $\cp(\pi_{\hp_E, \hq_E})>(\nu^+)^\P$, we have that $\pi_{\hp_E, \infty}\rest (\nu^+)^\P=\pi_{\hq_E, \infty}\rest (\nu^+)^\P$.
\end{proof}

\begin{theorem}\label{first prop1}
Suppose $\nu<\xi^\infty$ is a small $<\d^\infty$-strong cardinal in $\mH$. 
Then $\c\models ``\nu$ is a regular cardinal".
\end{theorem}

\begin{proof}
Towards a contradiction, suppose $\nu$ is not a regular cardinal in $\c$. 
Let $\a<\xi^\infty$ and $\b<\nu$ be such that for some $X\in \a^\omega$, there is $f:\b\rightarrow \nu$ such that $f$ is cofinal and ordinal definable from $X$. 
Fix such $X$ and $f$. 
The hard case is when $\a>\nu$, and so we assume that this is indeed the case. 
As in the proof of \rthm{first prop}, we can assume that the definition of $f$ doesn't contain any ordinal parameters beyond those already in $X$. 
We then fix some formula $\psi$ such that $f(a)=b$ if and only if $\c\models \psi[a, b, X]$. 
If $Z\in \a^\omega$ is any set, then we let $f^Z$ be the function defined by $f^Z(a)=b$ if and only if $\c\models \psi[a, b, Z]$. 
We say $Z$ is \textit{appropriate} if $f^Z:\b\rightarrow \nu$. 

Suppose that $\hq=(\Q, \Sigma)$ is a complete iterate of $\hp$ such that $\hq$ catches every ordinal in the set $\{\a, \b\}\cup X$. 
Let $\tau$ be the least Woodin cardinal of $\mH$ that is strictly greater than $\a$ (so $\tau>\nu$). 
Suppose that $g\subseteq Coll(\omega, \tau_\hq)$ is $\Q$-generic. 
Let $\gg_{\hq, g}$ be the supremum of all ordinals $\iota$ such that for some pair $(u, v)\in \Q[g]$,\\\\
(1.1) $u$ codes a function $g_u: \omega\rightarrow \Q|\a_\hq$,\\
(1.2) $v\subseteq \omega$,\\
(1.3) letting $Z=(\pi_{\hq, \infty}(g_u(i)): i<\omega)$, $Z$ is appropriate, and\\
(1.4) $\iota=\sup\{ f^Z(a): a\in \pi_{\hq, \infty}[\b_\hq]\}$.\\\\
Notice that\\\\
(2.1) $\gg_{\hq, g}$ is definable in $\Q$ (see \rsubsec{sec: chang model} and \rcl{sigma is def in der model}), \\
(2.2) $\gg_{\hq, g}$ is independent of $g$, and\\
(2.3) $\gg_{\hq, g}<\nu$.\\\\
(2.3) follows because $\cf(\nu)\geq \omega_1^M$, and there are at most countably many such $\iota$. 
We let $\gg^\hq=\gg_{\hq, g}$ (and hope that the reader will not confuse this notation with our notation for preimages). 

Fix now a complete iterate $\hq=(\Q, \Sigma)$ of $\hp$ that catches every ordinal in the set $\{\a, \b\}\cup X$. 
Using \rthm{complexity of window strategy}\footnote{This theorem is an important and technical theorem of independent interest, and its proof is independent of the material developed here. We therefore postpone its proof.} and using \rlem{doesnt matter}, we obtain a complete iterate $\hr=(\R, \Sigma_\R)$ of $\hq$ such that $\T_{\hq, \hr}$ is above $\tau_\hq$ and for some $\zeta$ that is an inaccessible cardinal of $\R$, letting $E\in \vec{E}^\R$ be the least such that $\zeta<\lh(E)$ and $\cp(E)=\nu_\hr$, \\\\
(3) $\gg^{\hq}<\pi_{\hr_E, \infty}(\nu_\hr)$, where $\hr_E=(Ult(\R, E), \Sigma_{Ult(\R, E)})$.\\\\
Since $\T_{\hq, \hr}$ is above $\tau_\hq$ (implying that $\hq|\tau_\hq=\hr|\tau_\hr$), we have that\\\\
(4) $\gg^{\hq}=\gg^{\hr}$.\\\\
This gives the key fact (which uses the results of \rsubsec{sec: chang model} and also \rlem{doesnt matter}) that:\\\\
(5) $\R\models \gg^{\hr}<\pi_{\hr_E, \infty}(\nu_\hr)$.\\\\
We get a contradiction by establishing the following lemma.  

\begin{lemma} $\rge(f)\subseteq \pi_{\hr_E, \infty}(\nu_\hr)<\nu$.
\end{lemma}

\begin{proof}
We have that $\nu_\hr<\nu_{\hr_E}=\pi_E^\R(\nu_\hr)$. 
Hence, $\pi_{\hr_E, \infty}(\nu_\hr)<\nu$. 
Suppose $\iota<\b$. 
We want to see that $f(\iota)<\pi_{\hr_E, \infty}(\nu_\hr)$. 
Let $\b'\in [\b_\hr, \nu_\hr)$ be a properly overlapped cardinal of $\R_E$,\footnote{The existence of such a $\b'$ follows from smallness of $\nu$.} and let $\hs$ be a complete iterate of $\hr_E$ such that\\\\
(6) $\gen(\T_{\hr_E, \hs})\subseteq \pi_{\hr_E, \hs}(\b')$ and $\hs$ catches $\iota$ (see \rlem{main lemma on small ordinals}).\\\\
Let now $\X$ be the full normalization of $E\oplus \T_{\hr_E, \hs}$. (Both $\X$ and $E\oplus \T_{\hr_E, \hs}$ are on $\R$.) 
It follows from (6) and the fact that $\b'<\nu_\hr$ that\\\\
(7.1) $\X=\Y^\frown F$,\\
(7.2) if $\hw=(\W, \Sigma_\W)$ is the last model of $\Y$, then $\gen(Y)\subseteq \pi_{\hr, \hw}(\b')$, and\\
(7.3) $F=\pi_{\hr, \hw}(E)$\footnote{This follows from full normalization. Since we chose $\zeta$ to be inaccessible in $\R$, we have that full normalization coincides with embedding normalization.} and $\S=Ult(\W, F)$.\\\\
It follows from (5) that, letting $\hw_F=(Ult(\W, F), \Sigma_{Ult(\W, F)})$,\\\\
(8) $\W\models \gg^{\hw}<\pi_{\hw_F, \infty}(\nu_\hw)$.\\\\
But since $\S=Ult(\W, F)$, we have that\\\\
(9) $\hs=\hw_F$, $\pi_{\hr_E, \hs}(\nu_\hr)=\nu_\hw$,\footnote{To see that $\pi_{\hr_E, \hs}(\nu_\hr)=\nu_\hw$, we argue as follows. Let $i=\pi^{\T_{\hr_E, \hs}}\circ \pi_E$ and $j=\pi_F\circ \pi^{\T_{\hr, \hw}}$. We have that $i=j$. Hence, $\sup(i[\nu_\hr])=\sup(j[\nu_\hr])$. But (6) implies that $\sup(i[\nu_\hr])=\pi_{\hr_E, \hs}(\nu_\hr)$ and $\sup(j[\nu_\hr])=\pi_{\hr, \hw}(\nu_\hr)=\nu_\hw$. Therefore, $\pi_{\hr_E, \hs}(\nu_\hr)=\nu_\hw$.} and $\pi_{\hw_F, \infty}(\nu_\hw)=\pi_{\hs, \infty}(\nu_\hw)$.\\\\
It then follows that\\\\
(10) $\pi_{\hr_E, \infty}(\nu_\hr)=\pi_{\hs, \infty}(\nu_\hw)$.\\\\
Since $\cp(F)=\nu_\hw>\iota_{\hs}$, we have that\\\\
(11) $\hw$ catches $\iota$ and $\iota_\hw=\iota_\hs$.\\\\
The next claim finishes the proof.
\begin{claim}\normalfont 
$f(\iota)< \pi_{\hr_E, \infty}(\nu_\hr)$
\end{claim}
\begin{proof}
By (10), it is enough to show that\\\\
(12) $f(\iota)<\pi_{\hs, \infty}(\nu_\hw)$.\\\\
Let $u_0$ be a real coding a surjection $g_{u_0}:\omega \rightarrow \W|\a_\hw$ and $u_1\subseteq \omega$ be such that $X=(\pi_{\hw, \infty}(g_{u_0}(i)): i\in \omega)$. 
We can find a complete iterate $\hw'=(\W', \Sigma_{\W'})$ of $\hw$ such that $\T_{\hw, \hw'}$ is based on the interval $(\a_\hw, \tau_\hw)$ and for some generic $h\subseteq Coll(\omega, \tau_{\hw'})$, $(u_0, u_1)\in \W'[h]$.
Then it follows from the definition of $\gg^{\hw'}$ that\\\\
(13) $f(\iota)\leq \gg^{\hw'}$.\\\\
Using (8) we get that, letting $G=\pi_{\hw, \hw'}(F)$ and $\hw'_G=(Ult(\W', G), \Sigma_{Ult(\W', G)})$, \\\\
(14) $\gg^{\hw'}<\pi_{\hw'_G, \infty}(\nu_{\hw'})$.\\\\
Then it follows from (10), (13) and (14) that it is enough to establish\\\\
(15) $\pi_{\hs, \infty}(\nu_\hw)=\pi_{\hw'_G, \infty}(\nu_{\hw'})$\\\\
Let $\U$ be the result of applying $\T_{\hw, \hw'}$ to $\hs$.
Thus $\U$ has the same extenders and tree structure as $\T_{\hw, \hw'}$.
Let $\hs'$ be the last model of $\U$. 
Then the full normalization of $F\oplus \U$ has the form $(\T_{\hw, \hw'})^{n\frown} G ^\frown \U'$, where $\U'$ is on $Ult(\W, G)$, is above $\pi_G(\nu_\hw)$, and is the minimal $G$-copy of $\T_{\hw, \hw'}$. 
Similarly, the full normalization of $\T_{\hw, \hw'}\oplus G$ is $(\T_{\hw, \hw'})^{n\frown} G ^\frown \U'$.
Thus the full normalization of $F\oplus \U$ and of  $\T_{\hw, \hw'}\oplus G$ is $\T_{\hw, \hs'}$. 
It follows that $\hs'=\hw'_G$ and $\T_{\hs, \hs'}=\U$ is above $\nu_\hw$.
Thus\\\\
    (16.1) $\pi_{\hs, \infty}(\nu_\hw)=\pi_{\hs', \infty}(\nu_\hw)$,\\
    (16.2) $\nu_\hw=\nu_{\hw'}$.\\\\
Since $\hs'=\hw'_G$, (16.1) gives (15).
\end{proof}
This proves the lemma.
\end{proof}
\rthm{first prop1} is proved.
\end{proof}

\subsection{Small strongs are successor cardinals}


\begin{theorem}\label{strongs are successors}
Suppose $\nu<\xi^\infty$ is in the $\d^\infty$-block of $\mH$ and is a small strong cardinal of $\mH$. 
Let $\l$ be the supremum of all ${<}\d^\infty$-strong cardinals of $\mH$ that are strictly less than $\nu$.
If $\l$ is a ${<}\d^\infty$-strong cardinal of $\mH$, then $\c\models \nu=\l^+$.
\end{theorem}

\begin{proof}
Let $\b\in (\l, \nu)$.
It is enough to prove that for some $X\in (\xi^\infty)^\omega$ that is bounded in $\xi^\infty$, there is a surjection $f:\l \rightarrow (\l, \b)$ such that $f$ is ordinal definable from $X$.
Let $\mu=\sup\{ \gg <\l :\mH\models ``\gg$ is $<\d^\infty$-strong"$\}$. Because $\nu$ is small, we have that $\mu<\l$. 
As $\nu$ is a limit of properly overlaped ordinals of $\mH$ that are inaccessible cardinals of $\mH$, it is enough to show the above assertion for a $\b$ that is a properly overlapped inaccessible cardinal in $\mH$. 
By iterating $\hp$ if necessary, we may further assume that $\hp$ catches $\b$ (and hence, $\nu, \l$ and $\mu$). 
Given a $<\d^\infty$-strong cardinal $\zeta$ of $\mH$ and a hod pair $\hr$ such that $\M_\infty(\hr)=\mH|\d^\infty$, let $\b^\zeta_\hr=\pi^\zeta_{\hr}(\b_\hr)$. 
It follows from item (3) in the proof of \rthm{first prop} that\\\\
(1) the tuple $\hp^\l|\b^\l_\hp=(\P^\l|\b^\l_\hp, \Sigma_{\hp^\l|\b^\l_\hp})$ is ordinal definable from a member of $\nu^\omega$.\\\\
Thus it is enough to find a surjection $f:\l\rightarrow (\l, \b)$\footnote{Since $\b$ is inaccessible in $\mH$, there is a bijection $h:\b\rightarrow (\l, \b)$.} that is ordinal definable from the tuple $(\hp^\l|\b^\l_\hp, \pi_{\hp}^\l\rest (\P|\b_\hp))$.
To do so, we employ the following terminology.
\begin{itemize}
    \item[(a)] Working in $M$, we say that a hod pair $\hr=(\R, \Lambda)$ is \textit{appropriate} if 
    \begin{enumerate}
    \item $\M_\infty(\hr)=\mH|\d^\infty$, 
    \item $\hr$ catches $\b$, 
    \item $\b^\l_\hr=\b^\l_\hp$, 
    \item $\M^{\hr^\l}=\M^{\hp^\l}$, 
    \item $\hr^\l|\b^\l_\hr=\hp^\l|\b^\l_\hp$,
    \item $\pi_{\hp}^\l [\P|\b_\hp]\subseteq \rge(\pi_{\hr}^\l)$, and
    \item letting $m:\P|(\l_{\hp}^+)^\P\rightarrow \R|(\l_\hr^+)^\R$ be given by $m(x)=(\pi_{\hr}^\l)^{-1}(\pi_{\hp}^\l(x))$, if $A\in \powerset(\l_\hr)\cap \R$, then for some $f\in \P|(\l^+_\hp)^\P$ and for some $a\in [\l_\hr]^{<\omega}$, $A=m(f)(a)$.
    \end{enumerate}
    For example, if $\hr$ is a complete $\l$-bounded\footnote{I.e. $\gen(\T_{\hp, \hr})\subseteq \l_\hr$. See \rdef{mu bounded iterations}.} iterate of $\hp$, then $\hr$ is appropriate (see \rlem{bound iterations factor}).

    \item[(b)] Working in $M$, given a hod pair $\hr$ that catches $\b$, we let $\zeta^\mu_\hr<\l$ and $j^{\mu}_\hr$ be such that if $\hh=_{def}(\mH, \Omega)$, then
    \begin{enumerate}
    \item $\hh|\zeta^\mu_\hr$ is a complete iterate of $\hr^{\mu}|\b^\mu_\hr$ and 
    \item $j^{\mu}_\hr=\pi^{\T_{\hr^{\mu}|\b^\mu_\hr, \hh^\mu_\hr|\zeta^\mu_\hr}}$.
    \end{enumerate} 
    Set $\epsilon_\hr=j^\mu_\hr(\pi^\mu_{\hr}(\l_\hr))$.

    \item[(c)] Working in $M$, we let $\U_\hr$ be the results of copying $\T_{\hr^{\mu}|\b^\mu_\hr, \hh^\mu_\hr|\zeta^\mu_\hr}$ on $\hr^\mu$ via $id$.
    Let $\X_\hr$ be the longest initial segment of $\U_\hr$ such that $\gen(\X_\hr)\subseteq \epsilon_\hr$, and let $\Y_\hr=(\U_\hr)_{\geq \iota}$, where $\iota=\lh(\X_\hr)$.

    \item[(d)] Let $\hx_\hr$ be the last model of $\X_\hr$ and $\hy_\hr$ be the last model of $\Y_\hr$. We have that
    \begin{enumerate}
    \item $\Y_\hr$ is strictly above $\epsilon_\hr$, and
    \item $\hx_\hr|\epsilon_\hr=\hh|\epsilon_\hr$. 
    \end{enumerate}

    \item[(e)]  Working in $M$, we say that $\gg$ is a \textit{trace} of $\rho$ if $\rho\in [\l, \b]$ and there is an appropriate hod pair $\hr$ and  a complete iterate $\hw$ of $\hr$ such that 
    \begin{enumerate}
    \item $\hw$ catches $\rho$,
    \item $\T_{\hr, \hw}$ is $\b$-bounded (see \rdef{mu bounded iterations}), 
    \item $\T_{\hr, \hw}$ is strictly above $\l_\hr$,\footnote{If $\T_{\hr, \hw}$ is non-trivial, then $\hw$ is not appropriate as $\M^{\hw^\l}\not =\M^{\hp^\l}$. However, in this case $\hw^\l$ is a complete iterate of $\hr^\l$ such that $\T_{\hr^\l, \hw^\l}$ is strictly above $\l$ and is the $F$-minimal copy of $\T_{\hr, \hw}$, where $F$ is the long extender derived from $\pi_{\hr}^\l\rest (\R|(\l_\hr^+)^\R)$.} and 
    \item $j^{\mu}_\hw(\pi^\mu_{\hw}(\rho_\hw))=\gg$.
    \end{enumerate}
    We say $\rho$ is a \textit{pre-trace} of $\gg$ if $\gg$ is a trace of $\rho$.
\end{itemize}

It follows from \rlem{main lemma on small ordinals} that $\l$ is a limit of ordinals $\gg$ that have a pre-trace. 
Define $f:\l\rightarrow (\l, \b)$ as follows. 
Given $\gg<\l$, if $\gg$ doesn't have a pre-trace, then we let $f(\gg)=\l+1$, and otherwise we let $f(\gg)=\rho$, where $\rho$ is a pre-trace of $\gg$.
\rlem{preprojection is unique} shows that $f$ is a function.
Since every $\rho\in [\l, \b]$ has a trace, we have that $f$ is a surjection, and therefore, \rlem{preprojection is unique} finishes the proof of \rthm{strongs are successors} provided we can argue that $f\in \c$.
But $f\in \c$ because the relation ``$\rho$ is the pre-trace of $\gg$" is ordinal definable from the tuple $(\P^\l, \hp^\l|\b^\l_\hp, \pi_{\hp}^\l\rest (\P|\b_\hp))$. 

It remains to prove the following lemma.
It is important to keep in mind that because $\b$ is properly overlapped in $\mH$, if $\gg$ is a trace of $\rho$ as witnessed by $(\hr, \hw)$, then $\epsilon_\hr$ and $\zeta^\mu_\hr$ are uniquely determined by the pair $(\mu, \gg)$. Indeed, $\epsilon_\hr$ is the least cardinal $\tau$ of $\mH$ such that $\tau>\mu$ and $\tau$ is $<\gg$-strong in $\mH$. Similarly, $\zeta^\mu_\hr$ is the largest cardinal $\tau$ of $\mH$ such that $\epsilon_\hr$ is $<\tau$-strong in $\mH$.

\begin{lemma}\label{preprojection is unique}
Suppose $\gg$ is a trace of $\rho^0$ and $\rho^1$.
Then $\rho^0=\rho^1$.
\end{lemma}

\begin{proof}
Fix $(\hr_0, \hw_0)$ and $(\hr_1, \hw_1)$ witnessing that $\gg$ is a trace of $\rho^0$ and $\rho^1$, respectively.
It follows from full normalization that if $G_i$, for $i\in 2$, is the long extender derived from $\pi^\mu_{\hr_i}$ (i.e., $(a, A)\in G_i$ if and only if $a\in [\mu]^{<\omega}$, $A\in \M^{\hr_i}$ and $a\in \pi^\mu_{\hr_i}(A)$) and $\Z^i$ is the minimal $G_i$-copy of $\T_{\hr_i, \hw_i}$, then\\\\
(2) $(\T^\mu_{\hr_i})^\frown \Z^i$ is the full normalization of $\T_{\hr_i, \hw_i}\oplus \T^\mu_{\hw_i}$.\\\\
Let $\a'\in (\d^\infty, \Theta^M)$ and $\Delta=\{ A\subseteq \bR: w(A)<\a'\}$.
Let $\a''=(\Theta^+)^{L(\Delta)}$, and let $Y\prec L_{\a''}(\Delta)$ be a countable substructure of $L_{\a''}(\Delta)$ such that $\{\hr_0, \hw_0, \hr_1, \hw_1, \b, \rho^0, \rho^1, \gg\}\in Y$.
Let $\psi: N\rightarrow L_{\a''}(\Delta)$ be the transitive collapse of $Y$ with $\psi$ the inverse of the transitive collapse.
We will use subscript $Y$ to denote $\psi$-preimages of objects in $Y\cap L_{\a''}(\Delta)$, but in doing so we will not decorate everything with $Y$.
For example, we let $j^\mu_{\hr_0, Y}=\psi^{-1}(j^\mu_{\hr_0})$.

Now we have that\\\\ 
(3.1) $j^\mu_{\hw_0, Y}(\pi^\mu_{\hw_0, Y}(\rho^0_{\hw_0}))=j^\mu_{\hw_1, Y}(\pi^\mu_{\hw_1, Y}(\rho^1_{\hw_1}))=_{def}\gg_Y$,\\
(3.2) $j^\mu_{\hw_0, Y}(\pi^\mu_{\hw_0, Y}(\b_{\hw_0}))=j^\mu_{\hw_1, Y}(\pi^\mu_{\hw_1, Y}(\b_{\hw_1}))=_{def}\b'_Y(=\zeta^\mu_{\hw_0, Y}=\zeta^\mu_{\hw_1, Y}$), and\\
(3.3) $j^\mu_{\hw_0, Y}(\pi^\mu_{\hw_0, Y}(\l_{\hw_0}))=j^\mu_{\hw_1, Y}(\pi^\mu_{\hw_1, Y}(\l_{\hw_1}))=_{def}\l'_Y(=\epsilon_{\hw_0, Y}=\epsilon_{\hw_1, Y})$.\\\\
(3.2) and (3.3) follow from (3.1) and the fact that $\gg_Y$ and $\mu_Y$ determine both the image of $\b_{\hw_0}$ and the image of $\l_{\hw_0}$. For example, $\l'_Y$ is the least $<\gg_Y$-strong cardinal of $\mH_Y$ which is above $\mu_Y$, and $\b'_Y$ is the largest ordinal $\iota$ such that $\mH_Y\models ``\l'_Y$ is $<\iota$-strong". 

We now use the notation introduced in (c) and (d) above. 
For $i\in 2$, we have the following objects: $\U_{\hw_i, Y}$, $\X_{\hw_i, Y}$, $\Y_{\hw_i, Y}$, $\epsilon_{\hw_i, Y}$, $\hx_{\hw_i, Y}$ and $\hy_{\hw_i, Y}$. For $i\in 2$, we let\\\\
(f.1) $\U_{\hw_i, Y}=\U_{i, Y}$, $\X_{\hw_i, Y}=\X^{', i}_Y$, $\Y_{\hw_i, Y}=\Y^{', i}_Y$, $\epsilon_{\hw_i, Y}=\epsilon_{i, Y}$, $(\hr_{i, Y})^{\mu_Y}=\hr_{i, Y}^\mu$, $(\hw_{i, Y})^{\mu_Y}=\hw_{i, Y}^\mu$, $\hx_{\hw_i, Y}=\hx'_{i, Y}$ and $\hy_{\hw_i, Y}=\hy_{i, Y}$.\\
(f.2) $\S^i_Y$ be the longest initial segment of $\T_{\hr_{i, Y}^\mu, \hx'_{i, Y}}$ such that $\gen(\S^i_{Y})\subseteq \l_{\hx'_{i, Y}}=\l'_Y$, and let $\hx_{i, Y}$ be the last model of $\S^i_Y$.\\\\
Then for $i\in 2$,\\\\
(4.1) $\hx_{i, Y}'$ is a complete iterate of $\hx_{i, Y}$, and $\T_{\hx_{i, Y}, \hx_{i, Y}'}$ is strictly above $\l'_Y$,\\
(4.2) $\epsilon_{0, Y}=\epsilon_{1, Y}=\l'_Y$, $\M^{\hx_{i, Y}}|\l'_Y=\mH_Y|\l'_Y$,\\ 
(4.3) letting $\phi_Y$ be the successor of $\l'_Y$ in $\mH_Y$, $\M^{\hx_{i, Y}}|\phi_Y=\mH_Y|\phi_Y$,\\
(4.4) $\hy_{i, Y}|\b'_Y=\hh_Y|\b'_Y$, and\\
(4.5) $\hy_{0, Y}|\b'_Y=\hy_{1, Y}|\b'_Y$.\\\\
(4.2) is just a reformulation of (3.3), and (4.1) is a consequence of full normalization. In fact, $\T_{\hx_{i, Y}, \hx'_{i, Y}}$ is the $\pi_{\hr^\mu_{i, Y}, \hx_{i, Y}}$-minimal copy of $\T_{\hr^\mu_{i, Y}, \hw^\mu_{i, Y}}$.
Let\\\\
(f.3) $\X^i_Y=\T_{\hx_{i, Y}, \hy_{i, Y}}$.\\\\
We think of the hod pairs $\hr_{i, Y}^\mu$ and the others displayed in (f.1)-(f.3) as hod pairs in $M$, not just in $N$\footnote{This means that their strategies act on all iterations not just those in $N$.}.
Let $F$ be the long extender derived from $\psi\rest (\mH_Y|(\mu_Y^+)^{\mH_Y})$.
It follows from \rlem{main lemmA} that\\\\
(5) for $i\in 2$, $\psi\rest \M^{\hr_{i, Y}^\mu}=\pi^\mu_{\hr_{i, Y}^\mu}$.\\\\
For $i\in 2$, let\\\\
(g) $\S^i$ be the minimal $F$-copy of $\S^i_Y$ and $\X^i$ be the minimal $F$-copy of $\X^i_Y$.\\
(h) $\hx_i$ be the last model of $\S^i$ and $\hy_i$ be the last model of $\X^i$.\\\\
Then for $i\in 2$,\\\\
(6.1) $(\T_{\hr_{i, Y}^\mu, \hr_i^\mu})^\frown \S^i$ is the full normalization of $(\S^i_Y) \oplus \T_{\hx_{i, Y}, \hx_i}$,\\
(6.2) $(\T_{\hx_{i, Y}, \hx_i})^\frown \X^i$ is the full normalization of $(\X^i_Y) \oplus \T_{\hy_{i, Y}, \hy_i}$, \\
(6.3) $\hx_i=(\hx_{i, Y})^\mu$ and $\hy_i=(\hy_{i, Y})^\mu$, and\\
(6.4) $\pi^\mu_{\hx_{i, Y}}=\pi_{\hx_{i, Y}, \hx_i}=\pi_F^{\M^{\hx_{i, Y}}}$ and $\pi^\mu_{\hy_{i, Y}}=\pi_{\hy_{i, Y}, \hy_i}=\pi_F^{\M^{\hy_{i, Y}}}$.\\\\
(6.4) holds by the same proof that established (5) (see \rlem{main lemmA}).
Putting (1)-(6) together, we have for $i\in 2$,\\\\
(7.1) $\M^{\hr_i^\mu}=Ult(\M^{\hr_{i, Y}^\mu}, F)$, $\M^{\hx_i}=Ult(\M^{\hx_{i, Y}}, F)$, and $\M^{\hy_i}=Ult(\M^{\hy_{i, Y}}, F)$,\\
(7.2) $\gen(\S^i)\subseteq \pi_{\hx_{i, Y}}^\mu(\l'_Y)$, and\\
(7.3) $\X^i$ is strictly above $\pi_{\hx_{i, Y}, \hx_i}(\l'_Y)$.\\

Since $\gen(\S^i_Y)\subseteq \l_{\hx_{i, Y}}=\l'_Y$ for $i\in 2$, we have that\\\\
(8) for $i\in 2$, $\hr^\l_i=(\hx_{i, Y})^\l$.\\\\
(8) then implies\\\\
(9) for $i\in 2$, $\hr^\l_i$ is an iterate of $\hx_i$, $\T_{\hx_i, \hr^\l_i}$ is strictly above $\mu$, and $(\T_{\hx_{i, Y}, \hx_i})^\frown \T_{\hx_i, \hr^\l_i} = \T_{\hx_{i, Y}, \hr^\l_i}$.\\\\
We now prove the following claim, where for $i\in 2$, $\l_{\hx_i}=\pi_{\hr_i, \hx_i}(\l_{\hr_i})$.
(It follows from (6.4) that $\l_{\hx_i}=\pi_F^{\M^{\hx_{i, Y}}}(\l'_Y)$, and hence $\l_{\hx_0}=\l_{\hx_1}$.)

\begin{claim}\label{claim embs are the same for lambda01}\normalfont $\pi_{\hx_0, \hr^\l_0}\rest \M^{\hx_0}|\l_{\hx_0}=\pi_{\hx_1, \hr^\l_1}\rest \M^{\hx_1}|\l_{\hx_1}$.
\end{claim}

\begin{proof}
The fact that $\M^{\hx_0}|\l_{\hx_0}=\M^{\hx_1}|\l_{\hx_1}$ follows from (4.2) and (7.1).
Suppose now $u\in \M^{\hx_0}|\l_{\hx_0}$.
We have that $u=\pi_F(h)(a)$, where $h\in \M^{\hx_{0, Y}}|\l'_Y$ and $a\in [\mu]^{<\omega}$. 
Here and in other places in the proof of this claim, $\pi_F$ stands for $\pi_F^{\M^{\hx_{0, Y}|\l'_Y}}=\pi_F^{\M^{\hx_{1, Y}|\l'_Y}}$.
We then have that 
\begin{center} $\pi_{\hx_0, \hr^\l_0}(u)=\pi_{\hx_0, \hr^\l_0}(\pi_F(h))(a)$ and $\pi_{\hx_1, \hr^\l_1}(u)=\pi_{\hx_1, \hr^\l_1}(\pi_F(h))(a)$.
\end{center} 
We need to see that\\\\
(*) $\pi_{\hx_0, \hr^\l_0}(\pi_F(h))=\pi_{\hx_1, \hr^\l_1}(\pi_F(h)).$\\\\
(6.4) implies that for $i\in 2$, $\pi_{\hx_i, \hr^\l_i}(\pi_F(h))=\pi_{\hx_{i, Y}, \hr^\l_i}(h)$.
We claim that\\\\
(**) $\pi_{\hx_{0, Y}, \hr^\l_0}(h)=\psi(h)=\pi_{\hx_{1, Y}, \hr^\l_1}(h)$.\\\\
(**) clearly implies (*).

We establish (**) as follows.
Recall that $\psi\rest \mH_Y=\pi_{\hh_Y, \hh}$. 
Set $\epsilon=\sup(\psi[\l'_Y])$ and $\l'=\sup(\rge(\pi_F))$. We have that $\l'=\l_{\hx_0}=\l_{\hx_1}$.
Then for $i\in 2$,\footnote{$\hx^\epsilon_i$ and $\hh^\mu_Y$ are defined as in \rnot{nu infty}.}\\\\
(10.1) $\T_{\hx_{i, Y}, \hr^\l_i}=(\T_{\hx_{i, Y}, \hx_i})^\frown (\T_{\hx_i, \hx^\epsilon_i})^\frown (\T_{\hx^\epsilon_i, \hr^\l_i})$,\\
(10.2) $\T_{\hx_{i, Y}, \hh}=(\T_{\hx_{i, Y}, \hr^\l_i})^\frown \T_{\hr^\l_i, \hh}$,\\
(10.3) $\T_{\hx^\epsilon_i, \hr^\l_i}$ is above $\epsilon$,\\
(10.4) $\hh_Y^\mu|\l'=\hx_{i}|\l'$\footnote{This equation is a consequence of full normalization and the fact that $\T_{\hx_{i, Y}, \hh_Y}$ is above $\l'_Y$.},\\
(10.5) $\hh|\epsilon$ is an iterate of $\hx_{i}|\l'$ via an iteration that is a segment of both $\T_{\hx_{i, Y}, \hh}$ and $\T_{\hh_Y^\mu, \hh}$ and is also strictly above $\mu$.\\\\
We then have that for $i\in 2$,\\\\
(11) $\pi_{\hx_{i, Y}, \hr^\l_i}(h)=\pi^{(\T_{\hx_{i, Y}, \hx_i})^\frown (\T_{\hx_i, \hx^\epsilon_i})}(h)$.\\\\
(**) follows from (10.2), (11), (6.4), (10.4) and (10.5).
\end{proof}
Now\\\\
(12) for $i\in 2$, $\pi_{\hw_i, \hy_i}(\rho^i_{\hw_i})=\pi^\mu_{\hy_{i, Y}}(\gg_Y)$.\\\\ 
It then follows from (3.1), (4.4) and (6.4) that\\\\
(13.1) $\pi_{\hw_0, \hy_0}(\rho^0_{\hw_0})=\pi_{\hw_1, \hy_1}(\rho^1_{\hw_1})$, and\\
(13.2) if for $i\in 2$, $\b'_i=\pi_{\hw_i, \hy_i}(\b_{\hw_i})$, then $\b_0'=\b_1'=_{def}\b'$ and $\M^{\hy_0}|\b'=\M^{\hy_1}|\b'$.\\\\
For $i\in 2$, let $H_i$ be the long extender derived from $\pi_{\hx_i, \hr_i^\l}$.
(More precisely, we let $H_i$ be the set of pairs $(a, A)$ such that $a\in [\l]^{<\omega}$, $A\in \powerset(\l_{\hx_i})\cap \M^{\hx_i}$ and $a\in \pi_{\hx_i, \hr_i^\l}(A)$.)
It follows from (a.6), (a.7), (4.3), (7.1)\footnote{We need (4.3) and (7.1) to conclude that $\powerset(\l_{\hx_0})\cap \M^{\hx_0}=\powerset(\l_{\hx_1})\cap \M^{\hx_1}$. Notice that for $i\in 2$, $\pi_{\hr_i, \hx_i}(\l_{\hr_i})=\hx_i$.}, \rcl{claim embs are the same for lambda01}, and the fact that $\gen(\T_{\hr_i, \hx_i})\subseteq \l_{\hx_i}$\footnote{These facts together imply that $\pi_{\hx_0, \hr^\l_0}\rest \powerset(\l_{\hx_0})\cap \M^{\hx_0} =\pi_{\hx_1, \hr^\l_1}\rest  \powerset(\l_{\hx_1})\cap \M^{\hx_1}$.} that\\\\
(14) $H_0=H_1=_{def} H$.\\\\
For $i\in 2$, let $\K^i$ be the $H$-minimal copy of $\X^i$, and let $\hk_i$ be the last model of $\K^i$.
It follows from full normalization, (4.5), (14), (a.5), (a.6), and (a.7) that, letting for $i\in 2$, $\b''_i=\pi_{\hw_i, \hy_i}(\b_{\hw_i})$ and $\b'''_i=\pi_{\hw_i, \hk_i}(\b_{\hw_i})$,\\\\
(15.1) $\b''_0=\b''_1=_{def}\b''$,\\ 
(15.2) $\b'''_0=\b'''_1=_{def}\b'''$,\\
(15.3) for $i\in 2$, $\M^{\hk_i}|\b'''=Ult(\M^{\hy_i}|\b'', H)$,\\
(15.4) $\hk_0|\b'''=\hk_1|\b'''$\footnote{This is because $\hr_0$ and $\hr_1$ are appropriate, and so this follows from clause (a.5) above.},\\
(15.5) for $i\in 2$, $\pi_{\hw_i, \hk_i}(\rho^i_{\hw_i})=\pi_H^{\M^{\hy_i}}(\pi_{\hw_i, \hy_i}(\rho^i_{\hw_i}))$,\\
(15.6) $\pi_{\hw_0, \hk_0}(\rho^0_{\hw_0})=\pi_{\hw_1, \hk_1}(\rho^1_{\hw_1})=_{def}\rho'$,\\
(15.7) $\pi_{\hk_0, \hh}\rest \b'''=\pi_{\hk_1, \hh}\rest \b'''$, and\\
(15.8) for $i\in 2$, $\rho^i=\pi_{\hk_i, \hh}(\rho')$.\\\\
Finally, (15.6), (15.7) and (15.8) imply that $\rho^0=\rho^1$. 
This finishes the proof of \rlem{preprojection is unique}.
\end{proof}
\rthm{strongs are successors} is proved.
 \end{proof}

\subsection{Small strongs have big cofinality}
The main result of this section is that the cofinality of strong cardinals $<\xi^\infty$ is $\Theta^\c$. 

\begin{theorem}\label{cof is kappa}
Suppose $\nu<\d^\infty$ is a small ${<}\d^\infty$-strong cardinal of $\mH$, and let $\k$ be the least ${<}\d^\infty$-strong cardinal of $\mH$.
Then $M\models ``\cf(\nu)=\k."$
\end{theorem}
\begin{proof} 
Without loss of generality, we assume that $\hp$ catches $\nu$. 
Let $\nu'=(\nu^+)^\mH$ and $w=(\nu', \d^\infty)$. We say that $\a$ is a \textit{$\k$-trace} of $\nu$  if there is a complete iterate $\hq=(\Q, \Sigma_\Q)$ of $\hp$ such that\\\\
(1.1) $\T_{\hp, \hq}$ is based on $w_\hp$, and\\
(1.2) for some $\Q$-inaccessible cardinal $\zeta\in w_\hq$, $\pi_{\hq|\zeta, \infty}(\k_\hp)=\a$.\\\\
Define $f:\k\rightarrow \nu$ by setting $f(\a)=\k$ if $\a$ is not a $\k$-trace of $\nu$, and if $\a$ is a $\k$-trace of $\nu$, then, letting $\hq=(\Q, \Sigma_\Q)$ and $\zeta$ witness that $\a$ is a $\k$-trace, and letting $E\in \vec{E}$ be the least such that $\zeta<\lh(E)$ and $\cp(E)=\nu_\hq$, set $f(\a)=\pi_{\hq_E, \infty}(\nu_\hq)$\footnote{Notice that $\nu_{\hq}=\nu_\hp$. It follows from \rlem{doesnt matter} that $\pi_{\hq_E, \infty}(\nu_\hq)=\pi_{\hp_E, \infty}(\nu_\hp)$.}, where $\Q_E=Ult(\Q, E)$ and $\hq_E=(\Q_E, \Sigma_{\Q_E})$.

\rlem{doesnt matter}, \rthm{complexity of window strategy} and \rcor{proof of c1 implies c2} imply that $\rge(f)$ is cofinal in $\nu$, so it is enough to show that
\begin{itemize}
    \item[(a)] $f$ is a function, and

    \item[(b)] $f$ is increasing.
\end{itemize}

We first show (a). 
To show that $f$ is a function, it is enough to show that if $\a$ is a $\k$-trace of $\nu$, then $f(\a)$ is independent of the witness $(\hq, \zeta)$. 
Let $(\hq=(\Q, \Sigma_\Q), \zeta)$ and $(\hr=(\R, \Sigma_\R), \tau)$ be two pairs witnessing that $\a$ is a $\k$-trace of $\nu$. 
Let $E\in \vec{E}^\Q$ be the least such that $\cp(E)=\nu_\hp$ and $\zeta<\lh(E)$, and let $F\in \vec{E}^\R$ be the least such that $\cp(F)=\nu_\hp$ and $\tau<\lh(F)$. 
Let $\Q_E=Ult(\Q, E)$, $\hq_E=(\Q_E, \Sigma_{\Q_E})$, $\R_F=Ult(\R, F)$ and $\hr_F=(\R_F, \Sigma_{\R_F})$. 
We want to see that $\pi_{\hq_E, \infty}(\nu_\hp)=\pi_{\hr_F, \infty}(\nu_\hp)$.

We can assume that $\gen(\T_{\hp, \hq})\subseteq \zeta$ and $\gen(\T_{\hp, \hr})\subseteq \tau$. 
[Indeed, notice that it follows from \rlem{doesnt matter} that if $\hq'=(\Q', \Sigma_{\Q'})$ is such that for some $\iota<\lh(\T_{\hp, \hq})$, $\hq'=\hm_\iota^{\T_{\hp, \hq}}$ and $E\in \vec{E}^{\Q'}$, then defining $\hp_E$ and $\hq'_E$ as we usually do, $\pi_{\hq'_E, \infty}(\nu_\hp)=\pi_{\hp_E, \infty}(\nu_{\hp})$. ]

Our strategy is to compare $\hq$ and $\hr$ via the least-extender-disagreement-coiteration.
Since $\pi_{\hq|\zeta, \infty}(\k_\hq)=\pi_{\hr|\tau, \infty}(\k_\hr)$, $\hq|\zeta$ and $\hr|\tau$ iterate to the same pair.\footnote{Notice that $\k_\hq=\k_\hr=\k_\hp$.} 
But because $\gen(\T_{\hp, \hq})\subseteq \zeta$ and $\gen(\T_{\hp, \hr})\subseteq \tau$, full normalization implies that if $\hs$ is the common iterate of $\hq$ and $\hr$ obtained via the least-extender-disagreement-coiteration, then $\T_{\hq, \hs}$ is based on $\hq|\zeta$ and $\T_{\hr, \hs}$ 
is based on $\hr|\tau$. 
Then $\pi_{\hq, \hs}(\zeta)=\pi_{\hr, \hs}(\tau)=_{def}\iota$ and $\pi_{\hq, \hs}(E)=\pi_{\hr, \hs}(F)=_{def}G$. 
It also follows from \rlem{no overlapping} that both $\T_{\hq, \hs}$ and $\T_{\hr, \hs}$ are above $\nu_{\hp}$.

If $\hw$ is the common iterate of $\hq_E$ and $\hr_F$ obtained via the least-extender-disagreement-comparison, then the full normalization of $E\oplus \T_{\hq_E, \hw}$ is $(\T_{\hq, \hs})^{n\frown} G ^\frown \U_0$ and the full normalization of $F\oplus \T_{\hr_F, \hw}$ is $(\T_{\hr, \hs})^{n\frown} G ^\frown \U_1$, where $\U_0$ and $U_1$ are above $\pi_G(\nu_\hp)$. 
It follows from \rlem{doesnt matter} that, letting $\hs_G=(\S_G, \Sigma_{\S_G})$ where $\S_G=Ult(\S, G)$, we have $\pi_{\hq_E, \infty}(\nu_\hp)=\pi_{\hs_G, \infty}(\nu_\hp)=\pi_{\hr_F, \infty}(\nu_\hp)$. This finishes the proof of (a). 

The proof of (b) is similar. 
Indeed, if $\a<\b$ are two $\k$-traces as witnessed by $(\hq, \zeta, E)$ and $(\hr, \tau, F)$, respectively, $\hs$ is the common iterate of $\hq$ and $\hr$ obtained via the least-extender-disagreement-comparison, and $\hw$ is the common iterate of $\hq_E$ and $\hr_F$ obtained via the least-extender-disagreement-comparison, then\\\\
(2.1) $\T_{\hq, \hs}$ is based on $\hq|\zeta$ and is above $\nu_{\hp}$, \\
(2.2) $\T_{\hr, \hs}$ is based on $\hr|\tau$ and is above $\nu_\hp$,\\
(2.3) $\pi_{\hq, \hs}(\zeta)<\pi_{\hr, \hs}(\tau)$,\\
(2.4) if $G=\pi_{\hq, \hs}(E)$, then the full normalization of $E\oplus \T_{\hq_E, \hw}$ is $(\T_{\hq, \hs})^{n\frown} G ^\frown \U_0$, where $\U_0$ is above $\pi_G(\nu_\hp)$,\\
(2.5) if $H=\pi_{\hr, \hs}(F)$, then the full normalization of $F\oplus \T_{\hr_F, \hw}$ is $(\T_{\hr, \hs})^{n\frown} H ^\frown \U_1$, where $\U_1$ is above $\pi_H(\nu_\hp)$.\\\\
Again, we have that $\pi_{\hq_E, \infty}(\nu_\hp)=\pi_{\hs_G, \infty}(\nu_\hp)$ and $\pi_{\hr_F, \infty}(\nu_\hp)=\pi_{\hs_H, \infty}(\nu_\hp)$.
Because $\lh(E)<\lh(H)$, the argument used in \rlem{doesnt matter} implies that $\pi_{\hs_G, \infty}(\nu_\hp)<\pi_{\hs_H, \infty}(\nu_\hp)$.\footnote{Indeed, we have that the full normalization of $H\oplus G$ is $G^\frown \pi_G(H)$, which implies that $\pi_{(\hs_H)_G, \infty}(\nu_\hp)=\pi_{\hs_G, \infty}(\nu_\hp)$.} Hence, (b) follows.
\end{proof}

\section{The proof of the key technical theorem}

In this section, we adhere to \rnot{notation for kom model} with the exception of $\xi$.
\emph{In this section, $\xi$ does not stand for the object introduced in \rnot{notation for kom model}.}

\begin{theorem}\label{complexity of window strategy}
Let $\a<\d$ be any ordinal in the $\d$ block of $\P$, and suppose $\nu<\a$ is a small ${<}\d$-strong cardinal of $\P$. 
Set $w=(\a, \d)$. 
Then the following hold. 
\begin{enumerate}
    \item Suppose $\nu$ is the least ${<}\d$-strong cardinal of $\P$. Then for any ordinal $\b<\nu^\infty=\pi_{\hp, \infty}(\nu)$, there is a complete iterate $\hq$ of $\hp$ and $\zeta<\d_\hq$ such that $\T_{\hp, \hq}$ is based on $w$, \[\M^\hq\models ``\zeta \text{ is inaccessible,"}\]  and $\b\leq \pi_{\hq|\zeta, \infty}(\nu_\hq)$. 
        
    \item For any ordinal $\b<\nu^\infty$, there is a complete iterate $\hq$ of $\hp$ and $\zeta<\d_\hq$ such that $\T_{\hp, \hq}$ is based on $w$, \[\M^{\hq}\models ``\zeta \text{ is inaccessible,"}\] and, letting
    \begin{enumerate}
    \item $E\in \vec{E}^{\M^{\hq}}$ be the least with the property that $\cp(E)=\nu_\hq$ and $\zeta<\lh(E)$,
    \item $\hq=(\Q, \Sigma_\Q)$ and $\hq_E=(Ult(\Q, E), \Sigma_{Ult(\Q, E)})$, and
    \item $\hp_E=(Ult(\P, E), \Sigma_{Ult(\P, E)})$,
    \end{enumerate}
    $\b\leq \pi_{\hp_E, \infty}(\nu)=\pi_{\hq_E, \infty}(\nu)$\footnote{The equality follows from \rlem{doesnt matter}.}.
\end{enumerate}
\end{theorem}

We start by proving that clause 1 implies clause 2. 
Set $w_\infty=\pi_{\P, \infty}(w)$. 
For $\hq\in \mathcal{F}$, we let $w_\hq=\pi_{\hp, \hq}(w)$. 

\subsection{Clause 1 implies clause 2} 
We assume clause 1. Let $\tau=(\nu^+)^\P$, and set $X=\pi_{\hp, \infty}[\P|\tau]$. 
Since $\nu^\infty$ is a limit of $X$-closed points, it is enough to establish clause 2 for $\b$ that are $X$-closed (see \rlem{closed points equiv}).
Suppose now that $\b$ of clause 2 above is an $X$-closed point. 
In what follows we show that, assuming clause 1 of \rthm{complexity of window strategy} holds, clause 2 of \rthm{complexity of window strategy} holds for $\b$. 
Fix $(\hr, \hs, E)$ as in \rlem{closed points equiv} witnessing that $\b$ is $X$-closed (we apply the lemma to $(\hp, \nu)$, so $\hr$ is an iterate of $\hp$). We will not use $\hs$ in the argument below, and so we will use the letter $\hs$ to denote other objects.
Set $\R_E=Ult(\R, E)$ and $\hr_E=(\R_E, \Sigma_{\R_E})$. 
Then\\\\
(1) $\b=\pi_{\hr_E, \infty}(\nu_\hr)$.\\\\
Using clause 1 of \rthm{complexity of window strategy}, let $\hq=(\Q, \Sigma_\Q)$ be a complete iterate of $\hp$ such that if $\k$ is the least $<\d$-strong cardinal of $\P$,\\\\
(2.1) $\T_{\hp, \hq}$ is based on $w_\hp$, and\\
(2.2) for some $\Q$-inaccessible $\zeta$, $\pi_{\hr_E|\tau_{\hr_E}, \infty}(\nu_\hr)<\pi_{\hq|\zeta, \infty}(\k)$.\\\\
Fix $\zeta$ as in (2.2) and notice that (2.2) implies the following key fact:\\\\
(3) $\hr_E|\tau_{\hr_E}$ loses the least-extender-disagreement-coiteration with $\hp_H|\zeta$.\\

\begin{lemma}\label{qzeta works}
$(\hq, \zeta)$ witnesses clause 2 of  \rthm{complexity of window strategy} for $\b$.
\end{lemma}
\begin{proof} Let $H\in \vec{E}^\Q$ be the least such that $\cp(H)=\nu_\hq=\nu$ and $\zeta<\lh(H)$. Let $\P_H=Ult(\P, H)$ and $\hp_H=(\P_H, \Sigma_{\P_H})$. We want to see that\\\\ 
(*) $\pi_{\hp_H, \infty}(\nu_\hp)>\b=\pi_{\hr_E, \infty}(\nu_\hr)$.\\\\ 
To see that (*) holds, we analyze the least-extender-disagreement-coiteration of $(\hp_H, \hr_E)$. Let $\hk=(\K, \Sigma_\K)$ be the common iterate of $\hp_H$ and $\hr_E$ obtained via the least-extender-disagreement-coiteration of $(\hp_H, \hr_E)$, and set $\T=\T_{\hp_H, \hk}$ and $\U=\T_{\hr_E, \hk}$. As 
\begin{center}
$\pi_{\hp_H, \infty}(\nu_\hp)=\pi_{\hk, \infty}(\pi^\T(\nu_\hp))$ and $\pi_{\hr_E, \infty}(\nu_\hr)=\pi_{\hk, \infty}(\pi^\U(\nu_\hr))$,
\end{center}
it is enough to show that $\pi^\T(\nu_\hp)>\pi^\U(\nu_\hr)$.
Let $\nu'$ be the supremum of the ${<}\d_\hp$-strong cardinals of $\P$ that are strictly less than $\nu$, and let $\tau'$ be the successor of $\nu'$ in $\P$. 
Since $\nu$ is small in $\P$, we have that $\nu'<\nu$, and moreover\\\\
(4.1) $\nu'_{\hp_H}=\nu'$ and $\tau'_{\hp_H}=\tau'$,\\
(4.2) $\pi_{\hp, \hk}\rest (\P|\tau')=\pi_{\hp_H, \hk}\rest (\P_H|\tau'_{\hp_H})$,\footnote{$\hp_H$ is the last model of $(\T_{\hp, \hq})^{n\frown} H$.}\\
(4.3) $\pi_{\hp_H, \hk}(\tau')=\pi_{\hr_E, \hk}(\tau'_{\hr_E})=\tau'_\hk$,\\
(4.4) $\hp_H|\zeta=\hq|\zeta$ and $\pi_{\hp_H|\zeta, \infty}=\pi_{\hq|\zeta, \infty}$.\\\\
 Let $\T'$ be the longest initial segment of $\T$ such that $\gen(\T')\subseteq \nu'_{\hk}$ and $\U'$ be the longest initial segment of $\U$ such that $\gen(\U')\subseteq \nu'_{\hk}$. 

Let $(\X, \Y)$ be the iteration trees produced via the least-extender-disagreement-coiteration of $(\hp_H|\zeta, \hr_E|\tau_{\hr_E})$. Because $\zeta$ is inaccessible in $\P_H$ and $\tau_{\hr_E}$ is a regular cardinal in $\R_E$, we have that  $\X^{\P_H}\insegeq \T$ and $\Y^{\R_E}\insegeq \U$, where $\X^{\P_H}$ and $\Y^{\R_E}$ are the $id$-copies of $\X$ and $\Y$ onto $\P_H$ and $\R_E$. 
In what follows we will abuse our notation and let $\X^{\P_H}=\X$ and $\Y^{\R_E}=\Y$.

\begin{sublemma}\label{tprime in seg} 
$\T'\insegeq \X$ and $\U'\insegeq \Y$.
\end{sublemma}

\begin{proof} Let $S$ be the set of all $\l\leq \nu'_{\hk}$ such that $\l$ is a ${<}\d_\hk$-strong cardinal of $\K$ or a limit of such cardinals. 
We prove by induction on $\l\in S$ that if $\T'_\l$ and $\U'_\l$ are the longest initial segments of $\T'$ and $\U'$ with the property that $\gen(\T'_\l)\subseteq \l$ and $\gen(\U')\subseteq \l$, then $\T'_\l\insegeq \X$ and $\U'_\l\insegeq \Y$.
This follows from \rcor{no extender disagreement} and \rlem{no overlapping}. 
We then define, for $\l\in S$, the statements\\\\
${\sf{IH}}_\l$: for all $\l'\in S\cap \l$, $\T'_{\l'}\insegeq \X$ and $\U'_{\l'}\insegeq \Y$,\\\\
and\\\\
${\sf{IH}}(\l)$: $\T'_\l\insegeq \X$ and $\U'_\l\insegeq \Y$.\\\\
Since we are allowing padding, for any $\l\in S$ and for any $\epsilon$, \begin{center}$\epsilon\leq \lh(\T'_\l)$ and $\T'_\l\rest \epsilon \inseg \X$ if and only if $\epsilon \leq \lh(\U'_\l)$ and $\U'_\l\rest \epsilon \inseg \Y$. \end{center}
It follows from clause (3) of \rlem{main lemma on small ordinals} that if $\l$ is a limit point of $S$, then ${\sf{IH}}_\l$ implies ${\sf{IH}}(\l)$.

\subsubsection{\textbf{Verifying ${\sf{IH}}(\l)$ for $\l=\min(S)$.}} 
Assume first that $\l$ is the least ${<}\d_\hk$-strong cardinal. 
In what follows, we will use $\l_\hm$ to denote the image or the pre-image of $\l$ in various iterates $\hm$ of $\hp$.  
Letting $\X_0$ be the longest initial segment of $\T_{\hp, \hr_E}$ such that $\gen(\X_0)\subseteq \sup(\pi_{\hp, \hr_E}[\l_\hp])$, $\X_0\insegeq \X$,\footnote{Here we abuse the notation and assume that $\X_0$ is an iteration tree on $\P_H$. The abuse is warranted because $\P|\l_\hp=\P_H|\l_{\hp_H}$. We will do this sort of abuse of notation whenever it helps to keep the notation simple.} and if $\iota_0=\lh(\X_0)$, then $\Y_{<\iota_0}$ has no extenders and is obtained by padding. 
There are two cases.

\textbf{Case 1.} $\sup(\pi_{\hp, \hr_E}[\l_\hp])=\l_{\hr_E}$. 
In this case, it follows from \rcor{no extender disagreement} that $\l_{\hr_E}=\l$ and both $\X_{\geq \iota_0}$ and $\Y_{\geq \iota_0}$ are strictly above $\l$.
It follows that $\T'_\l=\X_0$ and $\U'_\l=\Y_{<\iota_0}$, and the claim is proved.

\textbf{Case 2.} $\sup(\pi_{\hp, \hr_E}[\l_\hp])<\l_{\hr_E}$. 
In this case, since $\l_{\hp_H}=\l_{\hp}$ and $\hp_H|\l_{\hp_H}=\hp|\l_\hp$, it follows that $\sup(\pi_{\hp_H, \hk}[\l_{\hp_H}])<\l$.
\rlem{no overlapping} then implies that \[\sup(\pi_{\hr_E, \hk}[\l_{\hr_E}])=\l.\]
Therefore, $\Y$ never uses an extender whose critical point is a pre-image of $\l$.
This means that if $\Y_0$ is the longest initial segment of $\U'$ such that $\gen(\Y_0)\subseteq \l$, $\Y_0$ is based on $\R_E|\l_{\hr_E}$. 
Since $\hp_H|\zeta$ must win the coiteration with $\hr_E|\tau_{\hr_E}$ (see (3)), we have that $\Y_0\insegeq \Y$, and thus if $\iota_1=\lh(\Y_0)$, then $\T_{\leq \iota_1}$ is based on $\hp_H|\zeta$ and $\T'_\l=\X_{\leq \iota_1}$. 
Once again, the claim follows. 

\subsubsection{\textbf{Deriving ${\sf{IH}}(\l)$ from ${\sf{IH}}_{\l}$.}} 
As was mentioned above, if $\l$ is a limit point of $S$, then there is nothing to prove. 
Assume that $\l$ is a successor member of $S$, and let $\l'=\sup(S\cap \l)$. 
Since we allow padding,  $\lh(\T_{\l'})=\lh(\U_{\l'})=_{def}\iota'$. Let $\hs=(\S, \Sigma_\S)$ be the last model of $\T_{\l'}$ and $\hw=(\W, \Sigma_\W)$ be the last model of $\U_{\l'}$. 
Since $\nu$ is small in $\P$, we have that $(\T_{\l'})_{\geq \iota'}$ and $(\U_{\l'})_{\geq \iota'}$ are strictly above $\l'$. 
Let $\zeta_\hs=\pi_{\hp_H, \hs}(\zeta)$.
Then $(\X_{\geq \iota'}, \Y_{\geq \iota'})$ are the trees built using the least-extender-disagreement-coiteration of $(\hs|\zeta_{\hs}, \hw|\tau_{\hw})$, and moreover, as $\l\leq \nu'_{\hk}$, we have that $\l_\hs$ and $\l_\hw$ are defined.  

Let $\Z'$ be the full normalization of $(\T_{\hp, \hr_E})\oplus \U_{\l'}$, and let $\Z$ be the longest initial segment of $\Z'$ such that $\gen(\Z')\subseteq \l'$. 
Let $\hm=(\M, \Sigma_\M)$ be the last model of $\Z'$. 
Since $\P|\nu=\P_H|\nu$, we have that $\l_\hm=\l_\hs$ and $\M|\l_\hs=\S|\l_\hs$.\footnote{Notice that $\cp(\pi_{\hm, \hk})>\l'$ and $\cp(\pi_{\hs, \hk})>\l'$. This implies that if $G_0$ is the set of pairs $(a, A)$ such that $a\in [\l']^{<\omega}$, $A\in \P$, and $a\in \pi_{\hp, \hm}(A)$, and $G_1$ is the set of pairs $(a, A)$ such that $a\in [\l']^{<\omega}$, $A\in \P_H$, and $a\in \pi_{\hp_H, \hs}(A)$, then  $G_0=G_1$, $\M|\tau'_{\hm}=Ult(\P|\tau_\hp', G_0)$, and $\S|\tau'_{\hs}=Ult(\P_H|\tau_{\hp_H}', G_1)$.} 
Let $\X_0$ be the longest initial segment of $\Z'_{\geq \hm}$ such that $\gen(\X_0)\subseteq \l_{\hw}$.\footnote{Because $\nu$ is small we have that $\Z'_{\geq \hm}$ is strictly above $\l'$, and therefore it is a normal iteration tree on $\M$ (after the necessary re-ordering).} 
It follows that $(\T_{\l'})^\frown \X_0\insegeq \X$, and if $\iota_0=\lh((\T_{\l'})^\frown \X_0)$, then $\Y_{(\iota', \iota_0)}$ is obtained by simply padding. 
We again consider the two cases above, namely whether $\sup(\pi^{\X_0}[\l_\hs])=\l_\hw$ or $\sup(\pi^{\X_0}[\l_\hs])<\l_\hw$, and conclude that our claim holds for $\l$. 
\end{proof}

\emph{We now free all the letters introduced in the proof of Sublemma \ref{tprime in seg} of their duty and start re-using them.} 
Set $\rho'=\sup(\pi_{\hr_E, \hk}[\nu_{\hr_E}])$, and let $\Y'$ be the longest initial segment of $\U$ such that $\gen(\Y')\subseteq \rho'$. Let $\iota'=\lh(\T')=\lh(\U')$, and let $\iota''=\lh(\Y')$. 

\begin{claim}\label{y = y'} 
$\Y'=\Y$.
\end{claim} 

\begin{proof}
Let $\hs=(\S, \Sigma_\S)$ be the last model of $\T'$ and $\hw=(\W, \Sigma_\W)$ be the last model of $\U'$. 
In light of Sublemma \ref{tprime in seg}, $\T'$ can be considered to be both an iteration tree on $\hp_H$ and on $\hp_H|\zeta$.
To distinguish them, we let $\T'_0$ be the version of $\T'$ on $\hp_H$ and $\T'_1$ be the version of $\T'$ on $\hp_H|\zeta$. 
$\T'$, $\T'_0$ and $\T'_1$ use the same extenders and have the same tree structure. 
We treat $\hs$ as the last model of $\T'_0$. 
Because $\zeta$ is chosen to be inaccessible in $\Q$, it is inaccessible in $\P_H$, and so, letting $\mu=\pi_{\hp_H, \hs}(\zeta)$, $\hs|\mu$ is the last model of $\T'_1$. 
Similarly, if we let $\U'_0$ be the version of $\U'$ on $\hr_E$, $\U'_1$ be the version of $\U'$ on $\hr_E|\tau_{\hr_E}$, and treat $\hw$ as the last model of $\U_0'$, then the last model of $\U_1'$ is simply $\hw|\epsilon$, where $\epsilon=\pi_{\hr_E, \hw}(\tau_{\hr_E})$. 
Notice that (see Terminology \ref{small cardinal})\\\\
(5.1) $\S$ does not have a ${<}\mu$-strong cardinal in the interval $(\nu'_\hs, \mu)$, and so $\mu$ is properly overlapped in $\S$,\\
(5.2) $\W$ does not have a ${<}\epsilon$-strong cardinal in the interval $(\nu'_\hw, \epsilon)$, and so $\epsilon$ is properly overlapped in $\W$, and\\
(5.3) $\T_{\geq \iota'}$ and $\U_{\geq \iota'}$ are strictly above $\nu'_{\hs}=\nu'_{\hw}=\nu'_{\hk}$.\\\\
Then $\X_{\geq \iota'}$ and $\Y_{\leq \iota'}$ are the trees produced via the least-extender-disagreement-coiteration of $(\hs|\mu, \hw|\epsilon)$. 
Let $\hs'=(\S', \Sigma_{\S'})$ and $\hw'=(\W', \Sigma_{\W'})$ be the last models of $\X$ and $\Y$. 
It follows from (3) that $\W'\inseg \S'$ and from (5.2) that $\W'|\rho'=\K|\rho'$. 
Hence $\Y'=\Y$. 
\end{proof}

Let now $\zeta_\hk=\pi_{\hp_H, \hk}(\zeta)$. 
It follows from Sublemma \ref{tprime in seg} that $\hk|\zeta_\hk$ is a complete iterate of $\hp_H|\zeta$, and it follows from \rcl{y = y'} that if $\rho''$ is the successor of $\rho'$ in $\K$, then $\hk|\rho'$ is a complete iterate of $\hr_E|\tau_{\hr_E}$. 
It follows from (3) that $\rho'< \zeta_{\hk}$. 
Notice that \[\K\models ``\text{$\pi^\T(\nu_\hp)$ is ${<}\zeta_{\hk}$-strong and $\nu'_{\hk}<\pi^\T(\nu_\hp)$,"}\] while we also have that (see (5.2)) \[\K\models ``\text{ there is no $<\pi^{\U}(\nu_{\hr})$-strong cardinal that is strictly greater than $\nu'_{\hk}$".}\] 
It then follows that $\pi^\U(\nu_{\hr})<\pi^\T(\nu_\hp)$ (because $\pi^\U(\nu_\hr)<\rho'<\zeta_\hk$). This finishes the proof of \rlem{qzeta works}.
\end{proof}

We isolate exactly what the proof that Clause 1 implies Clause 2 established (the proof of \rlem{qzeta works} is essentially the proof of \rcor{proof of c1 implies c2}).

\begin{corollary}\label{proof of c1 implies c2}
Let $\k$ be the least ${<}\d$-strong cardinal of $\P$, $\nu>\k$ be a small ${<}\d$-strong cardinal of $\P$, and $\a\in (\nu, \d)$. 
Set $w=(\a, \d)$, $\tau=(\nu^+)^\P$, and $X=\pi_{\hp, \infty}[\P|\tau]$. 
Suppose $\b<\nu^\infty$ is an $X$-closed point as witnessed by $(\hr, \hs, E)$ (see \rlem{closed points equiv}). 
Set $\R_E=Ult(\R, E)$ and $\hr_E=(\R_E, \Sigma_{\R_E})$. 
Suppose $\hq=(\Q, \Sigma_\Q)$ is a complete iterate of $\hp$ such that $\T_{\hp, \hq}$ is based on $w_\hp$ and for some $\Q$-inaccessible $\zeta\in w_\Q=_{def}(\a, \d_\Q)$, $\pi_{\hr_E|\tau_{\hr_E}, \infty}(\nu_\hr)<\pi_{\hq|\zeta, \infty}(\k)$. 
Let $H\in \vec{E}^\Q$ be the least such that $\cp(H)=\nu_\hq(=\nu=\nu_\hp)$ and $\zeta<\lh(H)$, and set $\P_H=Ult(\P, H)$ and $\hp_H=(\P_H, \Sigma_{\P_H})$. 
Then $\b<\pi_{\hp_H, \infty}(\nu_\hp)$.
\end{corollary}

\subsection{The proof of Clause 1}
We now prove clause 1 of \rthm{complexity of window strategy}. 
Fix a hod pair $\hk=(\K, \Psi)$ such that $\Psi$ has Wadge rank $<\nu^\infty=\pi_{\hp, \infty}(\nu)$ in $M$. 
Let $\tau$ be the $\P$-successor of $\nu$. 
It is enough to establish the following lemma. 
Indeed, varying $\hk$, we can find $\hq$ and $\zeta$ as in the lemma such that the Wadge rank of $\hk|\zeta$ is as high below $\nu^\infty$ as we wish. 

\begin{lemma}\label{constructing arbitrary high iterates}
There is a complete iterate $\hq=(\Q, \Sigma_\Q)$ of $\hp$ such that $\T_{\hp, \hr}$ is based on $w_\hp$ and, for some $\Q$-inaccessible cardinal $\zeta\in w_\hr$, the hod pair construction of $\Q|\zeta$ in which extenders used have critical points $>\a$ reaches a complete iterate $\hs=(\S, \Psi_\S)$ of $\hk$ such that, if $\Phi$ is the strategy of $\S$ induced by $\Sigma_\Q$,\footnote{Via the resurrection procedure of \cite[Chapter 3.4, 3.5]{SteelCom}.} then $\Phi=\Psi_\S$. 
\end{lemma}

\begin{proof}
We can assume without loss of generality, by comparing $\hp|\d$ with $\hr$ and by using the results of \rsec{sec: chang model}, that for some complete iterate $\hk_0=(\K_0, \Sigma_{\K_0})$ of $\hp$, $\K=\K_0|\tau_{\hk_0}$ and $\Psi=\Sigma_{\K}$. 
Let $\xi\geq \d_\P$ be a cutpoint cardinal of $\P$ such that $\T_{\hp, \hk_0}\in \P[g\cap Coll(\omega, \xi)]$. 
Let $\xi'$ be the least Woodin cardinal of $\P$ above $\xi$. 

We now iterate $\hp$ to obtain a complete iterate $\hq=(\Q, \Sigma_\Q)$ of $\hp$ such that $\T_{\hp, \hq}$ is based on $\P|w_\P$ and the following maximality condition holds:\footnote{Our goal is to obtain $\Q$ as the output of a certain backgrounded construction. The reader is advised to keep this in mind while reading the definition of ${\sf{Maximality}}$.}\\\\
${\sf{Maximality}}:$ Suppose $\k\in (\xi, \xi')$ coheres $\Q$ in the sense that for each $\mu\in ((\k^+)^\P, \d_\hq)$, there is an extender $F\in \vec{E}^{\P|\d_\hq}$ such that $\cp(F)=\kappa$, $\gen(F)\geq \mu$ and $\Q|\mu\insegeq \pi_F(\Q|(\kappa^+)^\Q)$.
Then for a cofinal-in-$\d_\Q$ set of $\mu$, there is an $F\in \vec{E}^{\P|\d_\hq}$ such that
\begin{enumerate}
\item $\cp(F)=\kappa$, $\mu$ is an inaccessible cardinal of $\P$, and $\gen(F)=\mu$,
\item $\Q|\mu\insegeq \pi_F(\Q|(\kappa^+)^\Q)$, and
\item the Jensen completion\footnote{See \cite[Chapter 2]{SteelCom}.}  of $G=_{def}F\rest \mu\cap \Q$ is on the sequence of $\Q$.
    More precisely, if $F'$ is the $(\k, \pi_{G}^\Q(\k))$-extender derived from $\pi^\Q_{G}$, then $F'\in \vec{E}^\Q$ and is indexed at $\pi_{F'}^\Q((\k^+)^\Q)$.\\
\end{enumerate} 
We obtain $\hq$ as the last model of the construction that produces a sequence of models $(\M_\epsilon, \N_\epsilon : \epsilon <\iota)$\footnote{Unlike \cite{SteelCom}, here we will not be concerned with showing that the construction we describe converges. Such arguments were carefully done in \cite[Chapter 8.4]{SteelCom}. Because of this, we ignore the fine structural problems that arise while doing such constructions. Also, by ``fine structural core" we mean the coring procedure developed in \cite{SteelCom}.} and a sequence of iteration trees $(\T_\epsilon: \epsilon<\iota)$ meeting the following conditions.
The construction takes place in $\P|\xi'$.\\\\
\textbf{(Condition 1)} For all $\epsilon<\iota$, $\T_\epsilon$ is an iteration tree on $\P|w_\P$ of successor length and according to $\Sigma_{\P|\d_\P}$. Let $\hw_\epsilon=(\W_\epsilon, \Sigma_{\W_\epsilon})$ be the last model of $\T_\epsilon$.\\
\textbf{(Condition 2)} $\iota$ is a successor ordinal and 
\begin{enumerate}
\item for all $\epsilon<\iota$, if $\epsilon+1<\iota$, then $\M_\epsilon\inseg \W_\epsilon$, and 
\item for all $\epsilon<\iota$, $\T_\epsilon$ is produced by comparing $\P|\d^\P$ with 
$\M_\epsilon$,\footnote{In particular, this means that $\gen(\T_\epsilon)\subseteq {\sf{Ord}}\cap \M_\epsilon$.} and
\item if $\epsilon+1=\iota$, then $\W_\epsilon|\d_{\hw_\epsilon}=\M_\epsilon$.
\end{enumerate}
\textbf{(Condition 3)} For all $\epsilon<\iota$, if $\M_\epsilon=\W_\epsilon||\gg_\epsilon$, where $\gg_\epsilon={\sf{Ord}}\cap \M_\epsilon$, then 
\begin{enumerate}
\item $\N_\epsilon$ is the rudimentary closure of $\M_\epsilon$, and
\item $\M_{\epsilon+1}$ is the fine structural core of $\N_{\epsilon}$.
\end{enumerate}
\textbf{(Condition 4)} For all $\epsilon<\iota$, letting $\gg_\epsilon={\sf{Ord}}\cap \M_\epsilon$, if
\begin{itemize}
\item $\M_\epsilon=\W_\epsilon|\gg_\epsilon$, 
\item $\W_\epsilon||\gamma_\epsilon\not =\W_\epsilon|\gamma_\epsilon$, and 
\item the last predicate of $\W_\epsilon||\gamma_\epsilon$ is either a branch or an extender with critical point $\leq \a$,
\end{itemize}
then set $\N_\epsilon=\W_\epsilon||\gamma_\epsilon$, and let $\M_{\epsilon+1}$ be the fine structural core of $\N_\epsilon$\footnote{This means that we are just copying the predicates from $W$s to $\N$s.}. \\
\textbf{(Condition 5)} For all $\epsilon<\iota$, letting $\gg_\epsilon={\sf{Ord}}\cap \M_\epsilon$, if \begin{itemize}
\item $\M_\epsilon=\W_\epsilon|\gg_\epsilon$,
\item $\W_\epsilon||\gamma_\epsilon\not =\W_\epsilon|\gamma_\epsilon$, 
\item the last predicate of $\W_\epsilon||\gamma_\epsilon$ is an extender with critical point $>\a$, and 
\item letting $E$ be the last extender of $\W_\epsilon||\gamma_\epsilon$, there is no $F\in \vec{E}^\P$ such that $E$ is the Jensen completion of $\M_\epsilon|\gen(E)\cap F$,
\end{itemize}
then let
\begin{enumerate}
\item $\N_\epsilon$ be the rudimentary closure of $\M_\epsilon|\gg_\epsilon$, and
\item $\M_{\epsilon+1}$ be the fine structural core of $\N_\epsilon$.
\end{enumerate}
\textbf{(Condition 6)} For all $\epsilon<\iota$, letting $\gg_\epsilon={\sf{Ord}}\cap \M_\epsilon$, if
\begin{itemize}
\item $\M_\epsilon=\W_\epsilon|\gg_\epsilon$, 
\item $\W_\epsilon||\gamma_\epsilon\not =\W_\epsilon|\gamma_\epsilon$, 
\item the last predicate of $\W_\epsilon||\gamma_\epsilon$ is an extender with critical point $>\a$, and 
\item letting $E$ be the last extender of $\W_\epsilon||\gamma_\epsilon$, there is an $F\in \vec{E}^\P$ such that $E$ is the Jensen completion of $\M_\epsilon|\gen(E)\cap F$,
\end{itemize}
then let 
\begin{enumerate}
\item $\N_\epsilon=\W_\epsilon||\gg_\epsilon$, and  
\item $\M_{\epsilon+1}$ be the fine structural core of $\N_\epsilon$.
\end{enumerate}
\textbf{(Condition 7)} Assume $\epsilon<\iota$ is a limit ordinal. Then $\M_\epsilon$ is $limsup_{\rho\rightarrow \epsilon}\M_\rho$. More precisely, for each $\mu$, $\M_\epsilon|(\mu^+)^{\M_\epsilon}$ is the eventual value of the sequence $(\M_\rho|(\mu^+)^{\M_\rho}: \rho<\epsilon)$.\\\\
It then follows from \cite[Chapter 8.4]{SteelCom} that for some $\epsilon$, $\M_\epsilon=\W_\epsilon|\d^{\W_\epsilon}$. 
Fix such an $\epsilon$, and let $\hq=\hw_\epsilon=_{def}(\Q, \Sigma_\Q)$. 
It also follows from \cite[Chapter 8.4]{SteelCom} that, letting $\Phi$ be the strategy $\Q$ inherits from the background universe, i.e. the strategy induced by $\Sigma$ via the resurrection procedure,\footnote{$\Phi$ acts only on iteration trees that are based on $w_\hq$.} $\Phi$ is the same as the fragment of $\Sigma_\Q$ that acts on iterations that are based on $w_\hq$.
We observe\\\\
(1) $\P|\d_\hq$ is generic for the $\d_\hq$-generator version of the extender algebra of $\Q$ at $\d_\hq$ that uses extenders whose natural lengths are inaccessible cardinals in $\Q$.\footnote{This algebra is just like the extender algebra but has $\d_\hq$ many propositional symbols. See \cite{NegRes}.}\\\\
The reason behind (1) is that if $E\in \vec{E}^\Q$ is such that $\cp(E)>\a$ and $\gen(E)$ is an inaccessible cardinal in $\Q$, then if $F$ is the resurrection of $E$, we have an elementary embedding $k: Ult(\Q, E)\rightarrow \pi_F(\Q)$ such that $k\rest \gen(E)=id$ and $\pi_F^\P\rest \Q=k\circ \pi_E^\Q$.
As $\pi^\P_F(\P|\d_\hq)|\gen(E)=\P|\gen(E)$, we have that $\P|\d_\hq$ satisfies any axiom of the extender algebra that is generated by $E$.

Let $\hq^+=(\Q^+, \Sigma_{\Q^+})$ be the last model of the $id$-copy of $\T_{\hp|\d, \hq}$ onto $\P$, $i+1=\lh(\T_{\hp|\d, \hq})$ and $\T=\T_{\hp|\d, \hq}\rest i$. 
It follows from (1) that $\Q^+[\P|\d_\hq]$ makes sense and moreover $\d_\hq$ is a regular cardinal of $\Q^+[\P|\d_\hq]$. 
It then follows that\\\\
(2) $\T\in \Q^+[\P|\d_\hq]$.\\\\
Indeed, $\T$ is the output of the least-extender-disagreement-coiteration of $(\P|\d, \Q|\d_\hq)$.
The branches of $\T$ are chosen according to $S^{\P|\d_\hq}$, which is the strategy that is indexed on the sequence of $\P|\d_\hq$. Clearly we have that $S^{\P|\d_\hq}\in \Q^+[\P|\d_\hq]$. 
The next lemma, which establishes a stronger form of maximality, is standard, and its proof can be found in \cite{SteelCom} and \cite[Lemma 2.23]{schlutzenberg2016scales}.

\begin{lemma}\label{maximality}
Set $V=\P$, and suppose $A\subseteq \d_\hq$ is such that $\d_\hq$ is Woodin relative to $(A, \Q)$, i.e., for some $\k\in w_\hq$, for any $\P$-cardinal $\mu\in (\k, \d_\hq)$, there is an $F\in \vec{E}^\P$ such that $\cp(F)=\k$, $\im(F)>\mu$, $\pi^\P_F(A)\cap \mu=A\cap \mu$, and $\pi_F^\P(\Q)|\mu=\Q|\mu$. 
Fix $\k<\d_\hq$ witnessing that $\d_\hq$ is Woodin relative to $(A, \Q)$. 
Then for an unbounded-in-$\d_\hq$ set of $\Q$-inaccessible cardinals $\mu$, there is an $F\in \vec{E}^\P$ such that $\cp(F)=\k$, $\im(F)>\mu$, $\pi_F^\P(A)\cap \mu=A\cap \mu$, $\pi_F^\P(\Q)|\mu=\Q|\mu$, and the Jensen completion of $F\cap \Q|\mu$ is on the sequence of $\Q$. 
\end{lemma}

\begin{proof} 
Let $\mu\in w_\hq$ be a measurable cardinal of $\Q$\footnote{Hence at least inaccessible in $V$. See \cite{NegRes}.} such that $\mu>\k$.
As the set of such $\mu$ is cofinal in $\d_\hq$, it is enough to prove that $\mu$ is as desired. Let $\mu'$ be the least measurable cardinal of $\Q$ above $\mu$, $\mu''$ be the least measurable cardinal of $\Q$ above $\mu'$, and $F\in \vec{E}^{\P|\d_\hq}$ be such that $\cp(F)=\k$, $\im(F)>\mu'$, $\pi_F^\P(A)\cap \mu'=A\cap \mu'$, and $\pi_F^\P(\Q)|\mu''=\Q|\mu''$. 
Setting $\W=\Q|(\k^+)^\Q$, we have that\\\\
(3) $\Q|\mu''\inseg \pi_F^\P(\W)$.\\\\
We now want to see that the Jensen completion of $F\cap \Q|\mu$ is on the extender sequence of $\Q$. 
In what follows we will abuse our notation and consider the trees $\T_\epsilon$ as trees on $\P$. 
Notice that we have some $\epsilon<\iota$ such that $\W=\M_\epsilon$. Let $\epsilon$ be such that $\W=\M_\epsilon$. 
Because $\k$ is an inaccessible cardinal of $\P$ and $\d<\k$, we have that $\lh(\T_\epsilon)=\k+1$. 
Then $\k$ is on the main branch of $\pi_F^\P(\T_\epsilon)$. 
Let $G$ be the first extender used on the main branch of $\pi_F^\P(\T_\epsilon)$ at $\k$.
We observe that $G$ and $F$ are compatible. 
Let $\S=\M_\k^{\pi_F^\P(\T_\epsilon)}$ and $\l=\pi_G^\S(\k)$. 

We claim that $\l>\mu$.
Suppose otherwise. 
It follows from (3) that $\pi_G^{\S}(\W)\insegeq \Q|\mu$, and so it follows from Condition 6 that $G$ cannot be used in $\pi_F(\T_\epsilon)$ (in $Ult(\P, F)$, we can use $F\rest \mu'$ to certify $G$).
Thus $\l>\mu$, and so $G$ must actually be on the sequence of $\Q$ (see Condition 6 above). 
\end{proof}

Now let $\X_0=\T^\nu_{\hp}(=\T_{\hp, \hp^\nu})$ and $\X_1=\T^\nu_{\hq^+}$. 
Full normalization implies that $\X_0$ is an initial segment of the full normalization of  $\T\oplus \X_1$.\footnote{The full normalization of $\T^\frown \X_1$ starts out by comparing $\P|\d_\P$ with $\mH|\nu^\infty$, and the resulting normal tree is $\X_0$.}
It follows that\\\\
(4) $\X_0\in \Q^+[\P|\d_\hq]$.\footnote{Notice that $\T$ is of limit length, but to compute $\X_0$ from $\T^\frown \X_1$, we do not need the last branch of $\T$; the full normalization process can be performed without having the last branch of $\T$.}\\

Let $\U=\T_{\hp, \hk_0}$. 
Abusing notation, we regard $\U$ as an iteration tree on $\P|\d$ as it is based on it anyway.
In $M$,\\\\
(5) $\Psi$ is ordinal definable from $\X_0$, $\U$ and $\K$.\\\\
This follows from \cite[Theorem 1.9 and 1.10]{MPSC}. 
(Indeed, recall that $\hk=\hk_0|\tau_{\hk_0}$ and $\Psi=\Sigma_\K$.) 
We can furthermore assume that $\gen(\U)\leq \nu_{\hk_0}$. 
We now have that an iteration tree $\Y$ on $\hk$ is according to $\Psi$ if and only if the embedding normalization $W(\U, \Y)$ is a weak hull\footnote{See \cite[Definition 1.7]{SteelCom}.} of $\X_0$. Recall that by the results of \cite[Chapter 6.6]{SteelCom}, $W(\U, \Y)$ uniquely determines the branches of $\Y$. 

Now $\hk$ is added by $g_\xi=g\cap Coll(\omega, <\xi)$. 
The following is our key claim. 
Its proof is based on ideas of Woodin. 

\begin{claim}\label{uni claim}\normalfont
$\M_1^{\#, \Psi}(\P|\d_\hq[g_\xi])\models ``\d_\hq$ is a Woodin cardinal". 
\end{claim}
\begin{proof}
Let $f:\d_\hq\rightarrow \d_\hq$ be a function in $\M_1^{\#, \Psi}(\P|\d_\hq[g_\xi])$. 
Since we can iterate $\Q^+$ above $\d_\hq$ to realize $M$ as the derived model at the image of $\eta_\hq$, and since $\{\X_0, \U, \K\}\subseteq \Q^+[\P|\d_\hq[g_\xi]]$ and (5) holds, we have that \[\M_1^{\#, \Psi}(\P|\d_\hq[g_\xi])\in \Q^+[(\P|\d_\hq)[g_\xi]].\]

Since the extender algebra has $\d_\hq$-cc, there is $h:\d_\hq\rightarrow \d_\hq$ in $\Q$ which dominates $f$ in the sense that for every $\rho<\d_\hq$, $f(\rho)<h(\rho)$. 
Let $E\in \vec{E}^\Q$ be an extender on the sequence of $\Q$ such that $\gen(E)$ is a measurable cardinal of $\Q$ and if $\k=\cp(E)$, then $\pi^\Q_E(h)(\k)<\gen(E)$. 
Then if $F$ is the background extender of $E$, then $E$ is the Jensen completion of $F\cap \Q|\gen(E)$. 
But now we have that $\gen(E)>\pi_F^\P(h)(\kappa)>\pi_F^\P(f)(\kappa)$. 
Thus, $F$ witnesses the Woodiness of $\d_\hq$ for $f$. 
\end{proof}

Now let $\N$ be the output of the hod pair construction of $\Q^+|\d_{\Q}$ that is built by using extenders with critical points $>\xi$. 
${\sf{Maximality}}$, which we verified in \rlem{maximality}, implies that in the comparison of this construction with $\hk$, only the later side moves (e.g. see \cite[Lemma 2.11]{HMMSC}). 
If the construction side wins the coiteration, then \rlem{constructing arbitrary high iterates} follows as witnessed by $\hq$. 
On the other hand, $\hk$ cannot win the coiteration, as this contradicts the universality of the construction. 

In order to use universality, however, we need to be more careful. Suppose $\hk$ iterates past $\N$, and let $\Y$ be the iteration tree according to $\Psi$ such that the last model of $\Y$ extends $\N$. 
We have that $\Y\in \M_1^{\#, \Psi}(\P|\d_\hq[g_\xi])$, which means that $\Y\in \Q^+[\P|\d_\hq[g_\xi]]$. 
As in all universality proofs, we can now find a continuous chain $\vec{X}=(X_\epsilon: \epsilon<\d_\hq)$ of submodels of $\M_1^{\#, \Psi}(\P|\d_\hq[g_\xi])$ such that for a club $C$ of $\epsilon<\d_\hq$, if $N_\epsilon$ is the transitive collapse of $X_\epsilon$ and $\sigma_\epsilon: N_\epsilon\rightarrow X_\epsilon$ is the uncollapse map, then $\epsilon=\cp(\sigma_\epsilon)$ and\\\\
(6) $\powerset(\epsilon)\cap N_\epsilon=\powerset(\epsilon)\cap \N$. \\\\
\rlem{maximality} implies that we can find an $F\in \vec{E}^{\P|\d_\Q}$ such that, letting $\k=\cp(F)$ and $\mu=\gen(F)$, the following holds:\\\\
(7.1) $\k\in C$ and $\mu$ is an inaccessible cardinal of both $\P$ and $\Q$,\\
(7.2) the Jensen completion of $F\cap \Q|\mu$, call it $E$, is on the extender sequence of $\Q$, and\\
(7.3) for all $\epsilon<\mu$, $N_\epsilon$ is the transitive collapse of $\pi_F^\P(\vec{X})_\epsilon$.\\\\
It then follows that if $\l=\min(C-(\k+1))$, then $E\cap \N|(\l^+)^\N$ is on the sequence of $\N$, witnessing that $\k$ is a superstrong cardinal.\footnote{The details of this argument appear in \cite[Claim 4.1]{ESC}, \cite[Chapters 8.4 and 9]{SteelCom}, \cite{HMMSC} and \cite{LSA}. The issue in our case is that we are applying universality to a model built inside a backgrounded universe, and the extra complication arises from the use of Jensen  indexing.} 
\end{proof}

\section{The proof of \rthm{main theorem}}

We apply the work of the previous sections to prove \rthm{main theorem}.
Notice that the entirety of \textsection\ref{sec: kom models} is developed under the hypothesis of \rthm{main theorem}, as it used the terminology and notation introduced in \rnot{notation for kom model}.
Let $M$ be as in \rthm{main theorem}.
By \rcor{zf in c}, \[M\models {\sf{ZF}}+\text{ ``$\omega_1$ is supercompact."}\]

We now verify that for every $n<\omega$, \[M\models \Join^n_{\l}.\] 
It suffices to verify that the following hold in $M$, where  $\l=\xi^\infty$ and $\k=\Theta^{\c_{\l}}$:
\begin{enumerate}
\item the order type of the set $C=\{\tau\in [\k, \l]: \c_\l\models``\tau$ is a cardinal$"\}$ is at least $n+1$,
\item $\l<\Theta$ and $\l\in C$,
\item $\kappa$ is a regular member of the Solovay sequence,
\item letting $\langle \k_i: i\leq n \rangle$ be the first $n+1$ members of $C$ enumerated in increasing order,\footnote{Thus, $\k_0=\k$.} for every $i\leq n$, $\c^+_{\lambda}\models ``\k_i$ is a regular cardinal",
\item $\c_\lambda^{-} \cap \cP(\bR)
=\c_{\lambda}^+\cap \cP(\bR)=\Delta_\kappa$,
\item for all $i < n$, $\cP(\k_i^{\omega}) \cap \c_{\k_{i+1}}^- = \cP(\k_i^{\omega}) \cap \c^{+}_{\lambda}$,
\item for all $i+1\leq n$, $\cf(\k_{i+1}) \geq \kappa$. 
\end{enumerate}
\rthm{cating the powerset} and \rcor{zf in c} imply (1).
The first part of (2) is a consequence of the fact that in \rthm{main theorem}, $\d<\eta$; the second part follows from \rthm{lambda is regular}.
(3) follows from \rthm{first prop1}, and (4) follows from \rthm{first prop1} and \rthm{strongs are successors}.
(5) and (6) follow from \rthm{first prop} and \rthm{easier first prop}.
\rthm{cof is kappa} implies (7). \hfill$\qed$

\bibliographystyle{plain}
\bibliography{References}

\bls

\noindent\address{Department of Philosophy, University of Pittsburgh,
Pittsburgh, PA 15260}

\noindent\email{dwb44@pitt.edu}

\bls

\noindent\address{Department of Mathematics, Miami University, Oxford, Ohio 45056, USA}

\noindent\email{larsonpb@miamioh.edu}

\bls

\noindent\address{IMPAN, Antoniego Abrahama 18, 81-825 Sopot, Poland.}

\noindent\email{gsargsyan@impan.pl}

\end{document}